\documentclass[letterpaper,11pt,reqno]{amsart}

\usepackage[
  margin=1in,
  marginpar=0.8in,
  marginparsep=0.1in
]{geometry}

\usepackage{hyperref}
\hypersetup{
     colorlinks   = true,
     citecolor    = black,
     linkcolor    = blue,
     urlcolor     = black
}
\usepackage[alphabetic]{amsrefs}
\usepackage{graphicx,amsmath,amssymb,amsthm,paralist,color,tikz-cd}
\usepackage{mathrsfs}
\usepackage{mathtools}
\usepackage{cite}
\usepackage{setspace}
\usepackage{bbm}
\usepackage[final]{showlabels}
\usepackage[T1]{fontenc}
\fontencoding{T1}
\usepackage[utf8]{inputenc}

\usepackage[color=blue!30]{todonotes}
\usepackage{amscd}
\usepackage[all]{xy}
\usepackage{bbm}

\newcommand{\df}{\stackrel{\mathrm{def}}{=}}
\newcommand{\X}{\mathscr{X}}
\newcommand{\sm}{\smallsetminus}
\newcommand{\vre}{\varepsilon}

\newcommand{\gUdt}{\mrm{U}_{\mrm{dt}}}
\newcommand{\gUue}{\mrm{U}_{\mrm{ue}}}

\newcommand{\wstfrak}{\mathfrak{w}_{\mathrm{st}}}

\makeatletter
\let\orgdescriptionlabel\descriptionlabel
\renewcommand*{\descriptionlabel}[1]{
  \let\orglabel\label
  \let\label\@gobble
  \phantomsection
  \edef\@currentlabel{#1\unskip}
  \let\label\orglabel
  \orgdescriptionlabel{#1}
}
\makeatother

\usepackage{ifthen}

\newcommand{\Sue}{S_\mrm{ue}}
\newcommand{\Str}{S_\mrm{tr}}
\newcommand{\Sdt}{S_\mrm{dt}}

\newcommand{\Hne}{H^{\mathrm{fne}}}
\newcommand{\Vne}{\mathfrak{h}_{\mathrm{fne}}}
\newcommand{\Wst}{W^{\mrm{st}}}
\newcommand{\hibarF}{\bar{h}_i^F}
\newcommand{\XS}{\X^S_{d+1}}
\newcommand{\barmu}{\bar{\mu}}
\newcommand{\hibar}{\bar{h}_i}

\newcommand{\mf}{\mathfrak}
\newcommand{\mrm}{\mathrm}

\newcommand{\gG}{\mathrm{G}}

\newcommand{\gP}{\mathrm{P}}
\newcommand{\gU}{\mathrm{U}}

\newcommand{\Sf}{S_{\mrm{f}}}

\newcommand{\GL}{\mathrm{GL}}
\newcommand{\PGL}{\mathrm{PGL}}

\newcommand{\SL}{\mathrm{SL}}

\newcommand{\Ad}{\mathrm{Ad}}

\newcommand{\Lie}{\mathrm{Lie}}

\newcommand{\lieg}{\mathfrak{g}}

\newcommand{\lieu}{\mathfrak{u}}

\newcommand{\norm}[1]{\left\lVert {#1} \right\rVert}

\newcommand{\dist}{\mrm{dist}}

\newcommand{\der}[1][]{
  \ifthenelse{ \equal{#1}{} }
  {\ensuremath{\mathrm{d}}}
  {\ensuremath{\mathrm{d}_{#1}\!}}
}

\newcommand{\rquotient}[2]{\mathchoice
  {
    \raisebox{-.6ex}{\small \newline ${#1}$}\!\big\backslash\!\raisebox{.6ex}{\small ${#2}$}
  }
  {
    {#1}\backslash{#2}
  }
  {
    \raisebox{-.4ex}{\tiny \newline ${#1}$}\!\backslash\!\raisebox{.4ex}{\tiny ${#2}$}
  }
  {
    \raisebox{-.3ex}{\tiny \newline ${#1}$}\!\backslash\!\raisebox{.3ex}{\tiny ${#2}$}
  }
}

\newcommand{\id}{\mathrm{Id}}
\newcommand{\supp}{\operatorname{supp}}

\newcommand{\N}{\mathbb{N}}
\newcommand{\bN}{\mathbb{N}}
\newcommand{\NN}{\mathbb{N}}
\renewcommand{\P}{\mathbb{P}}
\newcommand{\Q}{\mathbb{Q}}
\newcommand{\bQ}{\mathbb{Q}}

\newcommand{\R}{\mathbb{R}}
\newcommand{\bR}{\mathbb{R}}

\newcommand{\Z}{\mathbb{Z}}
\newcommand{\bZ}{\mathbb{Z}}

\newcommand{\Kcal}{\mathcal{K}}

\newcommand{\Qcal}{\mathcal{Q}}

\newcommand{\Ucal}{\mathcal{U}}

\newcommand{\Zcal}{\mathcal{Z}}

\newcommand{\gfrak}{\mathfrak{g}}

\newcommand{\ufrak}{\mathfrak{u}}

\newcommand{\wfrak}{\mathfrak{w}}

\newcommand{\Xbf}{\mathbf{X}}

\newcommand{\ybf}{\mathbf{y}}

\newcommand{\pbf}{\mathbf{p}}

\newcommand{\xbf}{\mathbf{x}}

\newtheorem{theorem}{Theorem}[section]
\newtheorem{lem}[theorem]{Lemma}
\newtheorem{thm}[theorem]{Theorem}
\newtheorem*{thm-nonumber}{Theorem}

\newtheorem{prop}[theorem]{Proposition}

\newtheorem{corollary}[theorem]{Corollary}

\theoremstyle{definition}

\newtheorem*{definition-nono}{Definition}

\newtheorem{example}[theorem]{Example}

\newtheorem{remark}[theorem]{Remark}

\newtheorem*{acknowledgement}{Acknowledgements}

\newtheoremstyle{step}{}{}{}{}{}{:}{ }{}
\theoremstyle{step}

\renewcommand{\a}{\alpha}
\renewcommand{\b}{\beta}
\newcommand{\g}{\gamma}

\renewcommand{\d}{\delta}
\newcommand{\e}{\varepsilon}
\renewcommand{\l}{\lambda}

\newcommand{\s}{\sigma}

\renewcommand{\th}{\theta}

\renewcommand{\r}{\rightarrow}

\def\multiset#1#2{\ensuremath{\left(\kern-.3em\left(\genfrac{}{}{0pt}{}{#1}{#2}\right)\kern-.3em\right)}}

\newcommand{\seti}[1]{\left\{#1\right\}}

\numberwithin{equation}{section}

\title{Measure rigidity and equidistribution for fractal carpets}
\author{Osama Khalil}
 \address{Department of Mathematics, Statistics, and Computer Science, University of Illinois Chicago}
 \email{okhalil@uic.edu}
\author{Manuel Luethi}
\address{EPFL/SB/IMB/TAN, Station 8, EPFL CH-1015 Lausanne}
\email{manuel.luthi@epfl.ch}
\author{Barak Weiss}
\address{Department of Mathematics, Tel Aviv University}
\email{barakw@tauex.tau.ac.il}

\date{}
\AtBeginDocument{
   \def\MR#1{}
}

\begin{document}

\begin{abstract}
    Let $\theta$  be a Bernoulli measure which is stationary for a random walk generated by finitely many contracting rational affine dilations of $\R^d$, and let $\mathcal{K} = \mathrm{supp}(\theta)$ be the corresponding attractor. An example in dimension $d=1$ is the Hausdorff measure on Cantor's middle thirds set, and examples in higher dimensions include missing digits sets, Sierpi\'nski carpets and Menger sponges.
    Let $ \nu$ denote the image of $\theta$ under the map $\mathcal{K} \r \SL_{d+1}(\R)/\SL_{d+1}(\Z)$ which sends $\mathbf{x}$ to the lattice $\Lambda_{\mathbf{x}} = \mathrm{span}_{\Z} (\mathbf{e}_1, \ldots, \mathbf{e}_d, \mathbf{e}_{d+1}+ (\mathbf{x},0)).$
    We prove equidistribution of the pushforward measures $a_{n*}\nu$ along any diverging sequence of diagonal matrices $ (a_n)\subset \SL_{d+1}(\R)$ that expand the first $d$ coordinates
    under a natural non-escape of mass condition.
    The latter condition is known to hold whenever $\th$ is absolutely friendly.
    We also show that weighted badly approximable vectors and Dirichlet-improvable vectors (for arbitrary norm) form a subset of $\mathcal{K}$ of $\theta$-measure zero.
 The key ingredient is a measure classification theorem for the stationary measures of an associated random walk on an $S$-arithmetic space, introduced by two of the authors in \cite{KhalilLuethi}. A new feature of this setting is that this random walk admits stationary measures which are not invariant.
\end{abstract}

\maketitle


   \section{Introduction}
This paper proves equidistribution results for pushforwards of certain fractal measures on certain homogeneous spaces, by analyzing certain random walks adapted to the fractal measures. These dynamical results, are then applied to problems about the Diophantine properties of typical points on certain fractals. We refer the reader to \cite{KLW, BQ1, SimmonsWeiss, ProhaskaSertShi, KhalilLuethi, BHZ, DattaJana} for related work.

Let $k$ and $d$ be positive integers with $k \geq 2$  and let $\Phi = \{f_1, \ldots, f_k\}$ be a  collection of affine maps $f_i: \R^d \to \R^d$ of the form
\begin{equation}\label{eq: def fi}
f_i (\mathbf{x}) = \varrho \xbf   + \mathbf{y}_i, \ \ \text{ where }
\ybf_i
\in \Q^d, \ \varrho \in \Q, \ 0< |\varrho|<1.
\end{equation}
 We will refer to such a collection $\Phi$ as a {\em carpet IFS}. This terminology is motivated by the example of the Sierpi\'nski carpet.
 The {\em attractor} of $\Phi$ is the unique nonempty compact subset of $\R^d$ satisfying $\mathcal{K} = \bigcup f_i(\mathcal{K})$.
 We say $\Phi$ is {\em irreducible} if there is no finite collection of proper affine subspaces of $\R^d$ which is left invariant by each of the maps $f_i$.
For a probability vector $\mathbf{p} = (p_1, \ldots, p_k),$ let $\theta= \theta(\Phi, \mathbf{p})$ be the associated Bernoulli measure supported on its attractor $\mathcal{K}.$
These terms are explained in \textsection \ref{subsec: prelims attractors}, and
the sets $\mathcal{K}$ and  measures $\theta$ which arise in this way form a fairly large class of self-similar fractal sets and measures; for the purpose of this introduction it is enough to note that commonly studied self-similar sets, like Cantor's middle thirds set, missing digit  sets in $d \geq 1$ dimensions, or the Sierpi\'nski carpet, can arise as $\mathcal{K},$ and their Hausdorff measure  can arise as the Bernoulli measure $\theta$.

Let $\X_{d+1}$ denote the space of lattices of covolume one. This space is naturally identified with the quotient of Lie groups
$\SL_{d+1}(\R)/\SL_{d+1}(\Z)$,
via the map $g \SL_{d+1}(\Z) \mapsto g\Z^{d+1}$, and this identification equips $\X_{d+1}$ with the topology of a non-compact manifold and with the measure
$m_{\X_{d+1}},$ which is the unique $\SL_{d+1}(\R)$-invariant measure. Let
\begin{equation}\label{eq: def u+}
   u: \R^d \to \SL_{d+1}(\R), \ \  u(\xbf) \df  \begin{pmatrix}
      \mathrm{Id} & \xbf \\
      \mathbf{0} & 1
    \end{pmatrix}, \ \ \ U \df \left\{u(\xbf ) : \xbf \in \R^d \right\}.
\end{equation}
For $g \in \SL_{d+1}(\R)$, let $$\Lambda_g \df g \Z^{d+1},  \ \ \text{  let }
\Lambda_{\xbf} \df \Lambda_{u(\xbf)}, $$
and denote by $\nu_g $
the pushforward of $\theta$ under the map $\mathbf{x} \mapsto \Lambda_{u(\mathbf{x})g}$.

Given a sequence $(a_n)_{n\in \N}$ of elements of $\mrm{SL}_{d+1}(\R)$, we say that $\th$ is \textit{uniformly non-divergent along $(a_n)$} if for every $\e>0$, there is a compact set $K\subset \X_{d+1}$ so that for all $g$, for all large enough $n$ we have
\begin{align*}
    a_{n \ast}\nu_g(K)\geq 1-\e.
\end{align*}

With this notation we are ready to state our main equidistribution result.

\begin{thm}\label{thm: equidistribution main}
    Let $\Phi$ be an irreducible carpet-IFS,
   let $\theta = \theta(\Phi, \mathbf{p})$ be a Bernoulli measure, for some probability vector $\mathbf{p},$ and
   let
    $ \nu_0 $ be the pushforward of $\theta$ under $\xbf \mapsto \Lambda_{\xbf}$.
    Let $(a_n)_n$ be any sequence of diagonal matrices tending to $\infty$ in $\mrm{SL}_{d+1}(\R)$.
    If $\th$ is uniformly non-divergent along $(a_n)$, then
    \begin{equation}\label{eq: equidistribution main}
    \lim_{n \to \infty} a_{n*}  \nu_0 = m_{\X_{d+1}},
    \end{equation}
         with respect to the weak-$\ast$ topology on the space of probability measures on $\X_{d+1}.$
\end{thm}

As we will see in \textsection
\ref{subsec: friendliness}, conditions on $\theta$ and $(a_n)$ which guarantee uniform non-divergence are well-understood. Clearly, non-divergence of the sequence $a_{n*}\nu_0$ is a necessary condition for \eqref{eq: equidistribution main}, and all known methods for establishing this non-divergence estimate actually yield the stronger uniform version we use here. In particular, uniform non-divergence is known to hold when $(a_n)$ is \textit{drifting away from walls} and $\theta$ is \textit{absolutely friendly}; cf.~\textsection \ref{sec: preliminaries} for definitions.
Here we mention that absolute friendliness is satisfied for large classes of examples including the Hausdorff measure on $\Kcal$ under the open set condition, and for all Bernoulli measures under the strong separation condition.

Results about limits of measures as in \eqref{eq: equidistribution main} have a long history in homogeneous dynamics, and we briefly mention a few that are relevant to our discussion.
In the case
$a_n =\mathrm{diag}\left(e^{n/d}, \ldots, e^{n/d}, e^{-n}\right),$ the group $U$ is the expanding horospherical group for the action of $(a_n)$, and in this case, if
 $\theta$ is Lebesgue measure on $\R^d$ then
 \eqref{eq: equidistribution main} is an easy consequence  of the Howe-Moore theorem on mixing on $\mathscr{X}_{d+1}$.
 Also in the case of Lebesgue measure, for more general sequences $(a_n)$ which expand $U$ by conjugations,
 \eqref{eq: equidistribution main} was established in \cite{KWDirichletimprovable}
 and \cite{KleinbockMargulis}.

The first equidistribution results for a measure which is singular with respect to Lebesgue measure were given by Shah in \cite{ShahInventiones}, in the case that $\theta$ is the length measure on an analytic nondegenerate curve, and these results were  later extended by various authors (see \cite{ShahYang} and references therein).
The first equidistribution result for fractal measures came in the paper \cite{KhalilLuethi}, where an additional hypothesis was imposed, namely an inequality involving the contractions in the IFS $\Phi$ and the coefficients of the probability vector $\pbf.$ Under these conditions it was shown that \eqref{eq: equidistribution main} holds in an effective form.  Further effective results in case $d=1$ were obtained by Datta and Jana in  \cite{DattaJana}, under an assumption involving Fourier decay, and in  great generality in a breakthrough paper of  B\'enard, He and Zhang in \cite{BHZ}.
The
techniques of this paper are different from those used in \cite{KhalilLuethi,DattaJana,BHZ}.

\subsection{Applications to Diophantine approximations}
As with all of the previously mentioned dynamical results, Theorem \ref{thm: equidistribution main} is motivated by questions in Diophantine approximations. Our Diophantine applications, which we now state, concern so-called weighted approximation. Given a {\em weight vector}
$\mathbf{r} = (r_1, \ldots, r_d)$, with
$$
\sum_{i=1}^d r_i =1, \  \ \ \ \ \forall i, \ r_i>0,
$$
and $\xbf = (x_1, \ldots, x_d)$, we say that $\xbf$ is {\em $\mathbf{r}$-badly approximable} if there is $c>0$ such that for all  $Q \in \N$ and all $(P_1, \ldots, P_d) \in \Z^d$, we have
\begin{equation}\label{eq: def BA}
Q \, \max_i \left( |Qx_i - P_i|^{1/r_i} \right) \geq c.
\end{equation}
Let
\begin{equation}\label{eq: def weighted diagonal group}
a_t^{(\mathbf{r})} \df \mathrm{diag}\left(e^{r_1t}, \ldots, e^{r_dt}, e^{-t} \right) \subset \SL_{d+1}(\R).
\end{equation}
Given a norm $\| \cdot \|$ on $\R^{d+1},$ we set
\begin{equation}\label{eq: def epsilon norm}
\vre_{\| \cdot \|} \df \sup \{\vre > 0: \text{ there is } \Lambda \in \X_{d+1} \text{ such that } B_{\| \cdot \|}(0, \vre) \cap \Lambda =\{0\} \},
\end{equation}
and
say that $\xbf \in \R^d$ is {\em $(\mathbf{r}, \| \cdot \|)$-Dirichlet improvable} if there is $\vre < \vre_{\| \cdot\|} $ such that  for all sufficiently large $t$, the lattice $a_t^{(\mathbf{r})} \Lambda_{\xbf}$ contains  vectors $\mathbf{y}$ with
$0< \|y \| \leq \vre.$
The  case in which $\| \cdot \| = \| \cdot \|_\infty$ is the supremum norm was studied by Davenport and Schmidt \cite{DavenportSchmidt} and the case of general norms was studied by Kleinbock and Rao \cite{KleinbockRao}, see \textsection \ref{subsec: Dani correspondence} for more details.
With these notations we have:
\begin{thm}\label{thm: diophantine main}
    Let $\Phi$ and $\theta$ be as in Theorem \ref{thm: equidistribution main}.  Then for any weight vector $\mathbf{r},$ and for any norm $\| \cdot \|$ on $\mathbb{R}^{d+1}$,  the set of $(\mathbf{r}, \| \cdot \|)$-Dirichlet improvable vectors has $\theta$-measure zero.  In particular, the set of $\mathbf{r}$-badly approximable points has $\theta$-measure zero.
\end{thm}

The case $\mathbf{r} = (1/d, \cdots, 1/d)$ is called the case of {\em equal weights}.  In case $d=1$ the second assertion was proved by Einsiedler, Fishman and Shapira \cite{EFS}, and the case of equal weights in all dimensions was settled in \cite{SimmonsWeiss}.
Note that in case $d=1$ the only weights are the equal weights, thus our results are only new  when $d>1$.
The first results about general  weight vectors $\mathbf{r}$ were obtained by Prohaska, Sert and Shi \cite{ProhaskaSertShi}; however, they could only treat measures  defined by an IFS of affine maps, which depends on $\mathbf{r}$. Finally note that an  effective version of Theorem \ref{thm: equidistribution main} would have additional important Diophantine applications. Namely, in case $d=1$ it was used in \cite{BHZ} to  settle a question attributed to Mahler; the analogous conjecture in higher dimensions is still open.

\subsection{Random walks on $S$-arithmetic spaces}\label{subsec: RW S arithmetic}
  The results of this paper rely on a description of stationary measures for certain random walks on an $S$-arithmetic space, which depends on the IFS $\Phi. $
Let $d \in \N$ and denote by $\mathbf{G}$ the automorphism group of the $\bQ$-algebra $\mrm{Mat}_{d+1}$, i.e., $\mathbf{G} = \PGL_{d+1}$; cf.~\cite[\textsection3.1]{KhalilLuethi}. For a finite collection of places $S = \{\infty, p_1, \ldots, p_\ell\}$, where $p_1, \ldots, p_\ell$ are primes, we define an $S$-arithmetic homogeneous space
$$ \X_{d+1}^S \df G^S / \Lambda^S,$$
where
$$
G^S \df \mathbf{G}(\R) \times \prod_{j=1}^\ell \mathbf{G}(\Q_{p_j}),$$
and $\Lambda^S
$ is the diagonal embedding in $G^S$ of $$\mathbf{G}\left(\Z\left[\frac{1}{p_1}, \cdots, \frac{1}{p_\ell} \right ] \right).$$

The homogeneous space  $\X_{d+1}^S$ is equipped with a unique $G^S$-invariant probability measure $m_{\X_{d+1}^S}.$ Let $G \df \mathbf{G}(\R)$. Notice that the space of lattices $\X_{d+1}$ can also be identified with $G/ \mathbf{G}(\Z)$, and for $K_{\mathrm{f}} = \prod_{j=1}^\ell \mathbf{G}(\Z_{p_j})$, we have that  $K_\mathrm{f} \backslash \X^S_{d+1} \cong \X_{d+1}$; cf.~\cite[\textsection 3.1 \& Appendix B]{KhalilLuethi}.  There is a transitive action of $G^S$ on $\X_{d+1}^S$ by left translations, and since we have an inclusion $G \to G^S, $ we get an action of $G$ on $\X_{d+1}^S$, for which the projection $\X_{d+1}^S \to \X_{d+1}
$ is equivariant.  For more information about $S$-arithmetic groups and $S$-arithmetic homogeneous spaces, we refer the reader to \cite{PlatonovRapinchuk, Ratnerpadic, Tomanov}.

Following \cite{KhalilLuethi}, we now define
a random walk on $\X_{d+1}^S,$ depending on the carpet IFS $\Phi$ and a probability vector $\mathbf{p}$.
We will define elements $h_1, \ldots, h_k \in G^S$ below.
Then,
 fixing a probability vector $\mathbf{p},$
the {\em $\Phi$-adapted random walk}  is obtained by moving from $x \in \X^S_{d+1}$
 to $h_i x$, independently, with probability $p_i.$

Let
$\Sf \df \{ p_1, \ldots, p_\ell\}$ be all the primes which appear in the denominators of all the coefficients of the  maps of the IFS $\Phi$, as well as the numerator of the contraction ratio $\varrho$, and let $S \df \Sf \cup \{\infty\}$.
That is, $\Sf$ consists of the primes appearing in the decomposition of $rq$, where
$$\varrho = \frac{r}{q},  \ \ \ \gcd(r,q)=1,$$
and the denominators of all the coefficients of the translation vectors $\ybf_i$. This choice means that if we write
$$\Z\left[\frac{1}{\Sf}\right] \df \Z \left[\frac{1}{p_1}, \ldots, \frac{1}{p_\ell} \right],$$
then
$\Sf$ is the smallest set of primes such that $\varrho$ is invertible in $\Z\left[\frac{1}{\Sf}\right]$ and each $\mathbf{y}_i$ belongs to $\left(\Z\left[\frac{1}{\Sf}\right] \right)^d.$

Here and in what follows, we will define elements of $G^S$ by specifying $|S|$-tuples of matrices in $\GL_{d+1}(\Q_\s),$ where $\s$ ranges over $S$ and where $\Q_\infty$ is another notation for the field $\R$; the reader should keep in mind that each of these matrices should be thought of as a coset representatives modulo the center of $\GL_{d+1}(\Q_\sigma)$.
For
$$i=1, \ldots, k \ \ \text{ and } f_i \in \Phi, \ \ \ \ f_i(\xbf) =  \varrho \xbf + \ybf_i,$$
we define  $h_i$   in $G^S$ by
$h_i = \left( h_i^{(\sigma)} \right)_{\sigma\in S}$,
where
 \begin{equation}\label{eq: def h}
h_i^{(\sigma)} \df   \left \{ \begin{matrix}
  \begin{pmatrix}
    \varrho \, \mathrm{Id}_d & -\mathbf{y}_i \\
    \mathbf{0} & 1
  \end{pmatrix} & \, \ \ \ \ \ \ \  \ \text{if } \sigma \in S_{\mathrm{f}}  \\
 \begin{pmatrix}
     \varrho \, \mathrm{Id}_d & \mathbf{0} \\
    \mathbf{0} & 1
  \end{pmatrix} & \ \ \  \ \ \ \ \ \ \  \text{if } \sigma =\infty.
  \end{matrix} \right.
  \end{equation}

Let
 \begin{equation}\label{eq: def Sue} S_{\mathrm{ue}} \df \{ p \in S_{\mathrm{f}} : p|q\} \ \  \ \text { and } \ \  S_{\mathrm{ue}}^c \df S \sm S_{\mathrm{ue}}.
 \end{equation}

The subscript `ue' stands for `uniform expansion'.
 Let $\mu$
 be the measure on $G^S$ given by
 \begin{equation}\label{eq: def mu}
     \mu = \sum_{i=1}^k p_i \delta_{h_i}.
 \end{equation}
 By a {\em  measure on $\X_{d+1}^S$} we mean a  finite regular Borel measure. Recall that a measure $\nu$ on $\X_{d+1}^S$ is called {\em $\mu$-stationary} if $\mu * \nu = \nu,$ where $\mu * \nu$ is the measure on $\X_{d+1}^S$ defined by
 $$
 \forall f \in C_c\left(\X_{d+1}^S \right), \ \ \ \ \int_{\X_{d+1}^S} f \, \der(\mu * \nu) = \int_G  \int_{\X_{d+1}^S} f \, \der g_{*}\nu \, \der\mu(g).
 $$
 The collection of stationary probability measures is a closed convex set in the space of probability measures on $\X_{d+1}^S$, and $\nu$ is called {\em $\mu$-ergodic} if it is an extreme point of this convex set.

The following 
 property of stationary measures is our main result. Let
 $b \df \left(b^{(\sigma)} \right)_{\sigma \in S},$
 where
 \begin{equation}\label{eq: def a2}
 b^{(\sigma)} = \begin{cases}
     \begin{pmatrix}
            \varrho^{-1} \, \mathrm{Id} & \mathbf{0} \\
            \mathbf{0} & 1
        \end{pmatrix}&  \sigma \in \Sue,
        \\
        \mathrm{Id} &   \sigma \in \Sue^c.
 \end{cases}
 \end{equation}
For a sequence of measures $(\nu_k)_{k \in \N}$, we will write $\nu_k\to 0$ if  the sequence $(\nu_k)$ converges in the weak-* topology to the zero measure. In other words, for every compact set $K$, $\nu_k(K) \to 0.$ We will refer to this as {\em complete escape of mass}.

\begin{thm}
\label{thm: dynamical}
    Let $\Phi$ be an irreducible carpet IFS,
let $ h_i$ be as in \eqref{eq: def h},  let $\mathbf{p}$ be a probability vector, and let $\mu$ be as in \eqref{eq: def mu}.
Suppose $\nu$ is an ergodic $\mu$-stationary measure on $\X_{d+1}^S$.
Then one of the following holds:
\begin{enumerate}
    \item \label{item: mass escape}  $b^{k}_{\ast}\nu\to 0$ as $k\to \infty$.
    \item $\nu = m_{\X_{d+1}^S}$.
    \end{enumerate}
\end{thm}

\begin{remark}
    In Theorem \ref{thm: stationary main} we give a finer description of stationary measures. In particular, this theorem implies that if $\nu$ is an ergodic stationary measure that differs from $m_{\XS}$, then for $\nu$-almost every $x_0 = g_0 \Gamma^S \in \XS$, the trajectory $(b^k x_0)_{k\geq 0}$ diverges because there is a rational vector in an exterior power of $\mathfrak{g}$ which is contracted to zero under the adjoint action of $b^kg_0$, as $k \to \infty$; that is, in the terminology of \cite{Weiss_divergence_GAFA}, the trajectory diverges for `obvious reasons'. 
\end{remark}

\subsection{Organization of the paper}
Theorem \ref{thm: dynamical}, which provides a description of stationary measures for a certain random walk,  is the main result of this paper. In  \textsection \ref{sec: preliminaries}, we introduce  some standard facts and introduce our notation. In \textsection \ref{sec: examples} we illustrate some stationary measures which arise for the random walk we consider. As these examples show,  the random walk we consider is not {\em stiff}, i.e., there are stationary measures which are not invariant under individual elements in $\mathrm{supp}(\mu)$. This sets the random walk considered here apart from many of the setups studied in  prior works, and is a major complication in our setup. In \textsection \ref{sec: deducing other results} we deduce the other main results of the paper from Theorem \ref{thm: dynamical}. Note that our argument for equidistribution is quite different from the one used in prior work; the crucial point for these deductions is that the maps considered in \eqref{eq: equidistribution main} and in Theorem 
\ref{thm: dynamical}  commute with all elements in $\mathrm{supp}(\mu),$ and thus send stationary measures to stationary measures.

In \textsection \ref{sec: mu vs mu bar} we begin the proof of Theorem 
\ref{thm: dynamical}. We introduce an auxiliary random walk $\bar \mu$. The trajectories for the $\bar \mu$-random walk stay within a bounded distance from those of the $ \mu$-random walk, and treating $\bar \mu$ makes it possible to avoid some technical complications. We state Theorem \ref{thm: stationary main}, which is the analogue of Theorem 
\ref{thm: dynamical} for this random walk, and show, by considering a random walk on a common cover, that Theorem \ref{thm: stationary main} implies Theorem 
\ref{thm: dynamical}.
The proof of Theorem \ref{thm: stationary main} occupies the rest of the paper. The steps and main ideas for its proof are described in \textsection \ref{sec:outline of proof} below.

\begin{acknowledgement}
   The authors are grateful to Aaron Brown, Nishant Chandgotia, and \c{C}agr{\i} Sert  for stimulating discussions, and are grateful to the anonymous referees for pointing out some inaccuracies and helping make the paper more readable.
   O.K. is partially supported by NSF grants DMS-2247713 and DMS-2337911,
   M.L.~is partially supported by SNSF grants 200021-197045 and 200021L-231880, and
   B.W. is partially supported   by grants  ISF-NSFC 3739/21 and ISF 2021/24. The authors are grateful to the CMSA at Harvard University for its hospitality in May 2023, when some of the work on this paper was conducted. M.L.~thanks Dmitry Kleinbock for the hospitality.
   \end{acknowledgement}

\section{Outline of the Proof}
\label{sec:outline of proof}
For convenience of the reader, we provide an informal outline of the proof of Theorem~\ref{thm: stationary main}.
The proof follows the exponential drift strategy of Benoist and Quint \cite{BQ1}, but requires substantial adaptations to our non-stiff setup.
 In this section we recall the strategy of \cite{BQ1}, indicate the complications which arise in our  setting, and explain how we deal with them. For the purpose of this section, there will be no harm in ignoring the distinction between $\mu$ and $\bar \mu$, that is, assume that the random walk is the one described in \textsection \ref{subsec: RW S arithmetic}.

Let $B=(\supp \bar \mu)^{\N}$ be the space of infinite words in the random walk and $\b=\bar{\mu}^{\otimes \N}$ be the associated Bernoulli measure.   Let $\nu$ be a $\bar \mu$-stationary measure.  Recall that we have a disintegration  $\nu = \int_B \nu_b\, \der \b(b)$ in terms of Furstenberg limit measures $\nu_b=\lim_{n\to\infty} (b_1\cdots b_n)_\ast \nu$; cf.~\textsection\ref{sec:generalities stationary measures} for details.
Let $W_b$ denote the subgroup tangent to the directions that are contracted to $0$ under the adjoint action of $b_n^{-1}\cdots b_1^{-1}$ as $n\to\infty$.
In our setting, this group is deterministic (i.e., independent of $b$) and is a proper subgroup of the `stable group' $\Wst$ which we introduce in \eqref{eq: def Wst}.

We assume that Case~\eqref{item: case 2} of Theorem~\ref{thm: stationary main} does not hold, which by ergodicity, implies that $\nu$-almost every point has trivial stabilizer in $\Wst$.   The goal of the exponential drift argument, implemented in \textsection\ref{sec: exponential drift}, is to show that $\nu_b$ is invariant under a one-parameter (necessarily unipotent) subgroup of $W_b$ for almost every $b$. The result will then follow by an application of Ratner's theorem along with the special structure of our random walk that is used to rule out other homogeneous measures besides $m_{\XS}$. This is carried out in \textsection\ref{sec:endgame}.

Following \cite{BQ1}, given a generic $b$, with the aid of Lusin's theorem, we find two $\e$-close points $x,y$ in the support of $\nu_b$, so that the leafwise measures of $\nu_b$ along the $W_b$-orbits of $x,y$ are also close.  We then apply the random walk to this configuration with the goal of finding two new points $x_2,y_2$, in the support of $\nu_{b_2}$, for a suitable $b_2$,  which satisfy the following:
\begin{enumerate}
    \item\label{item:exp drift growth} The distance between $x_2,y_2$ is $\asymp 1$.
    \item\label{item:exp drift direction} The displacement between $x_2$ and $y_2$ essentially points in the direction of $W_{b_2}$.
    \item\label{item:exp drift cancel distortion} The leafwise measures of $\nu_{b_2}$ along the $W_{b_2}$-orbits of $x_2,y_2$ remain close.
\end{enumerate}
Taking $\e$ to $0$, and repeating this process, produces points with displacement belonging to  $W_{b'}$ for suitable $b'$, and whose leafwise measures agree, thus implying the desired invariance.

The random walk maneuvers that produce such $x_2,y_2$ consist of first deleting the length-$n$ prefix of $b$, $n=n(\e)$, and then adjoining a suitably chosen length-$m$ prefix, $m=m(n)$.
The first leg (resp.~second) of this itinerary is referred to as the backward (resp.~forward) random walk.
Item~\eqref{item:exp drift direction} follows from general properties of linear random walks, roughly that vectors tend to point towards the direction of a suitable top Lyapunov space under the random walk; cf.~\textsection \ref{sec:GrowthPropertiesRandomWalk}.

\subsection*{The role of the no-rotations hypothesis}
The reason the strategy involves both going backward and forward by the random walk is to achieve item~\eqref{item:exp drift cancel distortion}, where $n$ and $m$ are chosen so that the distortion of the leafwise measures in the two legs nearly cancel each other out.
Here, and in all prior works, the key property needed is the conformality of the action of the random walk on these leafwise measures. In our setting, this is the reason we assume the linear part of our IFS has no rotations. Indeed, otherwise, these rotations may generate a non-compact group over one of the primes in the definition of the induced random walk on $\XS$. Such non-compactness would lead to a non-conformal action on the  Lyapunov space corresponding to $W_b$.

\subsection*{The role of the group $\Wst$}
To achieve~\eqref{item:exp drift growth}, we must ensure that $x,y$ are not aligned along directions that may contract by going either backward or forward by the random walk.
In \cite{BQ1}, and almost all prior works, this is done by ensuring that $\nu_b$ gives $0$ mass to $W_b$-orbits, so that backward motion does not contract the displacement, while relying on growth properties of random walks to ensure that, for any given vector in the tangent space, most forward random walk trajectories will cause it to grow (with additional complications arising from neutral/central directions).

By contrast, our random walk admits a non-trivial deterministic \textit{contracting} subspace under \textit{every} forward random walk trajectory. The group $\Wst$ is thus defined to be the group generated by deterministic forward-contracting \textit{and} backward-contracting groups, in addition to the centralizer of the random walk.
A key step in carrying out the above strategy is thus to show that $\nu_b$ gives $0$ mass to $\Wst$-orbits. This is Theorem~\ref{prop:NonAlignmentLimitMeasures}.
Due to the mixed behavior of $\Wst$ (some of its directions expand in the future while others expand in the past), the proof of this result  represents the major departure in our proof compared to prior works.
A key ingredient is a projection argument introduced in Lemma~\ref{lem:AvoidanceStable} to separate these mixed behaviors allowing us to handle them individually.
This is the critical step where the hypothesis that points have trivial stabilizers in $\Wst$ is used.

To get that most forward trajectories of our random walk expand a given transverse direction to $\Wst$, we note that such directions point along $\mrm{G}_{\Sue}$ to which we apply 
results analogous to those of Simmons and the third author~\cite{SimmonsWeiss}; cf.~\textsection\ref{sec:GrowthPropertiesRandomWalk}.
Note that the results of~\cite{SimmonsWeiss} hold for a certain real random walk, which nonetheless has the same block structure as the restriction of the random walk considered in this article to the $\Sue$-adic places.

\subsection*{Non-atomicity in the presence of contracting spaces}

Non-alignment along $\Wst$-orbits is done by applying the forward and backward random walks to separately contract the respective pieces of the orbit, reducing the problem to showing almost sure non-atomicity of the measures $\nu_b$.
This is established in the strong form needed for the proof in Theorem~\ref{thm:NonAtomicityLimitMeasures}.
The presence of a deterministic forward contracting space poses significant difficulties in this step. For instance, the standard arguments involving a
 `Margulis inequality' (see \cite[Prop. 3.9 \& \textsection 6.2]{BQ1})
 do not seem applicable in this setting.
Note also that such non-atomicity fails to hold without the hypothesis on trivial stabilizers in $\Wst$, as shown by the examples of stationary, non-invariant measures given in \textsection\ref{sec: examples}.

Instead, we argue directly by a delicate local analysis. We first show that averaging on the set of words which grow a given displacement vector to a macroscopic size satisfies a certain pointwise ergodic theorem in the underlying Bernoulli shift space; cf.~Theorem~\ref{thm:PrefixErgodicTheorem}.
Our hypothesis that the IFS has a single contraction ratio is used in this step to ensure that this set of words has nearly full measure (indeed, otherwise these words will be given polynomially decaying measure in their length as can be checked by a direct computation).  We believe it is possible to push our arguments to bypass this difficulty, and hope to return to this problem in future work.

Equipped with the above ergodic theorem, if the measures $\nu_b$ were atomic, we would obtain a contradiction by finding words that simultaneously push generic points towards one another (by continuity and the aforementioned prefix ergodic theorem), and away from one another (by the expansion of the action transverse to orbits of the forward-contracted group).
The fact that the words have comparable norm in every irreducible component allows us to control the speed at which our points diverge from one another and capture our pair of generic points just as they are starting to separate from one another.

\section{Preliminaries}\label{sec: preliminaries}
In this section we collect some standard facts about our objects of study.
\subsection{Attractors of IFS's}\label{subsec: prelims attractors}

A mapping $f : \R^d \to \R^d$ is said to be a {\em contracting affine half-dilation} if it is of the form $f(\mathbf{x}) = \varrho \mathbf{x} +
\mathbf{y}$, where $\varrho \in \R$, $0< |\varrho|<1$ is the {\em contraction ratio} and $\mathbf{y
}$ is the {\em translation.}  The nomenclature `half-dilation' is due to the fact that we allow negative contraction ratios. Note that when $d>1$, the semigroup of contracting affine dilations is strictly contained in the well-studied semigroup of contracting similarity maps, in which one is allowed to compose the maps $f$ as above with orthogonal transformations.  We say that $f$ is {\em rational} if $\varrho \in \Q$ and $\mathbf{y} \in \Q^d$.
A collection $\Phi = \{f_1, \ldots, f_k\}$ of maps is called an {\em iterated function system (IFS).} With this terminology, the carpet-IFS's we consider in this paper are iterated function systems of rational half-dilation affine maps with constant contraction ratio.

Let $B = \{1, \ldots, k\}^{\N}$ and let $\mathrm{cod}: B \to \R^d$ be the map defined by
\begin{equation}\label{eq: def coding map}
\mathrm{cod}(b) = \lim_{n \to \infty} f_{i_1} \circ \cdots \circ f_{i_n} (\mathbf{x}_0),  \ \ \ \ \text{ where } b = (i_1, i_2, \ldots) \ \ \text{ and } \mathbf{x}_0 \in \R^d.
\end{equation}
It is well-known  that the limit in \eqref{eq: def coding map} exists for all $b$, is independent of the choice of $\mathbf{x}_0,$ and that the map $\mathrm{cod}$ is continuous. The image of $\mathrm{cod}$ is called the {\em attractor} of $\Phi$, and we denote it by
$\mathcal{K} = \mathcal{K}(\Phi)
.$
Basic results about the attractor $\mathcal{K}$ were obtained in classical work of Moran and Hutchinson \cite{Moran, Hutchinson}.  Among other things, they showed that $\mathcal{K}$ is the unique non-empty compact subset of $\R^d$ satisfying the stationarity property
\begin{equation}\label{eq: stationarity property K} \mathcal{K} = \bigcup_{i=1}^k
f_i(\mathcal{K}).
\end{equation}
When the elements in this union are disjoint we say that $\Phi$ satisfies {\em strong separation}.
If there is a non-empty open subset $U \subset \R^d$ such that $f_i(U) \subset U$ for all $i$ and $f_i(U) \cap f_j(U) = \varnothing $ for $i \neq j$, we say that $\Phi$ satisfies the {\em open set condition}.
It is known that strong separation  implies the open set condition.
Let $\mathbf{p}=(p_1, \ldots, p_k)$ be a {\em probability vector of full support}, that is a $k$-tuple of real numbers such that
$$
\sum_{i=1}^k p_i =1,
 \ \ \ \forall i, \ p_i>0.$$
 The {\em Bernoulli measure} $\theta = \theta(\Phi, \mathbf{p})$ on $\mathcal{K}$ is the image of the measure $(\sum_{i=1}^k p_i \delta_i)^\N$ under the map $\mathrm{cod}.$ It is the unique measure on $\R^d$ which satisfies the stationarity property
\begin{equation}\label{eq: stationarity theta}
\theta = \sum_{i=1}^k p_i f_{i*} \theta.
\end{equation}
Let $s = \dim (\mathcal{K})$ denote the Hausdorff dimension of $\mathcal{K}$.
It is well-known that if one assumes the open set condition, then up to scaling, the restriction of $s$-dimensional Hausdorff measure to $\mathcal{K}$ is a Bernoulli measure. If in addition, one assumes that the contraction ratios are equal to each other, this Bernoulli measure is given by the uniform probability vector $\mathbf{p} = (1/k, \ldots, 1/k).$

 We will be interested in the effect of conjugation of elements of $\Phi$ by an affine
 similarity mapping $f: \R^d \to \R^d$. If we let
 \begin{equation}\label{eq: if we let}
 \Phi = \{f_1, \ldots, f_k\} \ \ \text{ and } \Phi' = \{f'_1, \ldots, f'_k\}, \text{ where } f'_i = f \circ f_i \circ f^{-1},
 \end{equation}
then it can be easily checked  that
\begin{equation}\label{eq: Phi prime}
\mathcal{K}(\Phi') = f(\mathcal{K}(\Phi)) \ \ \text{ and } \ \theta(\Phi', \mathbf{p}) = f_* \theta(\Phi, \mathbf{p}).
\end{equation}

\begin{remark}\label{remark: referee}
It was shown in \cite[Prop. 3.1]{Broderick_Fishman_Simmons} that $\Phi$ is irreducible if and only if there is no proper affine subspace of $\R^d$ which is fixed by all of the maps in $\Phi$. We claim that in our setting of an IFS of affine half-dilations, this is also equivalent to the condition that the translation vectors $\ybf_i$ are not contained in a proper affine subspace. Indeed, a half-dilation map $f(\xbf) = \varrho \xbf+ \ybf$ preserves a linear subspace $V$ if and only if $\ybf \in V$. Thus, if the $\ybf_i$ are all contained in an affine subspace $V$, we can  apply a conjugation so that $V$ is linear (i.e., passes through the origin), and apply the previous observation to see that $\Phi$ is reducible. Conversely, if $V$ is proper and invariant under all the $\ybf_i$ then the $\ybf_i$ all belong to $V$. 
We are indebted to an anonymous referee of the paper for pointing this out.
\end{remark}

\subsubsection{Examples}
A {\em missing digit set} is a set of the form
$$
\mathcal{K} = \left\{\sum_{i=1}^\infty a_i b^{-1} : a_i \in D \right\},
$$
where $b \geq 3 $ and $D \subset \{0, \ldots, b-1\}, $ with $2 \leq | D| \leq b-1.$ The standard example is given by the Cantor middle thirds set, with $b=3$ and $D = \{0, 2\}$. These sets are attractors of the IFS in dimension $d=1$, given by  $f_i(x) = \frac{1}{b} x + \frac{a_i}{b},$ where $D = \{a_1, \ldots, a_k\}$. This IFS satisfies strong separation if $D$ does not contain consecutive indices, and satisfies the open set condition for any $D$. Thus the class of  missing digit sets is obviously contained in the class of carpet IFS's which we consider in this paper. Other well-known examples of irreducible carpet IFS's satisfying the open set condition are those whose attractors are the {\em Sierpi\'nski carpet} and {\em Menger sponge}, which are examples in dimensions $d=2$ and $d=3$ respectively.

An example of a well-known fractal not covered by our results is the Koch snowflake. This two-dimensional fractal is the attractor of an irreducible IFS on $\R^2$ satisfying the open set condition, where the maps in the IFS are similarities, but these similarities do not have rational coefficients and cannot be represented as dilations (rotations by multiples of $\pi/3$ are required). Another example not covered by our work is the translation of a Sierpi\'nski carpet by an irrational vector. It would be interesting to know
whether Bernoulli measures on these fractals satisfy the conclusion of Theorem \ref{thm: equidistribution main}.

\subsection{Friendliness of some fractal measures}\label{subsec: friendliness}
We now introduce some properties of a measure $\theta$ on $\R^d$, following \cite{KLW}. Let $B(\xbf, r)$ be the ball of radius $r$ centered at $\xbf$, with respect to the Euclidean metric on $\R^d$. For a constant $D \geq 1$, we say that $\theta$ is {\em $D$-Federer}
if for any $\xbf \in \mathrm{supp}(\theta)$ and any $r>0$ we have
\begin{equation}\label{eq: Federer}
\theta (B(\xbf, 3r))\leq D \,\theta(B(\xbf, r)).
\end{equation}
For any $A \subset \R^d$ and $\vre>0$, we write $A^{(\vre)} = \bigcup_{a \in A} B(a,\vre)$ for the $\vre$-neighborhood of $A.$  Given positive $C, \alpha$, we say that $\theta$ is {\em $(C, \alpha)$-absolutely decaying} if for any ball $B$ centered in $\mathrm{supp}(\theta),$ any affine hyperplane $\mathcal{L}$, any $B = B(\xbf, r)$ with $\xbf \in \mathrm{supp}(\theta), \, r \in (0,1)$ and any $\vre >0$, we have
\begin{equation}\label{eq: absolute decay}
\theta\left(B \cap \mathcal{L}^{(\vre)} \right)
\leq C\left( \frac{\vre}{r} \right)^{\alpha} \theta(B).
\end{equation}

A measure $\theta$ for which there are $D, C, \alpha  $ such that $\theta$ is $D$-Federer and $(C, \alpha)$-absolutely decaying is called {\em absolutely friendly.}

We will need the following:

\begin{prop}\label{prop: fractal measures are friendly}
    Let $\Phi$ be an irreducible carpet-IFS. Assume that $\theta$ is a measure on  the attractor $\mathcal{K}$ of $\Phi$, satisfying the conditions of Theorem \ref{thm: equidistribution main}, namely at least one of the following:
    \begin{enumerate}
        \item \label{item: case 1 friendly} $\Phi$ satisfies the open set condition and $\theta$ is the Hausdorff measure on $\mathcal{K}$;
        \item \label{item: case 2 friendly} $\Phi$ satisfies strong separation and  $\theta = \theta(\Phi, \mathbf{p})$ is a Bernoulli measure, for some probability vector $\mathbf{p}.$
    \end{enumerate}
Then $\theta$ is absolutely friendly. Moreover, there are $D, C, \alpha$ such that for any conjugate $\Phi'$ of $\Phi$ by an affine dilation map, the Hausdorff measure on the attractor $\mathcal{K}'$ of $\Phi'$  is $D$-Federer and $(C, \alpha)$-absolutely decaying.
\end{prop}
\begin{proof}
    The first assertion is proved in \cite[\textsection 8]{KLW} in case \eqref{item: case 1 friendly}, and in \cite[Thm. 1.7]{DFSU_dynamically_defined_measures} in case \eqref{item: case 2 friendly} (the latter proof is based on \cite{Urbanski}). For the second assertion we use \eqref{eq: Phi prime}, and note that an affine similarity map sends Euclidean balls to Euclidean balls, and hence does affect the validity of \eqref{eq: Federer} and \eqref{eq: absolute decay}.
\end{proof}

We will also need the following `self-similar Lebesgue density theorem':

\begin{prop}
    \label{prop: Lebesgue density}
    Let $\Phi$ and $\theta$ be as in Proposition \ref{prop: fractal measures are friendly}, let $\mathrm{cod}: \{1, \ldots, k\}^{\N} \to \mathcal{K}$ be the coding map, and let $A \subset \R^d$ be a Borel set. Then for $\theta$-a.e. $\xbf \in A$ we have that
    \begin{equation}
        \label{eq: from martingale}
    \lim_{n \to \infty} \frac{\theta\left(A \cap f_n(\mathcal{K})\right)}{\theta\left(f_n(\mathcal{K}) \right)}=1,
    \end{equation}
    where $\xbf = \mathrm{cod}(i_1, i_2, \ldots)$ and $f_n = f_{i_1} \circ \cdots \circ f_{i_n}.$
\end{prop}

In this statement we have used that $\xbf$ determines its coding sequence $(i_1, i_2, \dots)$, which is true for $\theta$-a.e.\,$\xbf$ in view of the separation assumptions of Proposition \ref{prop: fractal measures are friendly}.

Note that in the Lebesgue density theorem, valid for all finite Borel measures on $\R^d$ (see \cite[Cor. 2.14]{Mattila}), the appearance of $f_n(\mathcal{K})$ is replaced with a ball centered at $\xbf$, shrinking as $n \to \infty.$ The self-similar structure allows us to replace balls with shrinking copies of the attractor.

\begin{proof}
For each $n,$ let $\mathcal{B}_n$ be the finite $\sigma$-algebra of subsets of $\mathcal{K} $ generated by the sets $f(\mathcal{K})$, where $f$ ranges over all the compositions $f_{i_1} \circ \cdots \circ f_{i_n}$.  We claim that under the assumptions on $\theta$ we have that for any $i\neq j$, $\theta(f_i(\mathcal{K}) \cap f_j(\mathcal{K})) = 0.$
Indeed, under the strong separation condition $\theta(f_i(\mathcal{K}) \cap f_j(\mathcal{K})) = \varnothing$, and for the Hausdorff measure on $\mathcal{K}$ we use \eqref{eq: stationarity theta} and the scaling of Hausdorff measure under similarity maps.

This implies
that in the covering \eqref{eq: stationarity property K}, the sets are disjoint up to $\theta$-nullsets. Thus the atoms of  $\mathcal{B}_n$ are (up to $\theta$-nullsets)  the sets 
$f(\mathcal{K})$ themselves, and the quotient on the left-hand side of \eqref{eq: from martingale} is the conditional expectation of the indicator function $\mathbf{1}_{A}$ w.r.t. the $\sigma$-algebra $\mathcal{B}_n$, evaluated at $\xbf$. Since the diameter of the sets 
$f_n(\mathcal{K})$ goes to zero, the $\sigma$-algebra generated by $\bigcup_n \mathcal{B}_n$ is the Borel $\sigma$-algebra on $\mathcal{K}. $ The increasing Martingale theorem (see \cite[Chap. 5.2]{EW}) now implies that for $\theta$-a.e. $\xbf$, the left hand side of \eqref{eq: from martingale} converges to $\mathbf{1}_{A}(\xbf)$. This completes the proof.
\end{proof}

\subsection{The Dani-Kleinbock correspondence and homogeneous dynamics}\label{subsec: Dani correspondence}
The so-called {\em Dani correspondence} was developed by Dani in \cite{Dani86}. He showed that for equal weights approximation, a vector $\xbf \in \R^d$ is badly approximable if and only if the trajectory $\{a_t \Lambda_{\xbf} : t>0\}$ is bounded in $\X_{d+1},$ where $a_t = 
a_t^{(\mathbf{r}_0)}$ is the one-parameter group as in \eqref{eq: def weighted diagonal group}, corresponding to equal weights $\mathbf{r}_0 = (1/d, \cdots, 1/d)$. The correspondence was later developed and extended to the notions of Dirichlet improvability and weighted approximation in the papers \cite{Kleinbockweights, KWDirichletimprovable, KleinbockRao}. In particular the following hold:

\begin{prop}\label{prop: Dani Kleinbock correspondence}
Let $\xbf \in \R^d, \, \Lambda_{\xbf} $ and $\mathbf{r}$ be as above. Then the following hold:
  \begin{enumerate}
      \item \label{item: i} $\xbf $ is $\mathbf{r}$-badly approximable if and only if the trajectory $\left\{a^{(\mathbf{r})}_t \Lambda_{\xbf} : t>0 \right \}$ is bounded in $\X_{d+1}.$
      \item \label{item: ii} Denoting the sup-norm by $\|\cdot \|_\infty$, we have that $\mathbf{x}$ is $(\mathbf{r}, \| \cdot \|_\infty)$-Dirichlet improvable if and only if for some $\vre \in (0,1), $ for all sufficiently large $T$, there are  $Q \in \N$ and $(P_1, \ldots, P_d) \in \Z^d$  such that $Q \leq T$ and $\max_i \left( |Qx_i -P_i|^{1/r_i}\right ) \leq \frac{\vre}{T}.$
  \end{enumerate}
\end{prop}

We can view the expression appearing on the left-hand side of \eqref{eq: def BA} as the sup-norm of the vector $Q(|Qx_1-P_1|^{1/r_1}, \ldots, |Qx_d-P_d|^{1/r_d}).$
It follows from item \eqref{item: i} of Proposition \ref{prop: Dani Kleinbock correspondence}, and can also be easily inferred directly from the definition, that the property of being $\mathbf{r}$-badly approximable does not depend on the choice of a norm on $\R^d.$ On the other hand, Dirichlet improvability depends rather delicately on the chosen norm.
Note that following
\cite{KleinbockRao}, we have defined the notion of Dirichlet improvability for a general norm in terms of the dynamical behavior of the trajectory $\left\{a^{(\mathbf{r})}_t \Lambda_{\xbf} : t>0 \right \}$. Item \eqref{item: ii} of Proposition \ref{prop: Dani Kleinbock correspondence} shows that for the sup-norm, there is a simple interpretation of this property in terms of Diophantine inequalities, and indeed, this was the original definition introduced by Davenport and Schmidt \cite{DavenportSchmidt}. Similar Diophantine interpretations can be given for norms on $\R^{d+1}$ arising from a norm $\|\cdot \|'$ on $\R^d$ via the formula  $\| (y_1, \ldots, y_{d+1})\|'\df  \max(\|(y_1, \ldots, y_d)\|, \, |y_{d+1}|)$, but there is no such interpretation for general norms on $\R^{d+1}$. It was shown by Davenport and Schmidt (see \cite{DavenportSchmidt}) that $\mathbf{r}$-badly approximable are necessarily $(\mathbf{r}, \| \cdot \|_\infty)$-Dirichlet improvable.  For general norms this implication does not necessarily hold, see \cite{KleinbockRao}.

\subsubsection{Quantitative nondivergence for friendly measures}
Given a probability measure $\nu$ on $\X_{d+1}$ and a sequence of elements $(a_j) \subset \SL_{d+1}(\R)$, we say that $\nu$ has {\em no escape of mass under $(a_j)$} if any weak-* accumulation point of $(a_{j*}\nu)$ is a probability measure; equivalently, for any $\vre>0$ there is a compact subset $K \subset \X_{d+1}$ such that for all large enough $j$, $a_{j*} \nu(K) \geq 1-\vre.$

Let $A \subset G$ be the group of diagonal matrices of determinant one. Each $a \in A$ can be represented as
$$\exp(\Xbf) = \mathrm{diag}\left( e^{X_1}, \ldots, e^{X_{d+1}} \right)$$ where
$$\Xbf = (X_1, \ldots, X_{d+1}), \ \ \ X_i
\in \R, \ \ \ X_1 + \cdots + X_{d+1} =0.$$ Following \cite{KWDirichletimprovable} we say that a sequence $a_n = \exp\left(\Xbf^{(n)} \right)$ {\em drifts away from walls} if $$\left \lfloor \Xbf^{(n)} \right \rfloor \to \infty, \ \ \ \text{ where } \lfloor \Xbf \rfloor \df
\min \left\{X_i: i=1, \ldots, d \right \} .$$

The following result, proved in \cite{KLW}, is useful for exhibiting no escape of mass.

\begin{prop}\label{prop: KLW nondivergence}
Let $\theta$ be an absolutely friendly measure on $\R^d$, let $\nu$ be the pushforward of $\theta$ under the map $\xbf \mapsto \Lambda_{\xbf}$, and let $(a_j)$ be a sequence in $A$ that drifts away from walls. Then $\nu$ has no escape of mass under $(a_j).$
\end{prop}

We will need the following strengthening of Proposition \ref{prop: KLW nondivergence}, which establishes the uniform non-divergence property of Theorem \ref{thm: equidistribution main}:
\begin{prop}\label{prop: KLW nondivergence uniform}
For any positive $\varepsilon, C, \alpha,$ and any $ D\geq 1$ there is a compact $K\subset \X_{d+1}$ such that for any $(C, \alpha)$-decaying and $D$-Federer compactly supported measure $\theta$ on $\R^d$,  for any
$g \in \SL_{d+1}(\R)$, the measure $\nu = \nu_g$ defined below \eqref{eq: def u+} satisfies that
for any sequence $\{a_j \} \subset A$ which drifts away from walls, for all sufficiently large $j$ we have
$$a_{j*}  \nu (K) \geq 1-\varepsilon.$$

\end{prop}

The proof is essentially given in \cite{KWDirichletimprovable}, based on \cite{KLW}. We sketch the argument for the reader's convenience.

\begin{proof}[Sketch of proof.]
  We first claim that in proving this statement we may assume that $\mathrm{supp}(\theta) \subset \mathbf{B},$ where $\mathbf{B}$ is the unit ball around the origin in $\R^d.$
  Indeed, if the support of $\theta$ is not contained in $\mathbf{B}$, we can replace $\theta$
  by $f_*\theta$, for a linear contraction $f: \R^d \to \R^d$
 which maps $\mathrm{supp}(\theta)$ into $\mathbf{B}.$ As in the proof of Proposition \ref{prop: fractal measures are friendly}, such a map does not affect the constants $D, C, \alpha$,
 and can be realized by the conjugation action of $a_t$ on $U$ for some $t_0\in \R$. This amounts to replacing $x_0$ by $a_{t_0}x_0$, replacing $\nu$ by $a_{t_0*}\nu$, and replacing $(a_j)$ by $(a_j a_{t_0}),$ which is also a sequence drifting away from walls.

We now verify the assumptions of \cite[Thm. 4.3]{KLW}, with $U$ a big enough ball containing $\mathbf{B}$ and with $h(u) = a_j \tau(u)g_0$,  where  $\tau$ is as in \cite[Sec. 2.1]{KWDirichletimprovable}, $g_0$ satisfies $x_0 = \pi(g_0)$, and with $\rho =1$, and we verify the conditions for all $j$ large enough.  Assumption (1) is immediate from \cite[Lemma 4.1]{KLW} and \cite[Lemma 3.3]{KWDirichletimprovable} (where in \cite[Lemma 3.3]{KWDirichletimprovable} we take $\mathbf{w}$ to be $\mathbf{v}_1 \wedge \cdots \wedge \mathbf{v}_j$, where $\mathbf{v}_i = g_0 \mathbf{u}_i$ and $\mathbf{u}_1, \ldots, \mathbf{u}_j$ generate $\Z^k \cap V$ as a $\Z$-module). Assumption (2) is verified in \cite[Pf. of Thm. 3.3]{KLW} for the standard one parameter flow and for $g_0 = \mathrm{Id}$. For the general case we need here, we argue as in \cite[Section 3.3]{KWDirichletimprovable}, only using the drifting away from walls assumption. Namely,  in  the proof in \cite[Proof of Thm. 3.5]{KWDirichletimprovable}, the constant $\delta$ is allowed to depend on $g_0.$ This does not cause any issues and the same proof goes through.

\end{proof}
\subsection{Generalities on stationary measures}
\label{sec:generalities stationary measures}
We will use the following basic facts about stationary measures. Most of the results are valid in great generality but we only state the results which we will need here. See \cite{BQ1, BQbook} for details and proofs.

\begin{prop}\label{prop: commutes with RW}

Let $G$ be a locally compact second countable group, acting continuously on a locally compact second countable space $X$, and let $ \mu$ be a measure as in \eqref{eq: def mu}.  Let
\begin{equation}\label{eq: def B beta}
B \df \{1, \ldots, k\}^{\N} \cong \{h_1, \ldots, h_k\}^{\N},\ \ \  \ \ \ \beta  \df  \mu^{\N}.
\end{equation}
and $T: B \to B$ the shift map. Then:
\begin{itemize}
   \item
    The collection of $\mu$-stationary measures on $X$ is a compact convex set in the space of probability measures on $X. $ Ergodic stationary measures are extremal in this cone.
    \item
    If a homeomorphism $\varphi: X \to X$ commutes with each element in the support of $\mu$,   then it maps any (ergodic) $ \mu$-stationary measure to an (ergodic) $\mu$-stationary measure.

    \item
    An ergodic stationary measure which assigns full measure to  a countable subset of $X$  is supported on a finite set and gives equal mass to all elements in this set.
    \item
    For any $\mu$-stationary probability measure $\nu$, for $\beta$-almost every sequence $b=(h_n)_n$, the limit
    \begin{align}\label{eq: def nu b}
        \nu_b \df \lim_{n\to\infty} (h_1 \circ \cdots \circ h_n)_\ast \nu
    \end{align}
    exists and is a probability measure on $X$. The collection $(\nu_b)$ of {\em limit measures} satisfies
    \begin{equation}\label{eq: stationarity equivariance1} \text{ for every }i \text{ and }\beta\text{-a.e. } b \in B, \
    \nu_{h_ib}  = h_{i*} \nu_b,
    \end{equation}
    and
    \begin{equation}\label{eq: stationarity equivariance2}
\nu = \int \nu_b \; \der\beta(b) .
\end{equation}
\item Conversely, for any assignment $b\mapsto \nu_b$, where $\nu_b$ is a measure on $X$ and the assignment satisfies \eqref{eq: stationarity equivariance1}, the measure defined by \eqref{eq: stationarity equivariance2} is stationary and the collection $(\nu_b)_{b \in B}$ is its system of leafwise measures.
    \item

 Also  define
\begin{equation}
\label{eq: def beta X}
B^X \df B \times X, \  \ \ T^X(b,x) = \left(Tb, h_{b_1}^{-1}x \right), \ \ \ \beta^X \df \int_{B^X} \delta_b \otimes \nu_b \, \der \beta(b).
\end{equation}
Then $\beta^X$ is $T^X$-invariant, and is ergodic if and only if $\nu$ is ergodic.
    \end{itemize}
\end{prop}
We remark that what we denote here by $T^X$ is the {\em backward random walk transformation} which is denoted in \cite{BQbook} by $T^{\vee  X}.$ It should not be confused with the forward random walk transformation $(b,x) \mapsto (Tb, h_{b_1}x).$
\subsection{Irreducible $S$-adic lattices}\label{subsec: S adic lattices}
In this subsection we record some simple properties of the groups introduced in
\textsection \ref{subsec: RW S arithmetic}  which we will use repeatedly.

\begin{prop}\label{prop: in repeated use}
For any $x \in \XS$ and any nontrivial $g  = (g^{(\s)})_{\s \in S} \in G^S$ in the stabilizer of $x$, we have
\begin{equation}\label{eq: and thus lambda}
     \forall \sigma \in S, \ \ \ \ g^{(\sigma)} \neq \mathrm{Id}.\end{equation}
\end{prop}

\begin{proof}
    Since property \eqref{eq: and thus lambda} is invariant under conjugations, and the stabilizer of $x = \pi(g_0)$ is a conjugate of $\Lambda^S$, it suffices to consider the case that $x = \pi(\Lambda^S)$ and that $g \in \Lambda^S$. In this case \eqref{eq: and thus lambda} follows immediately from the definition of $\Lambda^S$ as a diagonal embedding,
\end{proof}

\begin{prop}\label{prop: projection Z trivial}
    For any $\s \in \Sue$, there are no nontrivial elements of   $\mathbf{G}(\Q_\s)$ which commute  with all of the matrices $h^{(\sigma)}_1, \ldots, h^{(\sigma)}_k$.
\end{prop}

\begin{proof}
Suppose $\s \in \Sue$, let
$$H \df \overline{\left \langle h_i^{(\s)}: i = 1, \ldots, k \right \rangle},$$
and let $C$ be the centralizer of $H$. Since $\mathbf{G}(\Q_\s)$ is center-free, it suffices to prove that $C$ is contained in the center of $\mathbf{G}(\Q_\s)$. Since $\mathbf{G}(\Q_\s)$ is a matrix group, it is enough to show that $H$ is epimorphic in $\mathbf{G}(\Q_\s)$; that is, in any finite-dimensional representation of algebraic groups, any vector fixed by $H$ is fixed by $\mathbf{G}(\Q_\s)$. In order to verify this statement, we can replace $H$ by its Zariski closure, and by using~\cite[Lemma 5.2]{Shah-LimitDistributions} or~\cite[Lemma 2.3]{KWDirichletimprovable}, we see that it suffices to show that the Zariski closure of $H$ contains the group
\begin{equation}\label{eq: def u+ sigma}
  U^{(\s)} \df u^{(\s)}(\Q^d_\s), \quad \text{ where } u^{(\s)} \colon \bQ_\s^d \to \mathbf{G}(\bQ_\s),
    \quad \quad u^{(\s)}(\xbf) \df \begin{pmatrix}
        \id & \xbf \\
        \mathbf{0} & 1
    \end{pmatrix}
\end{equation}
(indeed, the above-referenced results show that nonzero vectors fixed by $U^{(\sigma)}$ are expanded by the $h^{(\sigma)}_i$, and in particular, not fixed).

To this end, let
 $U_0 \df \bar{H}^Z \cap U^{(\s)}$, where $\bar{H}^Z$ denotes the Zariski closure of $H$.
 Since $U^{(\s)}$ is normalized by $H$, so is $U_0$.
 Also, for each $i,j$, $U_0$ contains the elements
 $$h_i^{(\s)} \circ \left(h_j^{(\s)}\right)^{-1}= u^{(\s)}(\mathbf{y}_j - \mathbf{y}_i);$$
 that is, $U_0$ contains the image under $u^{(\sigma)}$ of all the differences of the translations vectors appearing in the IFS $\Phi$. Let $V_0$ be the subspace of $\Q_{\s}^d$ spanned by these differences.
An easy calculation shows that the maps in $\Phi$ all preserve the affine subspace $(1- \varrho)^{-1} \mathbf{y}_1 + V_0$. Since $\Phi$ is irreducible, this means that $V_0 = \Q_{\s}^d$, and since $U_0$ is Zariski closed, this means that $U_0 = u^{(\s)}(V_0) = U^{(\s)}$, as required.
\end{proof}

\section{Examples of Stationary Measures}\label{sec: examples}
In analogy with Theorem \ref{thm: equidistribution main}, we define a measure $\hat \nu_0$ on $\X_{d+1}^S$ as the pushforward of $\theta$ under the map $\xbf \mapsto u(\xbf) \Lambda^S.$ 
 For a $\Q$-algebraic group $\mathbf{J}$ and $S_0 \subset S$ we will write
\begin{equation}\label{eq: notation places}\mathrm{J}_{S_0}\df \prod_{\sigma \in S_0} \mathbf{J}(\Q_\sigma).
 \end{equation}
We let $\mathbf{U}$ be the unipotent group which is the image of the embedding \eqref{eq: def u+}, in all places, and denote the  orbit of the identity coset $\Lambda^S$ under the group  $\mathrm{U}_S$ by $\mathscr{U}^S$.
This is a compact subset of $\X_{d+1}^S$, which is isomorphic to the compact $S$-arithmetic quotient
$$\mathbf{J}(\Q_S)/\mathbf{J}\left(\Z\left[ \frac{1}{S} \right] \right), \ \ \ \text{ where } \mathbf{J} = (\mathbf{G}_{\mathrm{a}})^d$$ and
$\mathbf{G}_{\mathrm{a}}$ is the additive group. We can also view $\mathscr{U}^S$ as a solenoid, that is, the natural extension of $\R^d/\Z^d$ for which the multiplication maps $\times_{p}, \, p \in \Sf$ are invertible; cf. \cite[Chap. II.10]{Hewitt_Ross} and ~\cite[Ex.~10.3]{EinsiedlerLuethi}. We denote the unique
$\mathrm{U}_S$-invariant measure on $\mathscr{U}^S$ by $m_{\mathscr{U}}.$
We denote by $A \subset \SL_{d+1}(\R)$ the group of diagonal matrices of determinant one.

We repeat here a simple computation from \cite[\textsection 4]{KhalilLuethi} that is central to our approach. By identifying $\SL_{d+1}(\R)$ with a subgroup of $\mathbf{G}(\R)$, we consider the map $u$ in \eqref{eq: def u+} as a map $\R^d \to G^S.$ From \eqref{eq: def h} one sees that for any $i \in \{1, \ldots, k\}$ and any $\xbf \in \R^d$ we have that
$
h_i u(\xbf) h_i^{-1} = u(\varrho \, \xbf).
$
Let $\lambda_i $ be the diagonal embedding of
$
 \begin{pmatrix}
    \varrho \, \mathrm{Id}_d & -\mathbf{y}_i \\
    \mathbf{0} & 1
  \end{pmatrix} \in \mathbf{GL}_{d+1}\left( \Z \left[ \frac{1}{S} \right] \right), $ so that
 $$
 \lambda_i \in \Lambda^S \ \ \text{ and } \
 h_i \lambda_i^{-1} =  u(\mathbf{y}_i).
 $$
As a consequence
\begin{equation}\label{eq: crucial identity}
h_i u(\xbf) \Lambda^S = u( \varrho \, \xbf)h_i \Lambda^S = u( \varrho \, \xbf ) \, u( \mathbf{y}_i) \Lambda^S = u(f_i(\xbf)) \Lambda^S.
\end{equation}

 \begin{prop}\label{prop: examples stationary measures}
     The  measures $m_{\X_{d+1}^S}, \, \hat \nu_0$ and $m_{\mathscr{U}}$ are all $\mu$-stationary. For any $a \in A$, the same holds for the pushforwards of these measures by $a$.
 \end{prop}

 \begin{proof}
    Since each $ h_i$
    preserves $m_{\X^S_{d+1}}$, it is clear that this measure is
    $\mu$-stationary.
    For $\hat \nu_0,$ we have from  \eqref{eq: crucial identity} that the map $\xbf \mapsto u(\xbf)\Lambda^S$ is a conjugacy between the actions of the  $f_i$ on $\R^d$, and of the  $h_i$ on the orbit $U\Lambda^S \subset \X_{d+1}^S$. Stationarity for $\mu$ now follows from \eqref{eq: stationarity theta}. For $m_{\mathscr{U}}$ we observe that  each $h_i$ belongs to the normalizer of $\mathrm{U}_S,$ and $h_i \lambda_i^{-1} \in \mathrm{U}_S$, and so by repeating the computation in \eqref{eq: crucial identity} and using that $\mathrm{U}_S$ is abelian, we have that
    \begin{equation}\label{eq: solenoid automorphism}
        \forall w \in \mathrm{U}_S, \ \ \ h_i w\Lambda^S =  u(\mathbf{y}_i) (h_i w h_i^{-1})\Lambda^S.
        \end{equation}
    That is, applying $h_i$ to $\mathscr{U}^S$ induces a solenoid affine endomorphism; i.e., a map which is the composition of a group endomorphism  and a group translation. Since its inverse is given by applying $h_i^{-1}$ this is actually an automorphism, and thus preserves Haar measure on $\mathscr{U}^S.$ This implies that every $h_i$ preserves $m_{\mathscr{U}}$, and in particular $m_{\mathscr{U}}$ is $\mu$-stationary.
    The last assertion follows from Proposition \ref{prop: commutes with RW}.
 \end{proof}

Note that $\hat \nu_0$ is supported on the
 solenoid $\mathscr{U}^S$ and is not invariant under the individual elements $ h_i$.

 \begin{remark}
     The measures appearing in Proposition \ref{prop: examples stationary measures} are all ergodic. For $\hat \nu_0$, this follows from the uniqueness in \eqref{eq: stationarity theta}. For $m_{\X_{d+1}^S}$ and for $m_{\mathscr{U}}$ one can give proofs along the lines of
     \cite[Lemma 7.1]{ChaikaKhalilSmillie}; we will not be using this statement and so we omit the details.
 \end{remark}

\begin{example}\label{example: new example}
For $ k \in\{1, \ldots, d\}$, let $\mathbf{U}^{(k)}$ denote the unipotent subgroup of $\mathbf{GL}_{d+1}$ whose Lie algebra is spanned by elementary matrices $E_{ij}$ for $i \in \{1, \ldots, k\}, \, j \in \{k+1, \dots, d+1\}$ (with this notation, the group $\mathrm{U}$ of the previous example is $U^{(d)}$). Let $\mathbf{N}^{(k)}$ denote the normalizer of $\mathbf{U}^{(k)}$ and let $N_1^{(k)}$ denote the subgroup of elements of $\mathbf{N}^{(k)}(\Q_S)$ whose action by conjugation on $\mathbf{U}^{(k)}(\Q_S)$ preserves the Haar measure (note that $N_1^{(k)}
$ is a topological group but not an $S$-adic group). It is not hard to check that each $h_i$ is contained in $N_1^{(k)}$. Thus, the $N_1^{(k)}$-invariant measure on any closed finite-volume orbit for $N_1^{(k)}$ is a stationary measure. 
\end{example}

 We now construct some additional ergodic $\mu$-stationary measures by  `finite-index perturbations' of the measures appearing in Proposition \ref{prop: examples stationary measures}. These examples are not surprising but are useful for understanding some of the conditions appearing in our results.

\begin{example}
    The orbit $U \Lambda^S$ is dense in the solenoid $\mathscr{U}^S$, and the measure $\hat \nu_0$ is supported on a compact subset of this orbit. Here we construct another $\mu$-stationary measure, supported on a different $U$-orbit in the solenoid. As we saw in \eqref{eq: solenoid automorphism}, this is equivalent to producing a stationary measure with this property on $\mathscr{U}^S$, invariant for the random walk by affine automorphisms.

    Write $\xbf = \left(x^{(\sigma)}\right)_{\sigma \in S}$ for an element of $\mathbf{G}_{\mathrm{a}}^d$,  and let $[\xbf]$  denote the corresponding coset in $
    \mathscr{U}^S.$ In these coordinates, the $U$-orbits in $\mathscr{U}^S$ are obtained by adding real numbers to the $\sigma = \infty$ coordinate of $\xbf,$ and $\hat \nu_0$ is supported on the orbit of the zero coset $\left[ 0 \right].$ Let $\varphi_i$ denote
    the affine automorphism induced by $h_i$. Then
    $$
    \varphi_i ([\xbf]) = \left[\left( \bar x^{(\sigma)} \right)_{\sigma
    \in S}  \right], \ \ \ \ \text{ where }
 \ \bar x^{(\sigma)} = \left\{  \begin{matrix} \varrho x^{(\sigma)} & \sigma \in S_{\mathrm{f}} \\ \varrho x^{(\sigma)} + \mathbf{y}_i & \ \sigma = \infty.  \end{matrix} \right.    $$

  For a first example, suppose that $\varrho = \frac{1}{q}$ for some $q \in \N$, and that $S_{\mathrm{f}}$ consists of the prime divisors of $q$ (in other words, the divisors of the translations $\mathbf{y}_i $ do not involve primes not dividing $q$).  Let $a$ be an integer coprime to $q$
 and let
 $$
  \mathbf{z} = \left(z^{(\sigma)} \right), \ \ \text{ where } z^{(\infty)} =0 \ \ \text{ and } \
    z^{(p)} \df \sum_{i \geq 0} aq^i \in \Q_p  \ \text{ for } p \in S_{\mathrm{f}}. $$
    This choice implies that for $p \in S_{\mathrm{f}},$
   $ \varrho z^{(p)} = \frac{a}{q} + z^{(p)}.$
    Since we can replace a coset representative by subtracting $\frac{a}{q}$ from each coordinate, we see that the random walk fixes the $U$-orbit through $[\mathbf{z}]$, acting on this $U$-orbit by $\xbf \mapsto \xbf + \mathbf{y}_i - \frac{a}{q}. $
Thus if we define
    $$
    \Phi' \df \{f'_1, \ldots, f'_k\}, \ \ \text{ where } \
    f'_i (\xbf) = f(\xbf)-\frac{a}{q}, $$
 let $\theta'$ be the unique stationary measure for $\Phi'$ on $\R^d$  (in the sense of \eqref{eq: stationarity theta}),
  and let $\nu'$ be the pushforward of $\theta'$ under the map $\xbf \mapsto u(\mathbf{z} + \xbf) \Lambda^S,$
   then $\nu'$ is a $\mu$-stationary measure on $\mathscr{U}^S$. It differs from $\hat \nu_0$ since it is supported on the orbit of $\left[\mathbf{z} \right]$.

   In this first example, the main property of the $U$-orbit of $u(\mathbf{z})\Lambda^S$ is that it is mapped to itself by the affine maps $\{\varphi_1, \ldots, \varphi_k)$. More complicated examples can be constructed by considering finite collections of $U$-orbits which are permuted by this collection of maps. In the preceding example, this will happen if in the definition of $\mathbf{z}$, we replace the series $\sum aq^i$ with a series with a periodic sequence of coefficients. We leave the details to the dedicated reader.
    \end{example}

\section{Deducing Theorems~\ref{thm: equidistribution main} and \ref{thm: diophantine main} from Theorem \ref{thm: dynamical}}\label{sec: deducing other results}

\subsection{Proof of Theorems~\ref{thm: equidistribution main} and \ref{thm: diophantine main}}
Let $\theta$ and $\nu_0$ be as in Theorem \ref{thm: equidistribution main}.
As in \textsection \ref{sec: examples}, let $\hat
\nu_0$ denote the pushforward of $\theta$ under $\xbf \mapsto u(\xbf)\Lambda^S$.
Let $P: \X_{d+1}^S \to \X_{d+1}$ be the projection, so that $\nu_0 = P_* \hat \nu_0.$
In this subsection we will assume Theorem \ref{thm: dynamical} and show:

\begin{theorem}\label{thm: equidistribution general}
    If $(a_n) \subset A$ is a sequence
   so that $\theta$ is uniformly non-divergent along $(a_n)$,
    then
    \begin{equation}\label{eq: drifting}
   \lim_{n\to \infty}  a_{n*} \hat \nu_0 =m_{\X_{d+1}^S}.
    \end{equation}
\end{theorem}

Since the map $P$ is equivariant for the action of $\SL_{d+1}(\R)$, Theorem \ref{thm: equidistribution main} follows immediately from Theorem \ref{thm: equidistribution general}.
For the proof, the key property of the diagonal element $b$ appearing in Theorem \ref{thm: dynamical}  is  that it commutes with $\mrm{PGL}_{d+1}(\R)$. This follows from \eqref{eq: def a2} and the fact that $\infty \notin S_{\mathrm{ue}}$.

\begin{proof}[Proof of Theorem \ref{thm: equidistribution general} assuming Theorem \ref{thm: dynamical}]
Let $\mathcal{T} = (a_n)$. The space of measures on $\X_{d+1}^S$ of total mass at most one, is compact with respect to the weak-* topology. Thus it suffices to show that \eqref{eq: drifting} holds along any subsequence of $\mathcal{T}$ for which the left-hand side converges. We will take a subsequence of $\mathcal{T},$ which we still denote by $(a_j)$ to simplify notation, and for which $\lim_{j \to \infty} a_{j*} \hat \nu_0 = \bar \nu_\infty$, where $\bar \nu_{\infty}$ is a measure on $\X_{d+1}^S$ of total mass at most one. We need to show $\bar \nu_{\infty} = m_{\X_{d+1}^S}.$ Along the proof we will freely pass to subsequences, which we will continue to denote by $(a_j)$.

We first claim that $\bar \nu_\infty$ is a probability measure. Indeed,  by the assumption of uniform non-divergence along $(a_n)$, there is no escape of mass in the real factor $\X_{d+1}$; that is, that any subsequential limit $\lim_{j \to \infty} a_{j*} \nu_0$ is a probability measure on $\X_{d+1}$.
 Since the projection $P$ is a proper map, this also implies that there is no escape of mass in the $S$-adic extension $\XS.$

Note that the projection of each of the elements $h_i$ to $\PGL_{d+1}(\R)$ is contained in the diagonal group, and thus commutes with the $a_j$. Also note from Proposition \ref{prop: examples stationary measures} that $\hat \nu_0$ is $ \mu$-stationary.
It follows by Proposition \ref{prop: commutes with RW} that $\bar \nu_\infty$ is a $\mu$-stationary measure. By Theorem  \ref{thm: dynamical}, there is $\eta \in [0,1]$ such that $\bar \nu_\infty = (1- \eta)  m_{\X_{d+1}^S} +\eta \nu' $, where $\nu'$ is a $ \mu$-stationary measure such that for $b$ as in \eqref{eq: def a2}, we have that $b^{k}_* \nu' \to_{k \to \infty} 0$.  We want to prove $\eta = 0.$

Assume for the sake of contradiction that $\eta>0$.
By
the assumption of uniform non-divergence along $(a_n)$, there is a compact $K \subset \X_{d+1}$ such that for any $x_0 \in \X_{d+1},$
 for all large enough $j$ we have
\begin{equation} \label{eq: on other hand}
a_{j*} \nu_{x_0}(K) \geq 1-\frac{\eta}{3},
\end{equation}
where $\nu_{x_0}$ is the pushforward of $\theta$ under the map $\xbf \mapsto u(\xbf)x_0.$
Let $\bar K \df P^{-1}(K) \subset \X^S_{d+1},$ where $P:\X_{d+1}^S \to \X_{d+1}$ is the projection. Since $P$ is proper, $\bar K$ is compact.

By the definition of $\eta,$ for some $k_0$ large enough we have
\begin{equation} \label{eq: on one hand}
b^{k_0}_* \bar \nu_\infty(\bar K) < 1-\frac{\eta}{2}.
\end{equation}

Now, using \eqref{eq: on one hand}, let $f$ be a compactly supported continuous function on $\X_{d+1}^S$ satisfying the pointwise bound $\mathbf{1}_{\bar K} \leq f \leq 1,$ and such that
\[\begin{split}
1 -\frac{\eta}{2}
> & \int f \, \der(b^{k_0}_* \bar \nu_\infty)  =  \int f(b^{k_0}x) \der \bar \nu_\infty(x)
= \lim_{j\to \infty} \int f(b^{k_0} a_j  x) \der \hat \nu_0(x)
\\  \geq &  \lim_{j\to \infty} \int \mathbf{1}_{\bar K} (b^{k_0} a_j  u(\xbf)\Lambda^S ) \der \theta(\xbf)
=  \lim_{j\to \infty} \int \mathbf{1}_{\bar K} (a_j u(\xbf) b^{k_0} \Lambda^S ) \der \theta(\xbf) \\
= &
\lim_{j\to \infty} \int \mathbf{1}_{ K} (a_j x ) \der \nu_{x_0}(x) =\lim_{j\to\infty} (a_{j*}\nu_{x_0}) ( K),
\end{split}\]
where $x_0  = P(b^{k_0}\Lambda^S)\in \X_{d+1}
.$
Note that in the second line we used the commutation relation between $b$ and the elements of the random walk.
We have obtained  a contradiction to \eqref{eq: on other hand}.
\end{proof}

\begin{proof}[Proof of Theorem \ref{thm: diophantine main}]
Given a norm $\| \cdot \|$ on $\R^{d+1},$ let
$\vre_{\| \cdot \|} $ be as in \eqref{eq: def epsilon norm}
and let $\vre<\vre_{\| \cdot \|}.$ Let $\mathbf{r}$ be a weight vector. We need to show that for $\theta$-a.e. $\xbf \in \R^d$, the trajectory $\left\{a^{(\mathbf{r})}_t \Lambda_{\xbf} : t>0 \right\}$ visits $\mathcal{U}$ along an unbounded set, where $$
\mathcal{U} = \left\{\Lambda \in \X_{d+1}: \Lambda \cap \overline{B}(0, \vre) = \{0\} \right \}.
$$
Here $\overline{B}(0, \vre)$ is the closed ball of radius $\vre$ around the origin in $\R^d$, with respect to the given norm.
Note that $\mathcal{U}$ is open, is nonempty by definition of $\vre_{\| \cdot \|},$ and hence $$
\delta \df m_{\X_{d+1}}\left( \mathcal{U}\right) \ \ \text{ satisfies } \delta >0.
$$

 Assume for the sake of contradiction that there is $t_0>0$ such that
$$
\theta \left( B \right)>0, \ \ \ \text{ where } B \df
\left\{\xbf \in \R^d: \forall t \geq t_0, \ a^{(\mathbf{r})}_t \Lambda_{\xbf} \notin \mathcal{U} \right \}.
$$
Using Proposition \ref{prop: Lebesgue density}, there is $f : \R^d \to \R^d$ of the form $f = f_{i_1} \circ \cdots \circ f_{i_n}$ such that
$$\frac{\theta(B \cap f(\mathcal{K})) }{ \theta(f(\mathcal{K}))}\geq 1-\frac{\delta}{2}.$$
Let $\theta_1 = f_* \theta $. Because of the self-similar structure, $\theta_1$ is the normalized restriction of $\theta$ to $f(\mathcal{K})$, and we have $\theta_1(B) \geq 1 - \frac{\delta}{2}.$
Let $\bar \nu_1$ be the measure obtained by pushing forward $\theta_1$ under the map $\xbf \mapsto u(\xbf) \Lambda^S.$ By Propositions \ref{prop: fractal measures are friendly} and \ref{prop: KLW nondivergence}, the measure $\theta_1$ satisfies uniform non-divergence along any unbounded subsequence $(a_j) \subset \left\{a^{(\mathbf{r})}_t \right\}$. Also, the definition of $B$ implies that for all $t \geq t_0,$
$$
a^{(\mathbf{r})}_{t*}\bar \nu_1(\mathcal{U}) \leq \frac{\delta}{2}.
$$
On the other hand, $\theta_1$ is the self-similar measure for the conjugated IFS $\Phi'$ as in \eqref{eq: if we let}, and thus we can apply Theorem \ref{thm: equidistribution general} to
obtain that
$$\liminf_{t\to
\infty} a^{(\mathbf{r})}_{t*}\bar \nu_1 (\mathcal{U}) \geq  m_{\X_{d+1}}(\mathcal{U}) = \delta.$$
This gives the desired contradiction.
\end{proof}

\begin{remark}
Following \cite{KWDirichletimprovable}, one can define Dirichlet improvability along any unbounded sequence $(a_n)$ of the diagonal group. Our arguments then show that for any measure $\theta$ as in Theorem \ref{thm: equidistribution main},   and any sequence $(a_n)$ drifting away from walls, the set of $\mathbf{x}$ which are Dirichlet improvable along $(a_n)$ are a nullset with respect to $\theta$.

\end{remark}

\section{Measure Classification for an Auxiliary Random Walk and Proof of Theorem~\ref{thm: dynamical}}
\label{sec: mu vs mu bar}

 Theorem \ref{thm: dynamical} is derived from a detailed analysis of stationary measures for a related measure $\bar \mu.$
We now introduce this random walk and state the corresponding measure rigidity result, Theorem~\ref{thm: stationary main} below.

Let $S_{\mathrm{ue}}$ be as in \eqref{eq: def Sue}. Recall that $\varrho = \frac{r}{q}$ with $\gcd(r,q)=1$, and let
\begin{equation}\label{eq: def Sdt}
\Sdt \df \{\infty\} \cup \left \{p \in \Sf: p|r \right \},  \ \ \ \ \
 S_{\mathrm{tr}} \df \left\{p \in \Sf : \gcd(p,q)=\gcd(p,r)=1 \right\} , \end{equation}
so that
 \begin{equation}\label{eq: partition of places}
S = S_{\mathrm{ue}}\sqcup \Sdt \sqcup S_{\mathrm{tr}} .
 \end{equation}
 In the remainder of the paper,  for an algebraic group $\mathbf{J}$, we will simplify the notation \eqref{eq: notation places} by writing $\mrm{J}_{\mrm{ue}} = \mrm{J}_{\Sue}, \, \mrm{J}_{\mrm{dt}} = \mrm{J}_{\Sdt}, \, \mrm{J}_{\mrm{tr}} = \mrm{J}_{\Str}$.

 Let $\mathbf{U} \subset \mathbf{G}$ denote the algebraic group consisting of elements of the form appearing in \eqref{eq: def u+},
 and let
\begin{equation}\label{eq: def Wst} W^{\mathrm{st}} \df \mathrm{U}_{\mathrm{ue}}\times \mathrm{G}_{\mathrm{dt}} \times \mathrm{G}_{\mathrm{tr}}.
\end{equation}
The subscript 
`st' stands for
`stable'.

Using the same probability vector $\mathbf{p}$ appearing in \eqref{eq: def mu}, we let
\begin{equation}\label{eq: def mu bar}
\bar \mu \df \sum p_i \delta_{\bar h_i},
\end{equation}
where
$\bar h_i=\left(\bar h_i^{(\sigma)}\right)_{\sigma \in S}$
  are defined by
\begin{equation}\label{eq: def h'}
\bar h_i^{(\sigma)} \df \left\{  \begin{matrix}
\begin{pmatrix}
     \varrho \, \mathrm{Id}_d & -\mathbf{y}_i \\
    \mathbf{0} & 1
  \end{pmatrix} & \ \ \ \ \text{if } \sigma \in S_{\mathrm{ue}} \\
\begin{pmatrix}
    \varrho \, \mathrm{Id}_d & \mathbf{0} \\
    \mathbf{0} & 1
  \end{pmatrix}
 & \ \ \ \text{ if } \sigma \in \Sdt \\     \mathrm{Id} & \ \ \, \ \text{ if } \sigma \in S_{\mathrm{tr}}.
 \end{matrix}
  \right.
  \end{equation}

The subscripts dt, tr stand respectively for  `deterministic' and `trivial'. To explain this terminology, and motivate the use of this modified random walk,  note first
 that words in the generators $\hibar$ always remain at a uniformly bounded distance from the corresponding words in the generators $h_i$. As a result, the two random walks exhibit similar dynamical properties. On the other hand, the $\bar\mu$-random walk is easier to analyze. The directions tangent to $\mrm{G}_{\mrm{tr}}$ remain bounded in forward and backward time under the $\mu$-random walk, but these directions are left  unchanged by the $\bar\mu$-random walk. Additionally, since the $\Sdt$-coordinates of $\hibar$ are simultaneously diagonal, the behavior of the $\bar \mu$ random walk on these coordinates is deterministic, while still capturing the essence of the action of the $\mu$-random walk.

 Note that when the numerator of the contraction ratio $\varrho$ is equal to $1$, and when all the primes appearing in the denominators of the $\mathbf{y}_i$ also appear in the denominator of $\varrho$, we have $h_i = \bar h_i$ and $
 \mu = \bar \mu.$
 However, in the general case, these random walks are different. 
 
 We now present  our result about $\bar \mu$-stationary measures. 
 Given a closed unimodular subgroup $M\subseteq G^S$, equipped with a Haar measure $m_M$, we denote by $N_1(M)$ the {\em unimodular normalizer of $M$}, that is the set of elements $g\in G^S$ such that $gMg^{-1}=M$ and conjugation by $g$ preserves $m_{M}$. 

\begin{thm}\label{thm: stationary main}
    Let $\Phi$ be an irreducible carpet IFS satisfying the open set condition, let $\bar h_i$ be as in \eqref{eq: def h'},  let $\mathbf{p}$ be a probability vector, let $\bar \mu$ be as in \eqref{eq: def mu bar}, and let $\Wst$ be as in \eqref{eq: def Wst}. Then for any $\bar \mu$-ergodic $\bar \mu$-stationary measure $\nu$, one of the following holds:
    \begin{enumerate}
    \item  \label{item: case 2}
    $
    \nu\left( \left\{x \in \X_{d+1}^S: \mathrm{Stab}(x) \cap W^{\mathrm{st}} \neq \{\mathrm{Id} \} \right\} \right) =1.$
    \item \label{item: case 1} there is $g_0 \in G^S$ and  a closed subgroup $M \subset G^S,$  such that for $x_0 = g_0 \Lambda^S \in \X_{d+1}^S$, 
    we have:
    \begin{enumerate}[(A)]
    \item \label{item: A list} $\nu$ is $M$-invariant;
    \item \label{item: B list}$Mx_0$ is a closed orbit and $M \cap \mathrm{Stab}(x_0)$ is a lattice in $M$; 
    \item \label{item: B1 list}
    The conjugate $g_0^{-1} M g_0$ is of finite index in $\mathbf{M}(\Q_S)$, where $\mathbf{M}$ is a $\Q$-algebraic subgroup of $\mathbf{M}$;
    \item \label{item: C list}
    $N_1(M)x_0$ is a closed orbit satisfying
    $\nu(N_1(M)x_0)=1$; 
    \item \label{item: D list}
    $\bar h_i \in N_1(M)$ for each $i$; 
    \item \label{item: E list} the element $b$ of \eqref{eq: def a2} normalizes $M$, and we have $b_* m_{M}=c\, m_{M}$ for $c<1$.
    \end{enumerate}
    
    \item \label{item: case 3} $\nu = m_{\X_{d+1}^S}.$
    \end{enumerate}
\end{thm}

It would be interesting to classify all ergodic stationary measures for the random walks generated by $\mu$ and $\bar \mu.$

\subsection{Compact and contracting extensions of random walks}
In this section, we construct a common cover of the random walks with laws $\mu$ and $\bar{\mu}$ and prove the key Proposition~\ref{prop:prod structure of compact extension}, which   transfers measure rigidity results from one to the other.  For related results on extensions of other random walks, see \cite[\textsection 5]{SimmonsWeiss} and \cite[Thm. 1.1]{AggarwalGhosh}.

Let $k_i\in G^S$ be such that $k_i\bar{h}_i =h_i$.
More explicitly, we have that $k_i=\left(k_i^{(\s)}\right)_\s$, where
\begin{align*}
    k_i^{(\sigma)} \df \left\{  \begin{matrix}
    \mathrm{Id} & \ \ \ \ \ \ \  \ \text{ if } \sigma \in S_{\mathrm{ue}} \cup \{\infty\} \\
 \begin{pmatrix}
    \mathrm{Id}_d &  -\mathbf{y}_i \\
    0 & 1
  \end{pmatrix}
 & \ \ \ \, \ \ \ \ \  \text{ if } \sigma \in \Sdt \sm \{\infty\}\\
  \begin{pmatrix}
     \varrho \, \mathrm{Id}_d & -\mathbf{y}_i \\
    0 & 1
  \end{pmatrix}
   &  \,\,\text{if } \sigma \in S_{\mathrm{tr}}.
 \end{matrix}
  \right.
\end{align*}

Consider the subgroup $F=\prod_{\s\in S}F^{(\s)}$ of $G^S$ defined as follows:
\begin{align*}
    F^{(\s)} \df \begin{cases}
        \{\id\} &  \s \in \Sue \cup\{\infty\} \\
        \mathbf{U}(\Q_\s) & \s \in \Sdt\sm \{\infty\} \\
        \overline{\left \langle k_i^{(\s)} \right\rangle}
        & \s \in \Str.
    \end{cases}
\end{align*}
Note that the group $F$ is normalized by both $h_i$ and $\bar{h}_i$ for all $i$.
We shall consider an extension of our random walks on $\X_{d+1}^S$ to the space
\begin{align*}
    E = \X_{d+1}^S \times F.
\end{align*}
The random walk is generated by the following elements:
\begin{align}\label{eq:RW on E}
    \hibarF (x,f) \df (\bar{h}_i x, k_i \bar{h}_i f \bar{h}_i^{-1}).
\end{align}
In analogy with \eqref{eq: def mu} and \eqref{eq: def mu bar}, define the law  of this random walk as
\begin{align*}
    \bar{\mu}^F = \sum_i p_i \delta_{\bar h_i^F}.
\end{align*}
We define two projections $P,\bar{P}:E\to \X^S_{d+1}$ by $P(x,f) = fx$ and $\bar{P}(x,f)=x$.
Then, we note the following equivariance properties of these projections:
\begin{align}\label{eq:equivariance under both projections}
    P\left(\hibarF(x,f) \right) = h_i P(x,f), \qquad
    \bar{P}\left(\hibarF(x,f) \right) = \bar{h}_i \bar{P}(x,f).
\end{align}
The projections $P, \bar P$ and the equivariance properties \eqref{eq:equivariance under both projections} essentially show that  the space $E$, equipped with the random walk generated by the maps $\bar h_i^F$, is a  common extension for the random walks driven by $\mu$ and $\bar \mu$. This will allow us to transfer information about stationary measures from one random walk to the other.

\begin{prop}\label{prop:prod structure of compact extension}
    There exists a probability measure $m_F$ on $F$ such that every $\barmu^F$-stationary probability measure $m$ on $E$ with $\bar{P}_\ast m= m_{\XS}$ is of the form $m=m_{\XS}\otimes m_F$.
\end{prop}
\begin{proof}
    We will deduce this result from a combination of uniqueness results for Haar measures on compact groups, stationary measures for contracting IFS's, and from the mixing property of the $\bar\mu$-random walk on $\XS$.

    We begin by giving a more explicit description of the action of the random walk on the $F$-coordinate. Let
    \begin{equation}\label{eq: definition Q}
    Q \df \prod_{\s\in \Str} F^{(\s)}. 
    \end{equation}
    Since $F^{(\sigma)}$ is compact for $\sigma \in \Str$, $Q$ is a compact factor of $F$.
    For $f = \left( f^{(\sigma)} \right) \in F$ we have
    \begin{align*}
        (k_i \bar{h}_i f \bar{h}_i^{-1})^{(\s)} =
        \begin{cases}
            \id & \text{ if } \s \in \Sue\cup \{\infty\} \\
            \begin{pmatrix}
                \id_d & -\mathbf{y}_i + \varrho \mathbf{x} \\ 0 & 1
            \end{pmatrix}
            & \text{ if } \s \in \Sdt \sm \{\infty\}
            \text{ and } f^{(\sigma)} = u(\xbf),
  \ \xbf \in \Q_\s^d
            \\
            k_i^{(\s)} f^{(\sigma)} & \text{ if } \s \in \Str.
        \end{cases}
    \end{align*}
In particular, if $f \in Q$ then $k_i \bar h_i f \bar h_i^{-1} = k^Q_if \in Q,$ where $k^Q_i$ is the projection of $k_i$ to $Q.$

    Let $m$ be a $\barmu^F$-stationary measure with $P_\ast m= m_{\XS}$.
    Let $\bar{P}_0:\XS\times F \r \XS \times Q $ and $\bar{P}_1 : \XS\times Q \r \XS$ denote the standard projections so that $\bar{P}=\bar{P}_1 \circ \bar{P}_0$.
    We first show that
    \begin{align}\label{eq:product with Q}
        (\bar{P}_0)_\ast m = m_{\XS}\otimes m_Q,
    \end{align}
    where $m_Q$ denotes the Haar probability measure on $Q$.

    To this end, let $\Gamma^Q$ denote the group generated by $\bar{h}_i^Q$, where $\bar{h}_i^Q$ denotes the restriction of the action of $\hibarF$ to $\XS\times Q$.
    Our claim will follow by~\cite[Proposition 5.3]{SimmonsWeiss} upon verifying ergodicity of the action of $\Gamma^Q$ with respect to $m_{\XS}\otimes m_Q$.
    By the Howe-Moore theorem, since the group generated by $\hibar$ is unbounded, its action on $\XS$ is weak-mixing with respect to $m_{\XS}$.
    Moreover, since the group generated by $\left(k_i^{(\s)} \right)_{\s\in \Str} $ is dense in $Q$, its action by left translations on $Q$ is ergodic with respect to $m_Q$.
    Hence, the action of $\Gamma^Q$ on $\XS\times Q$ is ergodic with respect to the product measure $m_{\XS}\otimes m_Q$; cf.~\cite[Theorem 9.23(2)]{GlasnerBook}. This gives
    \eqref{eq:product with Q}.

    To conclude the proof, we
    define
    $$\mathrm{U}_0 \df \mrm{U}_{\Sdt \sm \{\infty\}} = \prod_{\s \in \Sdt \sm \{\infty\}}F^{(\s)} ,$$
    which is a factor of $F$ complementary to $Q$,
    and
    take advantage of the fact that action of the random walk on $\mathrm{U}_0 $ is given by contracting affine maps.
    Let
    \begin{align}\label{eq: right smack}
        m_b \df \lim_{n\to\infty} \left(\bar{h}^F_{i_1}\circ \cdots \circ \bar{h}^F_{i_n} \right)_\ast m
    \end{align}
    be the limit measures as in Proposition \ref{prop: commutes with RW}, so that
    \begin{equation}\label{eq: integrating back}
    m = \int m_b\, \der(\barmu^F)^\N(b).
    \end{equation}
    We will show that for almost every $b$, there is $u_b\in \mrm{U}_0$ such that
    \begin{align}\label{eq:Siegel times Haar times Dirac}
        m_b=m_{\XS}\otimes m_Q \otimes \d_{u_b}.
    \end{align}
    Plugging into \eqref{eq: integrating back}, we see that this
    implies the proposition with $m_F = m_Q \otimes \int \d_{u_b}\;\der (\barmu^F)^\N(b)$.

    By~\eqref{eq:product with Q}, the projection of $m$ to $\XS\times Q$ is invariant by elements of $\Gamma^Q$, and hence, by \eqref{eq: right smack} we have that $(\bar{P}_0)_\ast m_b = m_{\XS}\otimes m_Q$.
    Moreover, note that $|\varrho|_\s <1$ for all $\s\in \Sdt$ by definition of the deterministic places.
    Hence, letting $\bar{P}^{U_0}: E \r \mrm{U}_0$ denote the standard projection, it follows that $\left(\bar{P}^{U_0}\right)_\ast m_b$ is a Dirac mass at some point $u_b\in \mrm{U}$.
    This implies~\eqref{eq:Siegel times Haar times Dirac} and concludes the proof.
\end{proof}

\subsection{Escape of mass, and derivation of Theorem~\ref{thm: dynamical} from Theorem~\ref{thm: stationary main}}
Let $\nu$ be an ergodic $\mu$-stationary measure on $\XS$ and let $\nu^F$ be some  lift of $\nu$ to $E$, that is a measure satisfying  $P_\ast \nu^F=\nu$;  the existence of such a lift can be constructed using a Borel section of the projection $P:E\to\XS$. We will choose the section so that the $F$-coordinate takes values in $Q$. This choice implies that the sequence of measures $\left(\barmu^F\right)^{\ast n} \ast \nu^F$ are all supported on  $\XS \times Q$.

Let $\bar{\nu}^F$ be a weak-$\ast$ limit measure of the sequence of measures
\begin{align*}
    \nu^F_N \df \frac{1}{N} \sum_{n=1}^N \left(\barmu^F\right)^{\ast n} \ast \nu^F
\end{align*}
as $N\to \infty$.
Then,~\eqref{eq:equivariance under both projections} implies the following properties of $\bar{\nu}^F$.
\begin{lem}\label{lem:equivariance}
    For all $N\geq 1$, the measures $\nu^F_N$ satisfy
    \begin{align*}
        P_\ast \nu^F_N = \nu.
    \end{align*}
    In particular, $\bar{\nu}^F$ is a $\barmu^F$-stationary probability measure.
    Moreover, we have that $\bar{P}_\ast \bar{\nu}^F$ is $\barmu$-stationary.
\end{lem}
\begin{proof}
    Since $\nu$ is $\mu$-stationary, the $P$-equivariance in~\eqref{eq:equivariance under both projections} implies that $P_\ast(\barmu^F\ast \nu^F)=\nu$.
    This implies the first claim.
    The second claim follows from the first assertion, Proposition \ref{prop: commutes with RW}, and the fact that $\XS \times Q \to \XS$ is a compact extension, so there is no loss of mass in taking limits.
    The last claim is immediate from the averaging construction of $\bar{\nu}^F$ and the $\bar{P}$-equivariance in~\eqref{eq:equivariance under both projections}.
\end{proof}

We now use the structure of the lattice $\Lambda^S, $ to obtain the following statement:
\begin{prop}\label{prop: escape of mass case 2}
Let  $W^{\mathrm{st}}$ and $b$ be as in \eqref{eq: def Wst} and \eqref{eq: def a2}, and let  $x_0 \in \X_{d+1}^S$ such that $\mathrm{Stab}(x_0) \cap W^{\mathrm{st}} \neq \{ \mathrm{Id} \}.$ Then $b^n x_0 \to \infty $ as $n\to \infty.$
\end{prop}
\begin{proof}
Recall from Proposition \ref{prop: in repeated use} that \eqref{eq: and thus lambda} holds for any nontrivial element in the stabilizer of any $x \in \XS$.
Let
  $w \in W^{\mathrm{st}} \sm \{\mathrm{Id}\}$ with $wx_0 = x_0$.
  Suppose by contradiction that for some subsequence $n_j \to \infty$ we have $b^{n_j}x_0 \to x_\infty,$  for some $x_\infty \in \X_{d+1}^S$. The stabilizer of $b^{n_j}x_0$ contains $b^{n_j}wb^{-n_j}$. Thus any accumulation point $w_\infty$ of the sequence
  $(b^{n_j}wb^{-n_j})_j$ belongs to the stabilizer of $x_\infty.$
  By definition of $\Sue$, we have that $|\varrho^{-1}|_\s <1$ for all $\s\in \Sue$.
  In particular, the conjugation action of $b$ on $\mrm{U}_{\mrm{ue}} \df \mrm{U}_{S_{\mrm{ue}}}$ is contracting.
  Hence, writing $w = \left(w^{(\sigma)} \right)_{\sigma \in S}$ and $w_\infty = \left(w_\infty^{(\sigma)} \right)_{\sigma \in S}$, we see that
  $$
  w_\infty^{(\sigma)} = \left\{ \begin{matrix}  \mathrm{Id} & \sigma \in S_{\mathrm{ue}} \\ w^{(\sigma)}
  & \sigma \in S \sm S_{\mathrm{ue}}.
  \end{matrix} \right.$$
  This contradicts \eqref{eq: and thus lambda}.
\end{proof}

\begin{corollary}
    \label{cor: divergence measure Wst}
If $\nu$ is a finite measure on $\X_{d+1}^S$
    and $b \in G^S$ satisfies that $b^nx_0\to \infty$ for $\nu$-a.e.\,\,$x_0$, then $b^n_* \nu \to 0.$ In particular,
 if $
    \nu\left( \left\{x \in \X_{d+1}^S: \mathrm{Stab}(x) \cap W^{\mathrm{st}} \neq \{\mathrm{Id} \} \right\} \right) =1$, then
$b^n_{\ast}\nu \to 0$ as $n\to \infty.$
\end{corollary}
\begin{proof}
We prove the first assertion; the second assertion follows immediately in view of  Proposition \ref{prop: escape of mass case 2}.
For any compact $K \subset \X_{d+1}^S$ and for $\nu$-a.e.\,\,$x \in \X_{d+1}^S$, by assumption there is $n_0 = n_0(x,K) $ such that for all $n\geq n_0,$ $b^nx \notin K.$ This implies that for any $\vre>0$, there is $n_0 = n_0(K)$ such that
$$\nu\left( \left\{x \in \X_{d+1}^S: \forall n \geq n_0, \ b^nx \notin K \right \} \right)< \vre.$$
Thus any measure $\nu'$  which is an accumulation point of the sequence $\left(b^n_{*} \nu\right)$ satisfies $\nu'(K) \leq \vre.$ Since $K$ and $\vre$ were arbitrary we have that $b^n_{*}\nu \to 0.$
    \end{proof}

\begin{lem}\label{lem:transfer of trivial stabilizers}
    Suppose that $\nu$ assigns zero mass to the set of points $x$ with 
    $b^kx\to \infty$.
    Then, the same holds for $\bar{P}_\ast(\bar{\nu}^F)$.
\end{lem}
\begin{proof}
Suppose $x \in \X_{d+1}^S$ and $\bar x \in \bar P(P^{-1}(x)).$ Then $\bar x \in Q x$, where $Q$ is as in \eqref{eq: definition Q}. Since $b$ commutes with all elements of $Q$, $b^kx\to \infty$ implies $b^k \bar x \to \infty$. Since 
$$\mathrm{supp}\left(\bar P_* \left(\bar {\nu}^F \right) \right) \subset \bar P \left(P^{-1}(\mathrm{supp}(\nu)) \right),$$ 
the assertion follows.  
\end{proof}

\begin{proof}[Proof of Theorem \ref{thm: dynamical}]
Let $\nu$ be an ergodic $\mu$-stationary measure on $\XS$, for which $b^k_* \nu \not \to 0.$ Let $\bar \nu^F$ be a $\bar \mu^F$-stationary measure on $E$ as in Lemma \ref{lem:equivariance}, and let $\bar \nu = \bar P_* \bar \nu^F.$ By Corollary \ref{cor: divergence measure Wst} and Lemma \ref{lem:transfer of trivial stabilizers}, we have that $b^kx\not \to \infty $ for $\bar \nu$-a.e. $x \in \XS$. By Proposition, \ref{prop: escape of mass case 2} we see that case \eqref{item: case 2} of Theorem \ref{thm: stationary main} cannot hold. 
\newline \indent
We claim that case \eqref{item: case 1} cannot hold either. Indeed, items \eqref{item: B list} and \eqref{item: C list} imply that for every $x \in \supp (\nu) \subset N_1(M)x_0$, we have that $M_Sx$ is a finite-volume closed $M_S$-orbit. But then \eqref{item: E list} shows that $b $ maps $M_Sx$ to a closed orbit of strictly smaller volume, and applying Chabauty's theorem (see \cite[Thm. I.20]{Raghunathan}) to the stabilizers $M_S \cap \mathrm{Stab}(b^kx)$, this implies that the injectivity radius at $b^kx$ goes to $0$ as $k \to \infty$. That is, $b^kx\to \infty,$ establishing the claim.
\newline \indent
It follows from Theorem \ref{thm: stationary main} that $\bar \nu = m_{\XS}. $ Now by Proposition \ref{prop:prod structure of compact extension} we see that $\bar \nu^F = m_{\XS}
 \otimes m_F$, and since $m_{\XS}$ is $F$-invariant, $\nu = P_* \bar \nu^F = m_{\XS}. $
    \end{proof}

\begin{remark}
We did not give an explicit example of a $\bar \mu$-stationary measure satisfying conclusion \eqref{item: case 2} in Theorem \ref{thm: stationary main}, but the preceding arguments show that such measures exist. Indeed, start with a $\mu$-stationary measure on $\XS$ which is supported on the solenoid $\mathscr{U}^S$, and note that
    the construction of $\bar P_* \bar \nu^F$ in the proof of Lemma \ref{lem:equivariance} gives rise to such a measure. Similarly, starting with a $\mu$-stationary measure as in Example \ref{example: new example}, we obtain a $\bar \mu$-stationary measure for which conclusion \eqref{item: case 1} holds.
\end{remark}

\section{Growth properties of the random walk}\label{sec:GrowthPropertiesRandomWalk}
Following Benoist and
Quint~\cite{BQ1}, we need to understand the growth properties of a random walk generated by the support of $\bar \mu$, acting linearly via several finite dimensional linear representations of $G^S.$
One major obstruction to running the same arguments given in  \cite{BQ1} without change is the absence of uniform expansion, which cannot be expected
in the case where the Zariski closure of the group generated by the random walk
is solvable.
However, it was noted in joint work of David Simmons with the third-named author, that in some cases a useful analogue is true;
cf.~\cite[Prop.~3.1(a)]{SimmonsWeiss}.
The goal of this section is to prove several analogous growth properties for the action of the random walk in its Adjoint representation on $\mf{g}_\s, \s\in \Sue$.

Throughout this section, we fix an irreducible carpet-IFS $\Phi$  with common rational contraction ratio $\varrho > 0$.
Given a finite list of indices $i_1, \ldots, i_n$ in $\{1, \ldots, k\}$, following \cite{SimmonsWeiss}
we write
\begin{equation}
\label{eq: weird notation from SW}
\bar h_1^n \df  \bar h_{i_n} \circ \cdots  \circ \bar h_{i_1}
\ \ \ \text{ and } \ \ \ \bar h_n^1 \df  \bar h_{i_1} \circ \cdots \circ \bar h_{i_n},
\end{equation}
where the $\bar h_i$ are as in \textsection \ref{sec: mu vs mu bar}.

Further, for $\sigma \in S$ and $x \in \Q$, $|x|_\sigma$ denotes the $\sigma$-adic absolute value of $x$,
$\mathfrak{g}_\s$
denotes the Lie algebra of $\mathbf{G}_\s$, where we view $\mathfrak{g}_\s$ as a vector space over $\Q_\s$, equipped with the $\s$-adic norm.
 We denote the Lie algebra of $G^S$ by $\lieg,$ that is,
$$\lieg = \bigoplus_{\s \in S} \lieg_\s.$$

\subsection{Action on the Lie algebra}
In this subsection, for $h  = (h_\s)_{\s \in S},\ \|h\|_\s$ denotes the operator norm of the action of $\Ad(h_\s)$ on  $\mathfrak{g}_\s$.
The following Proposition implies that the norm of each of the random walk elements in each place is dictated by the growth of the scalar contraction ratios.

\begin{prop}\label{prop:predictable growth}
There is a constant $C>1$ such that for any
    $n\in \N,$ the following holds. Denoting by $A_n \asymp_C B_n$ the inequalities $C^{-1} A_n \leq B_n \leq CA_n$,  for any word $
    (i_1,\dots,i_n)$ of length $n$,
     for all $\s\in S$ we have
    \begin{align}\label{eq: norm estimate we want}
        \left\| \bar h_1^n \right\|_\s \asymp_C \begin{cases}
             |\varrho|^n_\s, & \s \in\Sue, \\
             |\varrho|_\s^{-n}, & \s \in \Sdt.
        \end{cases}
    \end{align}
\end{prop}

\begin{proof}
The element  $\begin{pmatrix}
        \varrho^n \mrm{Id}_d & 0 \\
        0 & 1
    \end{pmatrix}$
    acts on $\mathfrak{g}_\sigma$ diagonally, with three eigenvalues $\varrho^n, 1, \varrho^{-n}.$  Thus in the case $\s\in \Sdt$, the statement follows from
\eqref{eq: def h'}.

For $\s \in\Sue$, we first show that the left-hand side of \eqref{eq: norm estimate we want}
dominates the right-hand side. Indeed, let $v=\left(\begin{smallmatrix}
    0 & X \\ 0 & 0
\end{smallmatrix}\right)$, for some $X\in \Q_\s^d$ with $\norm{X}_\s =1$.
Then $\mrm{Ad}\left(\bar h_1^n \right)(v) = \varrho^n v$, and hence $\left\| \bar h_1^n \right\|_\s \geq |\varrho|_\s^n$.

For the opposite inequality,
denote by
$$\mathbf{y}_1^n = f_{i_n}\circ \cdots \circ f_{i_1}(0)$$
the translation vector  of the map $f_{i_n}\circ\cdots \circ f_{i_1}$.
Then by a straightforward induction  using \eqref{eq: def fi} and \eqref{eq: def h'} we have 
\begin{equation}\label{eq: useful formula and correct}
   \left (\bar h_1^n \right)_{\s}  
    = \begin{pmatrix}
        \varrho^n \mrm{Id}_d & 0 \\
        0 & 1
    \end{pmatrix}
    \begin{pmatrix}
        \mrm{Id}_d & -\varrho^{-n}\ybf_1^n \\
        0 & 1
    \end{pmatrix}
    .
\end{equation}
Since the operator norm is sub-multiplicative, for an upper bound it suffices to give separate upper bounds for the two elements in this product. As we saw in the case $\s \in \Sdt$, the operator norm of the first element  $\begin{pmatrix}
        \varrho^n \mrm{Id}_d & 0 \\
        0 & 1
    \end{pmatrix}$
   is $|\varrho|_
\s^n$. Again by a straightforward induction we see that we can express each of the coefficients of $\varrho^{-n}\ybf_1^n$  as a sum $\sum_{i=0}^n b_i \varrho^{-i}$, where the $b_i$ are contained in the set of coordinates of the vectors $\ybf_1, \ldots, \ybf_k$, which is a finite set. By the ultrametric property of the $\sigma$-adic absolute value and the fact that $|\varrho|_\s \ge 1$, we deduce that
the operator norm of the second element $ \begin{pmatrix}
        \mrm{Id}_d & -\varrho^{-n}\ybf_1^n \\
        0 & 1
    \end{pmatrix}$ is bounded, independently of $n$. This   completes the proof.
\qedhere

\end{proof}

Let $B, \beta$ be as in \eqref{eq: def B beta}.
As in \cite{BQ1}, to simplify notation  we identify the index set $\{1, \ldots, k\}$ with the random walk elements $\{\bar h_1, \ldots, \bar h_k\}$. In particular, for $b \in B$ we let
$$b_1^n \df  \bar h_1^n \ \ \ \text{ and } \ \ b_n^1 \df \bar h_n^1
$$ be the elements given by \eqref{eq: weird notation from SW} corresponding to the $n$-prefix $(i_1, \ldots, i_n)$ of $b$. For $\s \in S$ we also let $\mathfrak{u}_\s$ denote the Lie algebra of the group $\mathbf{U}(\Q_\s)$ (see \eqref{eq: def Wst}). Let $\mathbb{P}(\mathfrak{g})$ be the projective space over $\mathfrak{g}$, let $d$ be some metric on $\mathbb{P}(\mathfrak{g})$ inducing the topology, and for two subsets $A, A' \subset \mathbb{P}(\mathfrak{g})$, we let  $\mathrm{dist}(A,A') \df \inf \{d(a,a'): a \in A, a' \in A'\}. $

  The following lemma provides a weaker version of the aforementioned results of~\cite{SimmonsWeiss} which suffices for our purposes (cf. \cite[Cor. 5.5]{BQ1} and \cite[Prop. 3.1]{SimmonsWeiss}).
\begin{lem}\label{lem:NormGrowthBound}
  For all $\delta>0$ there are $C>1$ and $m_{0}\in \NN$
  such that for all $\s \in \Sue$ and all non-zero $v_\s\in \mathfrak{g}_\s$ we have
  \begin{equation*}
    \beta\left(\big\{ b\in B \colon
        \forall m\geq m_{0}, \quad \lVert \Ad(b_1^m) \rVert_\s \lVert v_\s\rVert_\s
    \leq C \lVert \Ad(b_1^m) v_\s\rVert_\s \big\}\right) \geq 1-\delta.
  \end{equation*}
Moreover, for every $\delta>0$ and $\eta>0$, there is $m_{1} \in \N$ such that for all $\s \in \Sue$, for any $v \in \mathfrak{g}_\s\smallsetminus \{0\},$
    $$
    \beta \left( \left\{ b \in B: \forall m \geq m_{1}, \quad \mathrm{dist}\left(\Ad(b^m_1) v, \mathfrak{u}_\s \right) < \eta \right\} \right) > 1 - \delta
    $$
  (where we identify $v, \mathfrak{u}_\s$ with their image in $\mathbb{P}(\mathfrak{g})$).
\end{lem}

\begin{proof} This follows from \cite[Sections 3 \& 6]{SimmonsWeiss}. By the argument of \cite[Section 6.1]{SimmonsWeiss}, assumptions I, II, III are satisfied for the restriction of the random walk to $\mathfrak{g}_\s,$ with $W = \mathfrak{u}_\s.$ Now the desired estimates follow from \cite[Prop. 3.1]{SimmonsWeiss}. Note that in the reference \cite[\textsection 6]{SimmonsWeiss}, the condition that $V$ has no $G$-invariant vectors was omitted; this condition is valid in our setting for the Adjoint representation, since $G$ is simple.
Also note that in \cite{SimmonsWeiss},  $V$ is a real vector space, but the arguments given there are valid in vector spaces over $\Q_\s.$ Finally, note that the arguments in \cite{SimmonsWeiss} use the Oseledec theorem; for a p-adic version of the Oseledec theorem, see \cite{Raghunathan_Oseledec}. 

For a proof which is independent of \cite{SimmonsWeiss}, one could adapt the arguments given in the proof of Lemma \ref{lem:uniform growth in irreps} and Proposition \ref{prop:measures on proj spaces2} below.  
\end{proof}

\subsection{The effect of changing prefixes}
Finally, we record the following lemma which measures the effect of changing the prefix of a long random walk trajectory. We introduce the notation $\|\gamma_\sigma\|_{\mathrm{op}}$ to denote the operator norm of the linear operator $\gamma_\sigma$ acting linearly on $\Q_\sigma^{d+1}$ (this should not be confused with the notation $\|\gamma\|_\s$ used in Proposition \ref{prop:predictable growth} and Lemma \ref{lem:NormGrowthBound}).

\begin{lem}\label{lem: more linear2}
Given $a =(a_1, \ldots)$ and $ b = (b_1, \ldots)$ in $ B$, and given $n \in \N$, let
$$\gamma(a,n,b) \df
a_1^n \circ \left(b_n^1 \right)^{-1} =
\bar h_{a_n} \cdots \bar h_{a_1} \cdot \left(\bar h_{b_n}\right)^{-1} \cdots \left(\bar h_{b_1}\right)^{-1}.$$
Let $\s\in\Sue$ and
$$
G_{a,b, \s}(n) \df \left\| \gamma(a,n,b)_\s\right \|_{\mathrm{op}}\ \ \ \text{ and } \ \ G_{\max, b, \s}(n) \df \max_{a' \in B}G_{a',b}(n).
$$
Then
\begin{enumerate}
\item \label{item: 1}
 For $\s \notin \Sue$ we have $\gamma(a,n,b)_{\s} = \mathrm{Id}.$
\item \label{item: 2}
There is $c_1>1$  such that for any $\s \in \Sue$,
 any $b \in B$ and any $n \in \N,$
$$ \frac{1}{c_1} \leq  \frac{
G_{\max,b, \s}(n)}{\lvert \varrho\rvert_\s^{n}} \leq c_1.
$$
In particular
$
G_{\max,b, \s}(n) \stackrel{n \to \infty}{\longrightarrow} \infty.
$
\item \label{item: 3}
For any  $\alpha>0$ there is $c_2>0$  such that for all $b$ and all  $n \in \N$, there is a set $B_0 = B_0(b,n) \subset B$ with $\beta(B_0) \geq 1-\alpha$, such that for all $a \in B_0$  and any $\s \in \Sue$, we have
\begin{equation}\label{eq: uniform conclusion norm}
G_{a, b, \s}(n) \geq c_2 \, G_{\max, b, \s}(n).
\end{equation}
\end{enumerate}
\end{lem}
\begin{proof}
Assertion \eqref{item: 1} is clear from \eqref{eq: def h'}.
For $\sigma \in S_{\mathrm{ue}}$ write
 $(\bar h_i)_\sigma =  u(-\mathbf{y}_{i})g_0,$ where  $u : \Q_\sigma^d \to \mathbf{U}(\Q_\s)$ is the map as in \eqref{eq: def u+}, and
$
 g_0 \df \mathrm{diag}(\varrho, \ldots, \varrho, 1).
$
We have a commutation relation $g_0 \, u(\ybf)= u(\varrho \ybf)\, g_0.$
Carrying out a matrix multiplication, and using the commutation relation to move all the diagonal matrices to one side, we get that
\[\begin{split}\left(\gamma(a,n,b)\right)_\s  = & \left\{ \begin{matrix}   u(-\ybf_{a_m}) g_0 \cdots u(-\ybf_{a_1}) g_0 g_0^{-1} u(\ybf_{b_n}) \cdots g_0^{-1}u(\ybf_{b_1})    & \ \sigma \in S_{\mathrm{ue}} \\
\mathrm{Id} & \ \sigma \notin S_{\mathrm{ue}}\end{matrix} \right. \\= &
 \left\{ \begin{matrix} u(\ybf_{0})      & \sigma \in S_{\mathrm{ue}} \\
\mathrm{Id} & \ \sigma \notin S_{\mathrm{ue}},\end{matrix} \right.
 \end{split}\]
where $\ybf_0 = \ybf_0(a,n,b)$ is given by
\begin{equation}\label{eq: complicated expression for sigma}
\ybf_{0}= \sum_{i=1}^n \varrho^{i-1}\ybf_{b_i} - \sum_{i=1}^n \varrho^{n-i} \ybf_{a_i} = \sum_{j=1}^n \varrho^{n-j}(\ybf_{b_{n-j}} -\ybf_{a_j}).
\end{equation}

Since the $\ybf_i$ are contained in a finite set, the size of $G_{a,b, \s}(n)$ is
comparable to $\max(1,\|\ybf_0\|_\s)$,  i.e., to $|\varrho|_\s^r$ where $r =r(a,b,n)$ is the largest power $n-j$  appearing in \eqref{eq: complicated expression for sigma} with a nonzero
coefficient.
It is clear from \eqref{eq: complicated
expression for sigma} that $r(a,b,n) \leq n-1$ for all $a$. Also, for each
$b,n$, we can choose $a_1$ so that $\mathbf{y}_{b_{n-1}} - \mathbf{y}_{a_1}$ is nonzero, and
this is the coefficient of $\varrho^{n-1}$. This implies that for any sequence
starting with $a_1$ we have $r(a,b,n) =n-1$. This proves \eqref{item: 2}.

To prove \eqref{item: 3}, given $\alpha$, let $\ell$ be large enough so that any cylinder set in $B$ defined by specifying one prefix of length $\ell$ has $\beta$-measure less than $\alpha$; namely, we choose $\ell > \alpha / \log(\max_i p_i).$ Arguing as in the proof of \eqref{item: 2}, we see that the only way to have $r(a,b,n) < n -  \ell$ is to have $\tau_{a_j} = \tau_{b_{n-j}}$  for $j\leq \ell,$ and this means that  the first $\ell$ digits of $a$ are determined by the last $\ell$ digits of $b$. We define $B_0$ to be the complement of this prefix set of length $\ell$ corresponding to $b$, and the statement follows.
\end{proof}

\section{Preparations For Exponential Drift: Non-concentration of Limit Measures along Centralizer Orbits}
\label{sec:nonatomic}
Following \cite{BQ1}, the first key step in running the exponential drift argument is to show that, given a stationary measure $\nu$, the limit measures $\nu_b$ are non-atomic almost surely.
This property however fails for our random walks. Indeed, in case $\mu = \bar \mu$ the stationary measure $\nu = \hat \nu_0$ does have atomic limit measures, as can be seen from the proof of Proposition \ref{prop: examples stationary measures}. In the general case, it can be shown that the property fails due to the deterministic forward-contracting space $\mathfrak{u}_{\mathrm{dt}} = \mathrm{Lie}(\gUdt)$.
Nevertheless, in this section we will show that non-atomicity of limit measures (in a strong form) does hold under the hypotheses of Theorem~\ref{thm:NonAtomicityLimitMeasures} below.

In order to state the main result of this section, we introduce some notation. Let $B^X, \beta^X$ be as in
\eqref{eq: def beta X} and
Proposition \ref{prop: commutes with RW}.
Define $Z$ to be the subgroup of $G^S$ commuting with all of the $\bar h_i$.
By Proposition \ref{prop: projection Z trivial}, the elements $\left(z^{(\s)} \right)_{\s \in S}$  of $Z$ satisfy
$
z^{(\s)} = \mathrm{Id}$ if $\s \in S_{\mrm{ue}},$
$z^{(\s)} =\left(\begin{smallmatrix} A^{(\s)} & \mathbf{0} \\ \mathbf{0} & 1 \end{smallmatrix} \right)$   for some invertible $A^{(\s)}$ if $\s \in S_{\mrm{dt}}
$,
and with no restrictions on $z^{(\s)}$ for $\s \in S_{\mrm{tr}}.$
Let $\mathbf{P} \df \left(\begin{smallmatrix}
    * & * \\ \mathbf{0} & *
 \end{smallmatrix}\right)$
be the normalizer of the group $\mathbf{U}$ in \eqref{eq: def u+},
let $\Sdt$ and $\Str$ be as in \eqref{eq: def Sdt}, and let
\begin{equation}\label{eq: def Hne}  H^{\mathrm{fne}} \df \mathrm{P}_{\Sdt} \times  \mathrm{G}_{ S_{\mathrm{tr}}}.\end{equation}
The superscript `fne' stands for `forward non-expanding'; indeed, for the linear random walk consider in \textsection \ref{sec:GrowthPropertiesRandomWalk}, none of the vectors in the Lie algebra $\Vne$ of $\Hne$ expand under any of the elements $\bar h_i$.
Table~\ref{tab:groups and places} summarizes the structure of the groups $Z, \Hne,$ and $\Wst$.

\begin{table}[h]
    \centering
    \begin{tabular}{c|c|c|c}
           &  $\Sue$ & $\Sdt$ & $\Str$ \\
           \hline 
         $Z$ \rule[-2.5ex]{0pt}{6ex} & $\mathbf{1}$ & $\left(\begin{smallmatrix}
             \mrm{GL}_d & \mathbf{0}\\ \mathbf{0} & 1
         \end{smallmatrix} \right)$
         & $\gG_{\Str}$ \\ 
         \hline 
         $\Hne$ \rule[-2.5ex]{0pt}{6ex} & $\mathbf{1}$ & $\mrm{P}_{\Sdt} = \left(\begin{smallmatrix}
    * & * \\ \mathbf{0} & *
 \end{smallmatrix}\right)$ & $\gG_{\Str}$ \\
         \hline 
         $\Wst$ \rule[-2.5ex]{0pt}{6ex} & $\gUue$ & $\gG_{\Sdt}$ & $\gG_{\Str}$
         \\
         \hline 
    \end{tabular}
    \vspace{10pt}
    \caption{Projections of the groups $Z, \Hne,$ and $\Wst$ to $\gG_{S'}$, for $S'\in \{\Sue, \Sdt,\Str\}$. Here, $\mathbf{1}$ denotes the trivial group.}
    \label{tab:groups and places}
\end{table}

From \eqref{eq: def Wst}, we have
\begin{equation}\label{eq: perhaps missing something}
Z \subset H^{\mathrm{fne}} \subset W^{\mathrm{st}},
\end{equation}
and from \eqref{eq: def h'} that for all $i$,
\begin{equation}\label{eq: rw commutes with Hne}
    \bar h_i H^{\mathrm{fne}} \bar h_i^{-1} = H^{\mathrm{fne}}.
\end{equation}
In particular, each $\bar h_i$ normalizes each of the three groups in \eqref{eq: perhaps missing something}.

\begin{thm}
    \label{thm:NonAtomicityLimitMeasures}
    Let $\Phi, \bar h_i, \mathbf{p}$ and  $\bar \mu$ be as in Theorem \ref{thm: stationary main}.  Let $\nu$ be an ergodic $\bar \mu$-stationary measure.
 Suppose that $\nu(H^{\mathrm{fne}} x) =0$ for every $x \in \XS$.
    Then $\nu_b(Zx)=0$ for $\beta^X$-almost every $(b,x) \in B^X$.
\end{thm}

The remainder of this section is devoted to the proof of Theorem~\ref{thm:NonAtomicityLimitMeasures}.

\subsection{Prefix ergodic theorem}
We will need a pointwise ergodic theorem
specifically geared to the symbolic space. In order to state it we introduce some additional notation and terminology.
Let $B^*$ denote the set of finite words in the alphabet $\{1, \ldots, k\}$, and for $a \in B^*$ let $[a] \subset B$ denote the cylinder set of words in $B$ whose initial word is $a$.   A \emph{complete prefix set} is a finite subset $P\subset B^{\ast}$ such that
  \begin{equation*}
    \{[a] \colon a\in P\}
  \end{equation*}
  is a partition of $B$.
 Let $\mathrm{len}(a)$ denote the length of the word $a \in B^*,$ and for $a \in B^*$ and $b \in B$ we let $ab \in B$ be the infinite word obtained by concatenation.

\begin{thm}\label{thm:PrefixErgodicTheorem}
  Let $(P_{n})_{n\in\NN}$ be a sequence of complete prefix sets such that
  \begin{equation*}
    \min\{\mathrm{len}(a) \colon a\in P_n\} \xrightarrow{n \to
    \infty} \infty.
  \end{equation*}
  Then
  \begin{equation*}
    \forall f\in L^\infty(B,\beta), \quad
    \sum_{a\in P_n}f(ab)\beta([a]) \xrightarrow{n \to
    \infty} \int_{B} f\der\beta\quad
    \text{$\beta$-a.e.}
  \end{equation*}
\end{thm}

We give the proof of Theorem \ref{thm:PrefixErgodicTheorem} in \textsection \ref{subsec: prefix ergodic theorem}.
We will use the following useful consequence.
\begin{corollary}
  \label{cor:ChaconOrnsteinShells}
  Let $f\in L^\infty(B,\beta)$. Then
  \begin{equation*}
    \lim_{n\to\infty}\sum_{\mathrm{len}(a)=n}f(ab)\beta([a]) =
    \int_{B}f\der\beta(b)\quad \text{$\beta$-a.e.}
  \end{equation*}
\end{corollary}

\subsection{Consequences of concentration on the centralizer}

The following lemma provides a very useful consequence of the condition that $\nu_b(Zx)$ is positive for a positive measure set of pairs $(b,x)$.
\begin{lem}[{\cite[Prop.~7.8 and Lem.~7.9]{BQ2}}]
\label{lem:centralizer concentration}
     Let $\nu$ be an ergodic $\bar \mu$-stationary probability measure on $\XS$, and let $\nu_b$ be the system of limit measures as in \eqref{eq: def nu b}.
    Suppose that
    \begin{equation}\label{eq: suppose that limit measures}
    \beta^X \left( \left\{(b,x) \in B^X : \nu_b(Zx)>0 \right\} \right) >0.
    \end{equation}
     Then there exists a compact subgroup $Z_0\subset Z$ such that
    \begin{enumerate}[(i)]
   \item \label{item: measure is preserved}  $Z_0$ preserves the measure $\nu$;
   \item The $Z_0$-action on $\XS$ is free;
   \item \label{item:full measure set of positive measure Z_0 orbits} for a full measure set of $(b,x)$ in $B^X$ we have  $\nu_b(Z_0x)>0$;
        \item \label{item: invariance almost surely}
   $\nu_b$ is $Z_0$-invariant almost surely;
   \item \label{item: neighborhood exponential}
   there is a bounded neighborhood $W$ of $0$ in $\mathfrak{g}$ such that the Lie algebra $\mathfrak{z}_0$ of $Z_0 $ satisfies that the restriction of the exponential map to $W \cap \mathfrak{z}_0$ is well-defined and a  homeomorphism onto its image.
 \end{enumerate}
 \end{lem}
\begin{proof}
 First we show that the $Z$-action on $\XS$ is free, and hence so is the action of any subgroup of $Z$. Let $x \in \XS$ and let $Z(x)$ denote its stabilizer in $Z$. Since the $\Sue$-coordinates of $h_i$ and $\bar h_i$ agree,  Proposition \ref{prop: projection Z trivial} implies that the projection of elements of $Z(x)$ to the $\Sue$-coordinates is trivial,  and we have from Proposition \ref{prop: in repeated use} that $Z(x)$ is a trivial group.

    By \eqref{eq: suppose that limit measures}, and since $\nu$ is ergodic, we have $\nu_b(Zx)>0$  for $\b^X$-almost every pair $(b,x)$.
    Let $Z_1 \subset Z$ be the subgroup consisting of elements of $Z$ which preserve $\nu$, and let $E
    \subset \XS$ be the set of `typical points' for the random walk, in the sense of~\cite[Lemma 3.7]{BQ2}. Then $E$ is $Z_1$-invariant, $\nu(E)=1$ and $zE \cap E = \varnothing $ for $z \in Z \sm Z_1.$
    By the argument in~\cite[Proof of Prop.~7.8]{BQ2}, we get that $\nu_b(Z_1x)>0$ for $\beta^X$-almost every $(b,x)$. Since  $Z_1$ commutes with the random walk and preserves $\nu$, we have from \eqref{eq: def nu b} that $Z_1$ preserves almost every $\nu_b$.

    We now show that $Z_1$ is compact.
For this,
    note that for a.e.\,\,$b$, the $Z_1$-invariant measure $\nu_b$ is supported on countably many $Z_1$-orbits of positive measure. For any $x \in \XS$, let $Z_1(x)$ denote the stabilizer of $x$ in $Z_1$. As we have seen, $Z_1(x) = \{\mathrm{Id} \}.$ On the other hand, the orbit map $Z_1/ Z_1(x) \to Z_1x$ is a Borel isomorphism, and whenever $\nu_b(Z_1x)>0$, the finite $Z_1$-invariant measure $\nu_b|_{Z_1x}$ induces a finite $Z_1$-invariant measure on $Z_1/Z_1(x).$ This implies that there are $x \in \XS$ for which $ Z_1(x)$ is a lattice in $Z_1$.
      Thus $Z_1$ admits the trivial group as a lattice, so is compact.

     We now let  $Z_0 \subset Z_1$ be a subgroup  satisfying \eqref{item: neighborhood exponential}. To see that such a subgroup exists, see e.g. \cite[\textsection 3]{Ratnerpadic}. Since $Z_1$ is compact, $Z_0$ is a subgroup  of finite index, and we claim that it satisfies the required conclusions. Indeed, properties   \eqref{item: measure is preserved} and \eqref{item: invariance almost surely} hold for $Z_1$ and thus  hold for $Z_0$.  Moreover, since each $Z_1$-orbit is a finite union of $Z_0$-orbits, there is a positive measure subset of $(b,x)$ for which $\nu_b(Z_0x)>0$. Since $Z_0$ commutes with the random walk, the set of $(b,x)$ satisfying this property is invariant for the random walk, and by ergodicity, \eqref{item:full measure set of positive measure Z_0 orbits}  holds for $Z_0$.
\end{proof}

\subsection{Notation for the Proof of Theorem~\ref{thm:NonAtomicityLimitMeasures}}\label{subsec: notation 7.1}
In the proof we will argue by contradiction, we now introduce some notation that arises when assuming that  \eqref{eq: suppose that limit measures} holds.
   Let $Z_0 \subset Z$ be a subgroup satisfying the conclusions of   Lemma~\ref{lem:centralizer concentration}.
    Then, since $Z_0 \subset Z$, the random walk acts on the quotient space
    \begin{align*}
        X'\df\rquotient{Z_0}{\XS} .
    \end{align*}
Since $Z_0$ is compact, the quotient topology on $X'$ is Hausdorff, locally compact, and second countable, and there is a Borel section $\tau: X' \to \XS$; that is, $\tau$ satisfies $\mathrm{Id}_{X'} =\pi_{X'} \circ \tau,$ where $\pi_{X'} : \XS \to X'$ is the projection.

     We fix an $\Ad(Z_0)$-invariant norm on each $\lieg_\s$,  and use it to define a metric $\dist_{G^S}$ on $G^S$ which is both right-invariant and left $Z_0$-invariant.  For the real place this involves defining a suitable Riemannian metric on $\mathrm{G}_{\R}$, for  finite places the construction is explained in \cite[\textsection 3]{Renethesis}, and for a general vector $v = \sum_\s v^{(\s)}$ with $v^{(\s)} \in \lieg_\s$, the norm on $\lieg$ is given by $\|v\| = \max_\s \|v^{(\s)}\|_\s$ where $\| \cdot \|_\s$ is the norm on $\lieg_\s$. The metric $\mrm{dist}_{G^S}$ induces a metric $\dist_{\XS}$ on $\XS$, and we use it to define
     $$
     \dist_{X'}\left(x'_1, x'_2\right) = \inf\left\{\dist_{\XS}(y_1, y_2) : y_i \in Z_0 x_i, \ i=1,2\right\}, \ \ \text{ where } x'_i = Z_0 x_i.
     $$
     Since $Z_0$ is compact and since $\dist_{\XS}$ is $Z_0$-invariant, this can also be written as
     $$
     \dist_{X'}\left(x'_1, x'_2\right) = \min\left\{\dist_{\XS}( x_1, z_0 x_2) : z_0 \in Z_0 \right\}.
     $$
     In the sequel, when discussing  balls and distances between points in the spaces $G^S, \XS, X',$ or norms of vectors in $\lieg$, we will always have in mind the metrics arising from this norm on $\lieg.$ When  confusion is unavoidable  we will simplify notation by omitting subscripts, writing `dist' for each of these metrics.

    The pushforward $\nu'$ of $\nu$ to the quotient $X'$ is an ergodic stationary measure with the property that the limit measures $\nu'_b$ have atoms for almost every $b$.
By ergodicity of $\nu'$ (see~\cite[Lemma~3.11]{BQ1} for a similar argument),
 for almost every $b\in B$, the limit measure $\nu'_b$ is in fact a purely atomic uniform measure supported by
 a finite subset of $X'$ depending on $b$.
 Moreover, the cardinality of
 the support of $\nu'_b$ is constant almost surely and we denote it by $N_0$.

 Let $X_0$ denote the
 collection of subsets $\Sigma\subset X'$ of cardinality $N_{0}$.
 Given two elements $\Sigma_1,\Sigma_2\in X_0$, we define a metric on $X_0$ by
 \begin{equation*}
   \dist_{X_0}(\Sigma_1,\Sigma_2) = \max_{x_1 \in \Sigma_1 } \min_{x_2 \in \Sigma_2}  \dist(x_1, x_2) + \max_{x_2 \in \Sigma_2 } \min_{x_1 \in \Sigma_1}  \dist(x_1, x_2).
 \end{equation*}
 The diagonal action of the group elements of the random walk  on the product space $(X')^{N_0}$ induces an action on $X_0$.
 Similarly, since the group $\Hne$ is normalized by $Z_0$, its left multiplication action on $\XS$ induces an action on the space $X'$ and hence on the space $X_0$.

Let $\mathfrak{z}, \, \lieg_{\mrm{dt}}, \, \lieg_{\mrm{ue}} $ and $\Vne$ denote respectively the Lie algebras of $Z, \mathrm{G}_{\Sdt}, \mathrm{G}_{\Sue}$ and $H^{\mathrm{fne}}.$ Let $\mathfrak{u}_{\mathrm{dt}}^- \subset \lieg_{\mrm{dt}}$ be the Lie algebra contracted by the restriction of the $\Ad(\bar h^{-1}_i)$ to $\mathfrak{g}_{\mathrm{dt}},$
and denote
\begin{equation*}
  V_{\mrm{ex}} \df  \lieu_{\mathrm{dt}}^- \oplus \lieg_{\mathrm{ue}}.
\end{equation*}
The letters `ex' stand for `expanding'; indeed, under the linear random walk, all the nonzero vectors in
$\mathfrak{u}_{\mathrm{dt}}^-$
 expand exponentially, and the vectors in $\bigoplus_{\sigma \in S_{\mathrm{ue}}} \mathfrak{g}_\s$ expand for most infinite random walk paths by the results of \textsection \ref{sec:GrowthPropertiesRandomWalk}. As suggested by this terminology, $V_{\mrm{ex}}$ is a complementary subspace to $\Vne.$

Let  $\mathfrak{z}_0$ be the Lie algebra of $Z_0$, let $\mathfrak{w}$ be an $\Ad(Z_0)$-invariant complementary subspace to $\mathfrak{z}_0$ inside $\mathfrak{z} $, and define
$\mathfrak{z}_0^\perp \subset \lieg$ by
\begin{align*}
    \mathfrak{z}_0^\perp \df \mathfrak{w} \oplus \lieu_{\mrm{dt}}\oplus V_{\mathrm{ex}}.
\end{align*}
Then $\mathfrak{z}_0^\perp $  is a complementary subspace to $\mathfrak{z}_0$
 in $\mathfrak{g}$, and the subspaces $\Vne, \mathfrak{z}_0^{\perp}, \mathfrak{w}, \mathfrak{u}_{\mathrm{dt}}, \mathfrak{u}^-_{\mathrm{dt}}, V_{\mathrm{ex}}$ and $\mathfrak{z}_0$ are all  invariant under the linear random walk.

Every vector $v_{\mrm{dt}}\in \lieg_{\mrm{dt}}$ can be written uniquely as   $v_{\mrm{dt}}=  v_{\mrm{dt ,fne}}
+v_{\mrm{dt, ex}}
$, where
$v_{\mrm{dt, fne}} \in \mathfrak{g}_{\mrm{dt}} \cap \Vne$ and $v_{\mrm{dt, ex}} \in \mathfrak{g}_{\mrm{dt}} \cap V_{\mrm{ex}}$.
Given a vector $v\in \lieg$, we define $v_{\mrm{ex}}  = \left( v_{\mrm{ex}}^{(\sigma)}\right)_{\s \in S} \in \lieg$ by
\begin{equation*}
  v_{\mrm{ex}}^{(\sigma)} \df \begin{cases}
    v_{\mrm{dt,ex}} & \text{if $\sigma \in \Sdt $}, \\
    v^{(\sigma)} & \text{otherwise},
  \end{cases}
\end{equation*}
and define
$$v_{\mrm{fne}} = v - v_\mrm{ex}.$$

From the definition of the norm, we have that
for any $v \in \Vne $ and any $i$, $\left\|\Ad\left(\bar h_i \right)v \right \| \leq \|v\|.$
Also, since $V_{\mrm{ex}}$ is invariant under $\Ad(Z_0)$,
after adjusting the inner product defining $\| \cdot \|$, we may also assume that
for any $v\in \lieg$  we have
$  \lVert v \rVert \leq \lVert v_\mrm{ex} \rVert +
  \lVert v_{\mrm{fne}} \rVert.
$

\subsection{Proof of Theorem \ref{thm:NonAtomicityLimitMeasures}}
We first give an overview of the proof. We will assume \eqref{eq: suppose that limit measures} and obtain a contradiction.
 Let  $X', \, N_0$ and $X_0$ be as in \textsection \ref{subsec: notation 7.1}.
 Define
 \begin{equation}\label{eq: def kappa}
  \kappa\colon B\to X_0, \ \ \  \ \ \kappa(b) \df \supp \nu'_b
 \end{equation}
 (more precisely, the right hand side of \eqref{eq: def kappa}  is a well-defined measurable map on a  subset of $B$ of full measure, but we will ignore nullsets and continue to denote this subset by $B$).
 Let $\Delta \subset X_0^2$ denote the diagonal. Since $\nu(H^{\mathrm{fne}}x)=0$ for every $x$ and by \eqref{eq: perhaps missing something}, $\nu'$ is not a Dirac mass on $X_0$, but the random walk pushes it toward the Dirac mass $\nu'_b$ supported on $\kappa(b)$. In particular, for a.e.\,$b$,  off-diagonal points in $X_0^2$ get pushed toward $\Delta$ by $b_n^1$.

 We will show that on a certain neighborhood $\mathcal{U}$ of a compact subset of $\Delta$, the  action of the random walk on $X_0^2$ is essentially given by $\Ad \oplus \Ad$ on $\lieg^2.$
 This step is made somewhat complicated by the fact that we have to take a quotient by the action of the compact group $Z_0$. The reader may wish to first consider the simpler case in which $Z_0$ is trivial.

 Given this relation between the random walk on $X_0^2$, and the adjoint action, we recall from \textsection \ref{sec:GrowthPropertiesRandomWalk} that vectors tend to grow under most elements of the random walk.
 Using this, we  will show that for many random walk paths of controlled length, points in $\mathcal{U} \sm \Delta$ get pushed away from $\Delta,$ which contradicts convergence to the diagonal.
  A complication in the argument is  that the linear coordinates on $\mathcal{U}$ need to make sense on a large enough set of words so that expansion can be exploited.
The
 prefix ergodic theorem will be useful for dealing with this issue.

 We proceed to the details.
 In order to make the logic more transparent, we will break up the argument into steps.

 \medskip

 {\bf Step 1. Setting up constants, defining the neighborhood $\mathcal{U}$, and formulating the goal.}
Let $T: B \to B$ be the left-shift. By Proposition \ref{prop: commutes with RW},  and since $B^{\ast}$ is countable, for $\beta$-a.e.~$b \in B$ we have
  \begin{equation}\label{eq: satisfies the properties}
  \kappa(Tb) = b_1^{-1}\cdot \kappa(b), \ \ \ \text{ and }
\kappa(ab)
 = a^{1}_n\cdot\kappa(b) \text{ for any } a
 \in B^{\ast}, \ \ \text{ where } n= \mathrm{len}(a)
 .
 \end{equation}

 Let $\varepsilon \in \left(0, \frac16 \right)$.
 By Lusin's theorem, we can find a compact set $K_1\subset B$ such that
 $\kappa|_{K_1}$ is continuous, satisfies properties \eqref{eq: satisfies the properties}, and such that $\beta(K_1) > 1-\varepsilon$.
 Let $f = \mathbf{1}_{K_1}$. Given $n\in
 \bN$, define the function $f_n \colon B\to [0,1]$ by
 \begin{equation*}
   f_n(b) = \sum_{a \in (\supp \bar{\mu})^{n}}f(ab)\beta\big([a]\big).
 \end{equation*}
 By Corollary~\ref{cor:ChaconOrnsteinShells}, there exists
 $n_0 \in \bN$ such that the set
 \begin{equation*} \beta\big(E(n_0)\big) > 1 - \frac{\varepsilon}{2}, \ \ \ \text{ where } \ \
   E(n_0) = \left\{ b \in B \colon \forall n\geq n_0, \quad f_{n}(b) > 1 -
   2\varepsilon \right\}.
 \end{equation*}
  Hence there
 exists a compact set
 \begin{align*}\label{eq:K_2}
    K_2 \subset E(n_0)    \quad \text{ such that } \quad
     \b(K_2) > 1-\varepsilon.
 \end{align*}

 We define
 \begin{equation*}
   K_3 \df \kappa(K_1),
\end{equation*}
a compact subset of $X_0.$
  Recall that the {\em injectivity radius at $x \in \XS$} is the maximal $r$ such that the restriction of the map $G^S \to \XS, \ g \mapsto gx$ to the open ball around the identity of radius $r$, is injective.
 Given $\Sigma\in X_0$, i.e., a collection of $N_0$ orbits for the group $Z_0$, we use the same letter $\Sigma$ to denote the subset of $\XS$ comprised by these orbits, and denote by  $r_\Sigma$ and $d_\Sigma$ respectively  the minimal injectivity
 radius of a point in $\Sigma$, and the minimum of the
 pairwise distances between elements of $\Sigma$. Both of the numbers $r_{\Sigma}, \, d_{\Sigma}$ depend
 continuously on $\Sigma$.
  Hence  the numbers
 \begin{equation}\label{eq:d(K_3)}
   r(K_3) \df  \inf\{r_\Sigma \colon \Sigma \in K_3\}  \ \ \ \text{ and } \ \ \  d(K_3) \df \inf\{d_\Sigma \colon \Sigma \in K_3\}
 \end{equation}
 are both positive. Choose $\iota>0$ small enough so that
 \begin{equation*}
   \iota <  \min \left \{r(K_3),d(K_3)
   ,1 \right \},
 \end{equation*}
and so that
 there is a  neighborhood $W$ of $0$ in $\lieg$ such that
 \begin{equation}
 \label{eq: def W iota}
 \exp|_W: W \to B(\mathrm{Id}, \iota) \subset G^S
 \end{equation}
 is well-defined and is a homeomorphism.  Let $C_W >1$ be a bi-Lipschitz constant for  $\exp|_W$, that is,
 $$
 \forall w_1, w_2 \in W, \quad \quad \frac{ \dist(\exp(w_1, w_2))}{C_W} \leq \|w_1 - w_2 \| \leq  C_W \, \dist(\exp(w_1, w_2));
 $$
 the fact that $\exp$ is locally bi-Lipschitz follows from the construction of the metric $\dist.$
By  item \eqref{item: neighborhood exponential} of Lemma \ref{lem:centralizer concentration}, by making $W$ smaller we can also assume that the  map
 $\exp|_{W \cap \mathfrak{z}_0} $
is a  homeomorphism onto its image, which is open in $Z_0$. By making $W$ and $\iota$ even smaller we can find an open subset $W' \subset W$ containing $0$, and a constant $C_{W'}>1, $ such that the two maps $(W' \cap \mathfrak{z}_0) \times (W' \cap \mathfrak{z}_0^\perp) \to G^S$,
 \begin{equation}\label{eq: the two maps}
 (z_0, z_0^\perp) \mapsto \exp(z_0)\exp(z_0^\perp), \ \ \ \text{ and } (z_0, z_0^\perp) \mapsto \exp(z_0^\perp)\exp(z_0)
 \end{equation}
 are both bi-Lipschitz homeomorphisms onto their image, and this image contains $B(\mathrm{Id}, \iota)$ and is contained in  $\exp(W\cap \mathfrak{z}_0 )\exp(W\cap \mathfrak{z}^\perp_0 ) \cap \exp(W\cap \mathfrak{z}^\perp_0 )\exp(W\cap \mathfrak{z}_0)$. Finally, by making $W$ even smaller, and using the fact that the $Z_0$-action on $\XS$ is free, we can assume that if  $x_1 \in \exp(W)K_3$ and $x_2 =z_0 x_1$ for $z_0 \in Z_0$, then
 $$
 \frac{1}{C_W} \, \dist_{\XS}(x_1, x_2) \leq \dist_{G^S}(\mathrm{Id}, z_0) \leq C_W \,\dist_{\XS}(x_1, x_2).
 $$

 Let
 \begin{equation}\label{eq: setting delta}
 \delta \df \frac{\varepsilon}{N_0 \,|S|},
 \end{equation}
 and let $m_0  \in \N
 $ and $C >1
$ be constants
 (depending on $\delta$) for which the conclusions of  Proposition \ref{prop:predictable growth} and Lemma~\ref{lem:NormGrowthBound} hold.

For $\s \in S$ we let
$  \lambda_\s \df |\log |\varrho|_\s|$,  and set
$
     \l_\mrm{max} \df \max \seti{\lambda_\s
     :
      \s\in S
     }.
     $
We choose $L_1$ satisfying
\begin{equation}\label{eq: choice of L1}
 L_1 > \frac{C
 \, C_W \, |S|}{\iota}.
 \end{equation}
 Now we choose
     $L_2$
      large enough so that   for all $\s \in S,$
     \begin{equation}\label{eq: L large enough 3} (m_0+ n_0) \log(L_2) - \frac{1}{\lambda_{\s}} \log(
     C_W \, L_1) -1 \geq m_0+n_0,
     \end{equation}
and so that
     if we define
     \begin{equation}\label{eq:r and L}
   r \df \iota \cdot L_2^{-(m_0 + n_0)\l_{\mrm{max}} },
 \end{equation}
  then
     \begin{equation}\label{eq: L large enough 2}
    r < \frac{1}{C^2\,
    C_W \, e^{\lambda_{\max}} \, L_1}.
     \end{equation}

 We  define
 \begin{equation*}
\Delta(K_3) \df \{(\Sigma,\Sigma) \colon \Sigma\in K_3\}.
 \end{equation*}
 Finally we
 define $\Ucal$  to be the $r$-neighborhood of $\Delta(K_3)$.
 Using uniform continuity of $\kappa|_{K_1}$, let $n_1\in \NN$ be such that for all sequences $b,b'\in K_1$ which agree on a prefix of length at least $n_1$, we have
 $\big(\kappa(b),\kappa(b')\big) \in \mathcal{U}$.

Our goal is to find finite words $\tilde{a}, a \in B^*$  and $b, \bar b \in B,$  such that the following hold:
\begin{enumerate}[(I)]
    \item \label{item:in K_1} $\tilde{a}ab \in K_1$ and  $\tilde{a}a\bar b \in K_1$.
    \item \label{item: word long enough} $\mathrm{len}(a) \geq n_1.$
    \item \label{item:outside U} The pair $(\kappa(\tilde{a}ab),\kappa(\tilde{a}a\bar b))$ is outside $\Ucal$.
\end{enumerate}
To see that this gives a contradiction, note that items~\eqref{item:in K_1} and~\eqref{item: word long enough} and the definition of $n_1$ imply that $(\kappa(\tilde{a}ab),\kappa(\tilde{a}a\bar b)) \in \Ucal$.
This contradicts item~\eqref{item:outside U}.

\medskip

{\bf Step 2. Linearizing the action of the random walk near the diagonal. }
By definition of $\iota$,   we have that for any pair $\mathbf{\Sigma} =
 (\Sigma_1,\Sigma_2)$ for which there is $\Sigma \in K_3$ with $\dist(\Sigma_i, \Sigma) < \iota/2$ for $i=1,2$, for any $x_1\in \Sigma_1$ there is
 exactly one element $x_2 \in \Sigma_2$ such that $\dist(x_1,x_2) <
 r$. In particular this holds for $\mathbf{\Sigma} \in \Ucal$. We denote this element $x_2$ by $\varphi_{\mathbf{\Sigma}}(x_1)$, so that
 $\varphi_{\mathbf{\Sigma}} \colon \Sigma_1 \to \Sigma_2$ is a bijection.
 It is easy to see that the map
 $\varphi_{\mathbf{\Sigma}}$ depends continuously on $\mathbf{\Sigma}.$

 We would like to estimate the displacement $\dist(x', \varphi_{\mathbf{\Sigma}}(x'))$ in terms of the adjoint action on $\lieg$.
To this end, note that if $x' \in \Sigma_1$ then we can write $x' = Z_0x$ for some $x \in \XS$, and by the choice of $W$, there is a unique $\tilde{v}  = \tilde{v} (x) \in \mathfrak{z}_0^\perp \cap W$ such that $\varphi_{\mathbf{\Sigma}}(x') = Z_0 \exp(\tilde{v}) x. $
 As suggested by the notation, this choice of $\tilde{v}$ depends on the choice of $x \in \pi_{X'}^{-1}(x')$.
 However, we have
 that if $x'= Z_0x_1 = Z_0x_2$ then $\|\tilde{v}(x_1)\| = \|\tilde{v}(x_2)\|$; indeed, if $x_2=z_0x_1$ for some  $z_0 \in Z_0$ then by the fact that $\mathfrak{z}_0^\perp$ is $\Ad(Z_0)$-invariant we see that $\tilde{v}(x_2) = \Ad(z_0) \tilde{v}(x_1)$, and $\|\tilde{v}(x_1)\| = \|\tilde{v}(x_2)\|$ since the norm is $\Ad(Z_0)$-invariant.
Using the section $\tau: X' \to \XS,$ we define
\begin{equation}\label{eq:v_Sigma}
v_{\mathbf{\Sigma}}\colon \Sigma_1 \to \mathfrak{z}_0^{\perp},
\qquad
v_{\mathbf{\Sigma}}(x') \df  \tilde{v}(\tau(x')), \end{equation}
and we have that  $\|v_{\mathbf{\Sigma}}(x')\|$ does not depend on the choice of the section.

 Note further that
 \begin{align}\label{eq:dist and exponential map}
\forall \mathbf{\Sigma} \in \Ucal,\ \ \forall x' \in \Sigma_1, \ \ \lVert v_{\mathbf{\Sigma}}(x') \rVert \leq   C_W \, r.
 \end{align}
Using the fact that $\mathfrak{z}_0^\perp$ is invariant under the adjoint action of the random walk, the reader can now verify the following statement:

 \medskip

 Suppose $n \in \N, \ \bar h = \bar h_{i_1} \circ \cdots \circ \bar h_{i_n}$ and $\mathbf{\Sigma} = (\Sigma_1, \Sigma_2) \in \Ucal$
satisfy that $\bar h(\Sigma_1) \in K_3$ and
\begin{equation}\label{item: for some x'}
\|\Ad(\bar h) v_{\mathbf{\Sigma}}(x')\| < \frac{\iota}{ C_W} \ \ \ \ \ \text{ for all } x' \in \Sigma_1.
\end{equation}
Then
\begin{equation}\label{eq: comparison inequality}
\frac{\dist\left(\bar h(\Sigma_1), \bar h(\Sigma_2) \right)}{ C_W } \leq \max_{x' \in \Sigma_1} \| \Ad(\bar h)v_{\mathbf{\Sigma}}(x') \| \leq C_W \, \dist \left(\bar h(\Sigma_1), \bar h(\Sigma_2)\right).
\end{equation}

\medskip

{\bf Step 3. Choosing $a, b, \bar b$.}
By \eqref{eq: perhaps missing something} we have that $Z_0 \subset \Hne$ and thus the partition of $X$ into $\Hne$-orbits induces a well-defined partition of $X'$, which we will continue to refer to as $\Hne$-orbits and denote by $\Hne x'$ (although $\Hne$  might not act on $X'$).
Our assumption is that these $\Hne$-orbits are of zero measure with respect to $\nu'$.
On the other hand, the supports $\kappa(b)$ of the limit measures $\nu'_b$ are finite sets of $Z_0$-orbits, and again by \eqref{eq: perhaps missing something}, if $\kappa(b)$ intersects an $\Hne$-orbit, the measure $\nu'_b$ assigns this orbit  positive measure. This implies via Proposition \ref{prop: commutes with RW} that for any fixed $b \in B$, for any $x' \in \kappa(b)$ and for $\beta$-a.e. $\bar b$ we have
\begin{equation}\label{eq:H- NonAlignment}
   \kappa(\bar b) \cap \Hne x' = \varnothing.
   \end{equation}
Hence we can find
 $b, \bar b \in B$ such that
  for all $x' \in \kappa(b)$, we have
 \eqref{eq:H- NonAlignment}, and
 the points $b, \bar b$ are generic for $\mathbf{1}_{K_1 \cap K_2}$ in the sense of the prefix ergodic theorem. The latter condition means that for both $c =b$ and $c=\bar b$, we have
\begin{equation}\label{eq:ChaconOrnsteinCondition}
   \lim_{n\to \infty}\sum_{a\in(\supp\bar{\mu})^{n}}\mathbf{1}_{K_1 \cap
   K_2}(ac) \beta([a]) = \beta( K_1 \cap K_2).
   \end{equation}
By~\eqref{eq:ChaconOrnsteinCondition}, since $K_1$ and $K_2$ both have measure at least $1-\vre$,
 there is $k \geq n_1$ such that
 \begin{equation}\label{eq:applyChaconOrnstein}
   \sum_{a \in (\supp \bar{\mu})^{k}}\mathbf{1}_{K_1 \cap K_2}(ac)\beta([a]) >
   1-3\varepsilon\quad(c=b,\bar b).
 \end{equation}
 Since $1-3\varepsilon > 1/2$, we can find $a\in (\supp \bar{\mu})^k$
 such that
 \begin{align}\label{eq:landing in K_1 cap K_2}
     ab, a\bar b\in K_1 \cap K_2.
 \end{align}
By \eqref{eq: rw commutes with Hne} and \eqref{eq:H- NonAlignment}
we still have
 \begin{equation}\label{eq: continues to hold}
  \text{ for all } x'\in
 \kappa(ab), \ \ \ \   \kappa(a\bar b) \cap \Hne x' = \varnothing.
 \end{equation}
The choice of $k$ ensures that the word $a$ satisfies  \eqref{item: word long enough}.

\medskip

{\bf Step 4. Choosing the word $\tilde{a}.$}
We now choose $\tilde{a}$. Let
\begin{equation}\label{eq: setting Sigma}
  \mathbf{\Sigma} = (\Sigma_1, \Sigma_2) \ \ \ \text{ where } \ \   \Sigma_1 \df \kappa(ab), \quad \Sigma_2 \df \kappa(a\bar b).
\end{equation}
By \eqref{eq:landing in K_1 cap K_2} and the definition of $n_1$, we have that $\mathbf{\Sigma} \in \Ucal.$
 It follows from \eqref{eq: perhaps missing something} and \eqref{eq: continues to hold} that for any $x'\in \Sigma_1$,
 $v_{\mathbf{\Sigma}}(x') \not \in \Vne$. In the notation introduced in \textsection \ref{subsec: notation 7.1} above, this means that $v_{\mathbf{\Sigma}}(x')_{\mrm{ex}} \neq 0.$
 Given $\sigma\in S$, and $x' \in \Sigma_1,$ let
 \begin{equation}\label{eq: def alpha sigma}
  \alpha_\sigma(x') \df
     \big\lVert v_{\mathbf{\Sigma}}(x')_{\mrm{ex}} \big\rVert_{\sigma}
    \ \ \ \ \text{ and } \ \ \ S_{\mrm{good}}(x') \df \{\sigma \in S \colon \alpha_\sigma (x') > 0\}.
 \end{equation}
Now set
 \begin{equation*}
     n(x') \df \min_{\s \in S_{\mrm{good}}(x')} \left\lfloor   \frac{1}{\lambda_\s} \left(\log \left(\frac{1}{\a_\s(x')}\right) - \log (L_1) \right)\right\rfloor.
 \end{equation*}
 This choice implies that for all $\s\in S_{\mrm{good}}(x')$, we have
 \begin{align}\label{eq:choice of n1}
     e^{\l_\s n(x')} \a_\s (x')\leq  \frac{1}{L_1},
     \end{align}
     and there is $\s\in S_{
     \mathrm{good}}(x')$ (the one for which the minimum is attained) satisfying
     \begin{align}\label{eq:choice of n2}
     \frac{1}{\lambda_{\s}} \log\left(\frac{1}{\alpha_\s(x')} \right) \leq n(x')+1 + \frac{1}{\lambda_\s} \log(L_1) \ \ \ \ \text{ and } \ \ e^{\lambda_{\s}n(x')} \alpha_\s(x') \geq \frac{1}{e^{\lambda_\s}L_1}.
 \end{align}
 Our choices ensure
\begin{equation}\label{eq:n is large enough}
    n (x') \geq m_0 +n_0.
\end{equation}
Indeed, since $\mathbf{\Sigma}\in \Ucal$, for all $\s$ we have
\begin{align*}
  \frac{1}{\lambda_\s} \left(\log \left(\frac{1}{\alpha_\s(x')} \right) + \log (C_W) \right)
  \stackrel{\eqref{eq:dist and exponential map} \& \eqref{eq: def alpha sigma}}{\geq} -\frac{1}{\l_\s} \log r \stackrel{\eqref{eq:r and L}}{\geq} (m_0 +n_0)\log L_2.
\end{align*}
Now~\eqref{eq:n is large enough} follows by \eqref{eq: L large enough 3} and \eqref{eq:choice of n2}.

Let
\begin{equation}\label{eq: def n}
n\df \min_{x' \in \Sigma_1} n(x').
\end{equation}
With this choice of $n$, let
\begin{equation*}
    P_n \df \{\tilde{a} \in (\supp\bar\mu)^n \colon  \text{ for } c = b, \bar b, \quad
   \tilde{a}ac \in K_1\}, \qquad  \mathcal{P}_n \df \bigcup_{\tilde{a}\in P_n}[\tilde{a}].
\end{equation*}
Since  $ab, a\bar b \in K_2$ by~\eqref{eq:landing in K_1 cap K_2}, and $n\geq n_0$ by~\eqref{eq:n is large enough} and \eqref{eq: def n}, the definition of $n_0$ ensures
\begin{equation}\label{eq:GoodPartitionElementsMeasure}
    \b(\mathcal{P}_n) \geq 1-2\varepsilon.
 \end{equation}

 Next
 define
 \begin{align*}
     \Xi_n(x') \df \seti{
     (\tilde{a}_1,\dots, \tilde{a}_n)\in (\supp\bar\mu)^n:  \forall \sigma \in S_{\mrm{good}}(x'), \
     \lVert \tilde{a}_n^1 \rVert_\sigma \lVert v_{\mathbf{\Sigma}}(x')_{\mrm{ex}}
     \rVert_\sigma \leq C \lVert \tilde{a}^1_n \cdot v_{\mathbf{\Sigma}}(x')_{\mrm{ex}} \rVert_\sigma}.
 \end{align*}
Note that the Bernoulli measure has a symmetry property $\beta([a_1^n]) = \beta([a_n^1]).$
 Thus using   Lemma~\ref{lem:NormGrowthBound}, and since  $n\geq m_0$,  we have
 $$
 \beta(\Xi_n(x'))\geq 1-\delta \ \ \text{ for each } x' \in \Sigma_1.$$
 The choice \eqref{eq: setting delta} now implies that $\beta\big(\bigcap_{x'\in\Sigma_1} \Xi_n(x'))\big) > 1-\varepsilon$.
 Combining with~\eqref{eq:GoodPartitionElementsMeasure} we have
 \begin{align*}
     \b \left( P_n \cap \bigcap_{ x'\in \Sigma_1} \Xi_n(x') \right)
     \geq 1- 3 \varepsilon >0.
 \end{align*}
 In particular, the above intersection is non-empty.
 Fix a word $\tilde{a}$
 in this intersection.
The definition of $P_n$ now ensures that property \eqref{item:in K_1} holds.

\medskip

{\bf Step 5. Verifying  property \eqref{item:outside U}.}
We need to show that for $\mathbf{\Sigma} = (\Sigma_1, \Sigma_2)$ as in \eqref{eq: setting Sigma} we have
 $\dist\big( \tilde{a}_n^{1} \Sigma_1 , \tilde{a}_n^{1} \Sigma_2 \big) \geq r$. For this it suffices to check condition \eqref{item: for some x'}, and show that for some $x' \in \Sigma_1,$
\begin{equation}\label{eq: suffice conditions}
\left\|\Ad\left(\tilde{a}_n^1\right)v_{\mathbf{\Sigma}}(x') \right \| \geq  C_W \, r.
\end{equation}
For \eqref{item: for some x'}, we note that for any $x' \in \Sigma_1$, the non-expanding coordinates of $v_{\mathbf{\Sigma}}(x')$ are of small norm since $\mathbf{\Sigma} \in \Ucal$. For $\s \in S_{\mrm{good}}(x')$ we have
\begin{equation*}
   \left \|\Ad(\bar h) \right\|_{\mathrm{op}}\alpha_\s(x')  \stackrel{\eqref{eq: norm estimate we want}}{\leq}  C \, e^{\lambda_\s n} \alpha_\s(x') \stackrel{\eqref{eq:choice of n1} \& \eqref{eq: def n}}{\leq}  \frac{C}{L_1} \stackrel{\eqref{eq: choice of L1}}{\leq} \frac{\iota}{C_W}.
\end{equation*}

Let $x'\in \Sigma_1$ be the element for which the minimum in \eqref{eq: def n} is attained, and let $\s \in S_{\mrm{good}}(x')$ be the place for which \eqref{eq:choice of n2} holds.
Since $\tilde{a}\in \Xi_n(x')$, we get that
\begin{align*}
      \norm{\Ad(\tilde{a}_n^1) v_{\mathbf{\Sigma}}(x')}
    &\geq \norm{\Ad(\tilde{a}^1_n)
         v_{\mathbf{\Sigma}}(x')_\s}
    \geq \frac{1}{C} \norm{\Ad(\tilde{a}_1^n)}_\s
    \norm{    v_{\mathbf{\Sigma}}(x_1)_\s}_\s
    \\
    & \stackrel{\eqref{eq: norm estimate we want}}{\geq}
    \frac{1}{C^2} \, e^{\l_\s n} \alpha_\s(x') \stackrel{\eqref{eq:choice of n2}}{\geq}
    \frac{1}{e^{\lambda_{\max}}C^2L_1}\stackrel{\eqref{eq: L large enough 2}}{\geq} C_W r.
\end{align*}
This proves \eqref{eq: suffice conditions} and completes the proof.
\qed

\subsection{Proof of the prefix ergodic theorem}
\label{subsec: prefix ergodic theorem}
In this subsection we prove Theorem \ref{thm:PrefixErgodicTheorem}.
For a bounded function $f$ on $B$ and a $A\subset B$, we define the variation of $f$ on $A$ by
$$
\mathrm{Var}
(f, A) \df \sup_{x \in A} f(x) - \inf_{x \in A} f(x).$$
Clearly for any measurable function $f$, any set $A$ with $\beta(A)>0$, and any $x_0 \in A$, we have
$$\left|  f(x_0) - \frac{1}{\beta(A)} \int_A f \der\beta \right| \leq \mathrm{Var}(f, A).$$
Equip $B$ with its standard ultra-metric, and
suppose first that $f$ is continuous.
Then, since $B$ is compact, $f$ is uniformly continuous. Using the condition that the minimal length of a word $a \in P_n$ goes to infinity, which implies that the diameter of the corresponding cylinder set $[a]$ goes to zero, we see that there is $n_0$ so that for all $n>n_0$ and $a\in P_n$,
$$
\mathrm{Var}
(f, [a]) < \varepsilon.$$

Since the sets $\{[a] : a \in P_n\}$ are a partition of $B$, we have for all $n>n_0$ and any $b$ that
\[\begin{split}
    &\left| \sum_{a \in P_n}\beta([a]) f(ab) - \int f \der\beta \right| =  \left| \sum_{a \in P_n} \beta([a]) f(ab) - \sum_{a \in P_n} \int_{[a]} f \der\beta \right| \\ \leq & \sum_{a \in P_n} \beta([a]) \,
 \left|f(ab) -
 \frac{1}{\beta([a])} \, \int_{[a]} f \der\beta
   \right| < \sum_{a \in P_n} \beta([a]) \, \varepsilon = \varepsilon.
 \end{split}
\]

The general case, in which $f \in L^\infty(B, \beta)$, will now be proved by an approximation argument. By replacing $f$ with a bounded function agreeing with it almost everywhere, subtracting a constant and rescaling, we may assume that
$$\int_B f \der\beta =0 \ \ \text{ and } \ \  \|f\|_\infty =1.$$
Let $\varepsilon>0$ and, using Lusin's theorem,
let $K = K_\varepsilon$ be a compact set such that $f|_K$ is continuous and $\beta(K) > 1-\varepsilon.$ Let $n_0$ so that for all $n>n_0$ and all $a \in P_n$ we have
$$ \mathrm{Var} (f, [a] \cap K) < \varepsilon.
$$
For fixed $n>n_0$, and an element $a \in P_n$, the corresponding cylinder set $[a]$ is called a {\em continuity atom} or {\em CA} if $\frac{\beta([a]\cap K)}{\beta([a])} \geq 1 - \varepsilon^{1/2}.$
If $a$ is a continuity atom then for any $x_0 \in [a]\cap K,$
\begin{equation}\begin{split} \label{eq: a continuity atom}
&\left| \frac{1}{\beta([a])} \int_{[a]} f \, \der\beta  -
f(x_0) \right|  \leq   \frac{1}{\beta([a])} \left|\int_{[a]} f \, \der\beta - \int_{[a]\cap K} f \, \der\beta \right| \\ + &\left|  \frac{1}{\beta([a])} \int_{[a]\cap K} f  \, \der\beta - \frac{1}{\beta([a]\cap K)} \int_{[a] \cap K} f \, \der\beta \right| + \left| \frac{1}{\beta([a]\cap K)} \int_{[a] \cap K} f \, \der\beta  - f(x_0) \right| \\
\leq & \frac{1}{\beta([a]) } \int_{[a] \smallsetminus K} |f| \, \der\beta
+ \frac{\beta([a] \smallsetminus K )}{\beta([a]) \cdot \beta([a]\cap K)} \int_{[a] \cap K} |f| \, \der\beta  +\mathrm{Var} (f, [a]\cap K) \\
\leq  &  \frac{\beta([a]\smallsetminus K)}{\beta([a])}+ \frac{\beta([a]\smallsetminus K)}{\beta([a])}+ \varepsilon < 2\varepsilon^{1/2} + \varepsilon < 3\varepsilon^{1/2}.
\end{split}\end{equation}

We claim that
\begin{equation}\label{eq: indeed otherwise}
\beta\left(\bigcup \text{continuity atoms} \right)
\geq 1 - \varepsilon^{1/2}.
\end{equation}
 Indeed, otherwise, denoting $K^c = B \smallsetminus K$, we have
$$
\beta(K^c) \geq
\beta \left(\bigcup_{[a] \text{ not a CA}} K^c \cap [a] \right)
\geq
\sum_{a \text{ not a CA}} \beta([a]) \varepsilon^{1/2}
\geq \varepsilon^{1/2} \, \varepsilon^{1/2} = \varepsilon,
$$
and we get a contradiction to $\beta(K^c)<\varepsilon.$

We now define
$$ \mathrm{Bad} = \left\{ b \in B: \beta \left( \bigcup \left\{ [a]:  ab \notin K \right\} \right ) > \varepsilon^{1/2} \right \},
$$
and claim that
\begin{equation}\label{eq: izmel}
\beta(\mathrm{Bad})\leq \varepsilon^{1/2}.
\end{equation}
Indeed, denote by $\mathbf{1} = \mathbf{1}_{K^c}$ the indicator of $K^c$. Recall that for any prefix $[a]$, the measure $\beta|_{[a]}$ is the same as the pushforward of $\beta$ by the map $b \mapsto ab$, multiplied by the scalar $\beta([a])$; this is easily verified for cylinder sets contained in $[a]$ and thus is true for all measurable subsets of $[a]$. Thus, if  \eqref{eq: izmel} is not true then
\[\begin{split}
\beta \left( K^c\right) = &
\sum_{a \in P_n} \int_{[a]} \mathbf{1}\, \der\beta  = \sum_{a \in P_n} \beta([a]) \int_{B} \mathbf{1}(ab) \, \der\beta(b)\\ = & \int_B \sum_{a \in P_n} \beta([a]) \mathbf{1} (ab) \, \der\beta(b) \geq
\int_{\mathrm{Bad}} \sum_{a \in P_n} \beta([a])\mathbf{1} (ab) \, \der\beta(b) \\
= & \int_{\mathrm{Bad}} \beta \left(\bigcup \left\{[a] :  ab \notin K \right\}  \right)  \, \der\beta(b) \geq \beta \left(\mathrm{Bad} \right) \varepsilon^{1/2} > \varepsilon^{1/2} \cdot \varepsilon^{1/2} = \varepsilon,
\end{split}\]
giving a contradiction to $\beta(K^c)<\varepsilon. $

Now for $b \notin \mathrm{Bad}$ we have
 \begin{equation}\label{eq: taar}\begin{split}
    &  \left| \sum_{a\in P_n}f(ab)\beta([a]) - \int_{B} f \, \der\beta \right|   =  \left| \sum_{a\in P_n}f(ab)\beta([a]) - \sum_{a\in P_n} \int_{[a]} f \, \der\beta  \right|\\
     \leq &  \sum_{\substack{[a] \text{ is a CA} \\ ab \in K }} \beta([a]) \left|f(ab) - \frac{1}{\beta([a])} \int_{[a]}f\, \der\beta \right|+
     \sum_{\substack{[a] \text{ not a CA} \\ \text{ or } ab \notin K }}
     \left( \beta([a]) \, |f(ab)| +  \int_{[a]} \left|f\right| \, \der\beta  \right)\\
      \leq
      & 3 \varepsilon^{1/2} + 2 \beta\left(\bigcup \left\{[a]: [a] \text{ is not a CA} \right\} \right) + \beta \left( \bigcup \left\{ [a]:  ab \notin K\right\} \right)< 7 \varepsilon^{1/2},
\end{split}\end{equation}
where in the last line we used  \eqref{eq: a continuity atom}, the definition of $\mathrm{Bad}$, and \eqref{eq: indeed otherwise}.
 Since $\varepsilon$ was arbitrary, combining \eqref{eq: izmel} and \eqref{eq: taar} we get  the desired conclusion.

\section{Preparations for Exponential Drift: Non-concentration of Limit Measures Along \texorpdfstring{W\textsuperscript{st}}{Wst}-Orbits}
\label{sec: non alignment}

The goal of this section is to establish the following result regarding non-alignment of the limit measures $\nu_b$ along orbits of the group $\Wst$. This will serve as crucial input for the exponential drift argument as described in \textsection\ref{sec:outline of proof}.
We retain the notation introduced in \textsection \ref{sec:nonatomic}.

\begin{theorem}
    \label{prop:NonAlignmentLimitMeasures}
  Let $\Phi, \bar h_i, \mathbf{p}, \bar \mu$ and  $\nu$ be as in Theorem \ref{thm: stationary main}.
 Suppose that
    $$
    \nu\left( \left\{x \in \X_{d+1}^S: \mathrm{Stab}(x) \cap W^{\mathrm{st}} \neq \{\mathrm{Id} \} \right\} \right) =0.$$
  Then, for $\beta^X$-almost
  every pair $(b,x) \in B^X$, we have
  \begin{equation*}
    \nu_{b}(\Wst x) = 0.
  \end{equation*}
\end{theorem}

The proof of Theorem~\ref{prop:NonAlignmentLimitMeasures} occupies the rest of this section. The notations and conditions of Theorem \ref{thm: stationary main} will be assumed throughout this section.

A key ingredient in the  proof
of Theorem~\ref{prop:NonAlignmentLimitMeasures}
is  Theorem~\ref{thm:NonAtomicityLimitMeasures}.
The following Lemma shows that the hypotheses of Theorem~\ref{prop:NonAlignmentLimitMeasures} imply the hypotheses of Theorem~\ref{thm:NonAtomicityLimitMeasures}.

\begin{lem}\label{lem: hypotheses imply hypotheses}
    Let $\nu$ be a $\bar \mu$-stationary measure such that $\nu(\Hne x)>0$ for some $x$.
    Then $\mrm{Stab}(x)\cap\Wst \neq \seti{\id}$.
\end{lem}
\begin{proof}
We first claim that
 we can find two distinct words $g_1,g_2$ in the semigroup generated by $\left\{\bar h^{-1}_i: i =1 , \ldots, k \right\}$, of the same length,  such that the orbits  $\Hne g_1x, \Hne g_2 x$ are the same. To see this, note that in order to show that two orbits are the same, it is enough to show that these orbits intersect nontrivially. Moreover,  it follows from \eqref{eq: def h'} and \eqref{eq: def Hne} that the elements  $\bar h_i$ normalize $\Hne$, and hence the same is true for the group generated by the $\bar h_i$. Hence for every word $g$ in the $\bar h^{-1}_i$ we have $g \Hne x = \Hne gx. $ Let $\bar \mu^{\otimes \ell}$ denote the $\ell$-th convolution power of $\bar\mu$; this measure is supported on finitely many products of  $\ell$ elements in $\mathrm{supp} (\bar \mu)$.  Suppose by contradiction that for all $\ell$, the collection $\left\{g^{-1} \Hne x : g \in \supp \left(\bar \mu^{\otimes \ell} \right)\right\}$ is disjoint.  
  Let $c_0\df\nu(\Hne x)>0$ and 
  consider the collection of numbers
 $$\mathcal{C} \df \left\{c_g 
 : g \in \supp \left( \bar \mu^{\otimes\ell} \right) \right\} 
 ,
 \ \ \ \text{ where } c_g \df g_\ast\nu(\Hne x)$$
 (this set depends on $\ell$ but we omit the dependence). Disjointness and the fact that $\nu$ 
is a probability measure implies $\sum_{c \in \mathcal{C}}c \leq 1$.   
 We write $\mathcal{C} = \mrm{Big} \sqcup \mrm{Small},$ where 
 $$\mrm{Big} \df \seti{c \in \mathcal{C} : 
 c >\frac{c_0}{2} } \ \ \text{ and } \ \mrm{Small} \df \seti{c \in \mathcal{C} : 
c \leq \frac{c_0}{2} }.$$ 
 By the stationarity of $\nu$ we have $c_0= \sum_g \bar{\mu}^{\otimes \ell}(\{g\}) c_g,  $ and thus
 $$c_0\leq  \bar{\mu}^{\otimes \ell}(\mrm{Big}) 
 + \frac{c_0}{2}\, \bar{\mu}^{\otimes \ell}(\mrm{Small)} \leq \bar{\mu}^{\otimes \ell}(\mrm{Big}) 
 + \frac{c_0}{2}.$$
 In particular, $\bar{\mu}^{\otimes \ell}(\mrm{Big})\geq \frac{c_0}{2}$.
 On the other hand, letting $p \df \max_i \bar{\mu}(\bar{h}_i)\in (0,1)$ denote the maximal weight for the random walk, we have that 
 $$\bar{\mu}^{\otimes \ell}(\mrm{Big})\leq p^\ell \, \#
 \mrm{Big}$$
 and thus 
 $$\#
 \, \mrm{Big} 
\geq \frac{c_0}{2p^{\ell}} \to_{\ell \to \infty} \infty,$$
 which contradicts the fact that $\nu$ is a probability measure if  the orbits $g^{-1}\Hne x$ are all disjoint for $g\in \supp(\bar{\mu}^{\otimes \ell})$.

The claim implies
that there is $h_1 \in \Hne$ such that $h_1g_1x=g_2 x$.
 Hence, $h \df g_1^{-1}h_1^{-1} g_1 \in \Hne$ satisfies that $h g_1^{-1}g_2 \in \mathrm{Stab}(x)$.
    We write $g^{(\mathrm{ue})}$ for the projection of $g \in G^S $ to $\mrm{G}_{\mrm{ue}}.$
    Since $g_1$ and $g_2$ are of the same length, it follows from \eqref{eq: def h'} that
    $$ g_1^{-1}g_2 = (g_1^{-1}g_2)^{(\mrm{ue})} \in \mrm{U}_{\mrm{ue}},$$ and this implies via \eqref{eq: def Wst} that $g_1^{-1}g_2 \in \Wst$. Since we also have $\Hne \subset \Wst, $ we have found that $hg_1^{-1}g_2$ belongs to $\Wst$. By \eqref{eq: def Hne} we have that $h^{(\mathrm{ue})} = \mathrm{Id}$. Since $(g_1^{-1}g_2)^{(\mrm{ue})} = g_1^{-1}g_2 \neq \mathrm{Id}$, we see that $hg_1^{-1}g_2 \neq \mathrm{Id}$.
\end{proof}

In the notation of \textsection \ref{subsec: notation 7.1},  let $\mathfrak{u}^-_{\mathrm{dt}}$ be the subspace of vectors in $\mathfrak{g}_{\mathrm{dt}}$ which are contracted by the action of $\Ad\left(\bar h_i^{-1} \right)$ for each $i$, let $\mathfrak{u}_{\mathrm{ue}} = \Lie(\mathrm{U}_{S_{\mrm{ue}}}),$ and set
$$
\mathfrak{w}^{\mrm{bc}} \df  \mathfrak{u}_{\mathrm{ue}}\oplus \mathfrak{u}^-_{\mathrm{dt}} .
$$
It is easily checked that, with the notation \eqref{eq: weird notation from SW}, we have
$$
\mathfrak{w}^{\mrm{bc}} = \left \{ x \in \mathfrak{g}: \forall b \in B, \lim_{n\to \infty} \left(\bar h_n^1 \right)^{-1} x =0 \right \},
$$
and
\begin{equation}\label{eq: Lie Wst}
    \Lie(\Wst)
=
\Vne \oplus \mathfrak{w}^{\mrm{bc}}.
\end{equation}
 We define
 $$
 W^{\mrm{bc}} \df\left\{g \in G^S: \forall b \in B, \ \lim_{n\to \infty} \left(\bar h_n^1 \right)^{-1} g \left(\bar h_n^1 \right) = \mathrm{Id}  \right\}. $$
 The superscript `bc' stands for `backward contracting'.
Clearly $\Lie(W^{\mrm{bc}}) = \mathfrak{w}^{\mrm{bc}}$ and $W^{\mrm{bc}} $ is normalized by the random walk.
Let $\Hne W^{\mrm{bc}} $ denote the set of pairwise products $\{hw: h \in \Hne, w \in W^{\mrm{bc}} \}$. Clearly $\Hne W^{\mrm{bc}}  \subset \Wst$, and it follows from \eqref{eq: Lie Wst} that $ \Hne W^{\mrm{bc}} $ contains an open neighborhood of $\mathrm{Id}$ in $\Wst. $ We have the following:

\begin{lem}\label{lem: conclusions imply conclusions}
    There is a finite set $F \subset G^S$ such that $\Wst = \bigcup_{f \in F}  \Hne W^{\mrm{bc}}  f.$ In particular, if $\nu_b(\Hne W^{\mrm{bc}} x)=0$ for $\beta^X$-a.e. pair $(b,x)$, then $\nu_b(\Wst x)=0$ for  $\beta^X$-a.e. pair $(b,x)$.
\end{lem}

\begin{proof}
In order to prove the first assertion we note
that with respect to the partition \eqref{eq: partition of places}, we have from \eqref{eq: def Wst} and \eqref{eq: def Hne} that
$$
\Wst = \mathrm{U}_{S_{\mrm{ue}}} \times \mathrm{G}_{S_{\mrm{dt}}} \times \mathrm{G}_{S_{\mrm{tr}}}, \ \ \ \ \Hne = \{0\} \times \mathrm{P}_{S_{\mrm{dt}}}
\times \mathrm{G}_{S_{\mrm{tr}}}.
$$
Also $W^{\mrm{bc}} $ contains the connected groups $\mathrm{U}_{S_{\mrm{ue}}}$ and $\mathrm{U}^-_{S_{\mathrm{dt}}}$ with Lie algebras $\mathfrak{u}_{\mrm{ue}}$ and $\mathfrak{u}^-_{\mrm{dt}}.$ Thus it suffices to prove the statement for $\mathrm{G}_{S_{\mrm{dt}}} , \mathrm{P}_{S_{\mathrm{dt}}}, \mathrm{U}^-_{S_{\mathrm{dt}}}$ in place of $\Wst, \Hne, W^{\mrm{bc}} .$ Note that the quotient $P_{S_{\mrm{dt}}} \backslash \mrm{G}_{S_{\mrm{dt}}}  $ is compact,  and since $W^{\mrm{bc}}  \Hne$ contains a neighborhood of the identity, $\mrm{U}_{S_{\mrm{dt}}}$ projects onto a subset of $\mrm{P}_{S_{\mrm{dt}}} \backslash \mrm{G}_{S_{\mrm{dt}}} $ with nonempty interior. Thus the cover $\{\mrm{P}_{S_{\mrm{dt}}} \mathrm{U}^-_{S_{\mathrm{dt}}} g : g \in \mathrm{G}_{S_{\mrm{dt}}}\}$ has a finite sub-cover, and this proves the first assertion.
    The second assertion follows from the first one, by covering an orbit $\Wst x$ of positive measure for $\nu_b$, by the finitely many sets $\Hne W^{\mrm{bc}}  f x, \ f \in F$.
\end{proof}

\begin{lem}\label{lem:AvoidanceStable}
 Suppose that
    \begin{equation}\label{eq: probably assumption}
    \nu\left( \left\{x \in \X_{d+1}^S: \mathrm{Stab}(x) \cap W^{\mathrm{st}} \neq \{\mathrm{Id} \} \right\} \right) =0.
    \end{equation}
    Then
  \begin{equation*}
    \nu_{b}(\Hne W^{\mrm{bc}}  x \sm \Hne x) = 0
  \end{equation*}
  for almost every pair $(b,x) \in B^X$.
\end{lem}
\begin{proof}
Let $\mathbf{Z} \df B \times X \times X$  and define $R \colon \mathbf{Z} \to \mathbf{Z}$ by
  \begin{equation*}
    R(b,x,x') \df \big( Tb , b_1^{-1}x, b_1^{-1}x' \big), \qquad \text{where } b\in B, \ b= (b_1, b_2, \ldots),
  \end{equation*}
  and $T: B \to B$ is the shift.
  We define a probability measure $\tilde{\beta}$ on $\mathbf{Z}$ by
  \begin{equation}\label{eq: def tilde beta}
    \tilde{\beta} = \int_B \delta_b \otimes \nu_b \otimes \nu_b \, \der \beta(b).
  \end{equation}
  It follows from
  the equivariance property $\nu_{h_ib} = h_{i*} \nu_b$ in Proposition \ref{prop: commutes with RW} that
  $\nu_{Tb} = (b_1^{-1})_* \nu_b,$
  and this implies that $\tilde{\beta}$ is $R$-invariant.
  Suppose by contradiction that $E \subset
  B^X$ is a measurable set such that $\beta^X(E) > 0$ and
  \begin{equation*}
    \forall (b,x) \in E \quad \nu_b(\Hne W^{\mrm{bc}}  x \sm \Hne x) > 0.
  \end{equation*}
  Let
  \begin{equation*}
    \Zcal = \{(b,x,x') \in \mathbf{Z} \colon x' \in \Hne W^{\mrm{bc}}  x \sm \Hne x\}.
  \end{equation*}
  Note that $\Zcal$ is $R$-invariant. Indeed, suppose that $z = (b,x,x') \in \Zcal$
  and write $x' = pux$ where $p\in \Hne$ and $u\in W^{\mrm{bc}}  \sm \{\mathrm{Id}\}$.
  Then
  $$R(z) =
  \left(Tb, b_1^{-1}x, \bar b_1 b_1^{-1}x \right), \ \ \text{ where } \bar b_1 = b_1^{-1}(pu) b_1. $$
  Since the conjugation map $g \mapsto b_1^{-1} g b_1$
  is a group automorphism, and  the groups $\Hne, \, W^{\mrm{bc}} $ are both invariant under conjugation by elements of the random walk, we have that $\bar b_1 \in \Hne W^{\mrm{bc}} $ and its $W^{\mrm{bc}} $-component $b_1^{-1}ub_1$ is nontrivial.

  We claim that $\tilde{\beta}(\Zcal) > 0$.
  Indeed, if $\tilde{\beta}(\Zcal) = 0$ then by \eqref{eq: def tilde beta} and Fubini's theorem we have that
  for $\beta$-a.e.~$b\in B$ and for
  $\nu_b$-a.e.~$x \in X$ we have that
  \begin{equation*}
    \nu_b(\Hne W^{\mrm{bc}} x \sm \Hne x) = \nu_b\big(\{ x' \in X \colon x' \in
    \Hne W^{\mrm{bc}} x \sm \Hne x\}\big) = 0;
  \end{equation*}
 but recalling \eqref{eq: def beta X}, we see that this contradicts  the assumption that $\beta^X(E)>0.$

  We now remove from $\Zcal$ the set of $(b,x,x')$ for which $\mathrm{Stab}(x) \cap \Wst \neq \{\mathrm{Id} \}.$ We continue to denote the resulting set by $\Zcal$ and note that by assumption \eqref{eq: probably assumption}, it still satisfies $\tilde \beta (\Zcal)>0.$ Following this modification,  if
  $z=(b,x,x') \in \Zcal$, then there are {\em unique} $p \in \Hne$ and $w  \in W^{\mrm{bc}}  \sm
  \{\mathrm{Id}\}$ such that $x' = pw x$; indeed, if  $\tilde{p} \in \Hne$ and
  $\tilde{w} \in W^{\mrm{bc}} $ also satisfy $x' = \tilde{p}\tilde{w} x$, then
  $w^{-1}p^{-1}\tilde{p}\tilde{w} \in \mrm{Stab}(x) \cap W^{\mrm{st}}$.
  Using this uniqueness, and denoting $p = p(z), w = w(z)$ the elements in $\Hne$ for which $x'=pwx,$ we obtain an almost surely well-defined map
  \begin{equation}\label{eq:project mod forward contracting}
  \Theta \colon \mathcal{Z} \to [0,\infty), \ \ \
    \Theta(b,x,x') \df \mrm{dist}_{\XS} \left(x,p(z)^{-1}x' \right) \qquad \text{(where $z = (b,x,x') \in
    \mathcal{Z}$).}
  \end{equation}
  Here $\dist_{\XS}$ is the metric introduced in \textsection \ref{subsec: notation 7.1}.
 Then $\Theta$ is positive on $\mathcal{Z}.$

  Given $b \in B$ and $n\in\bN$, and writing $b_n^1$
  as in \eqref{eq: weird notation from SW},
we have
  \begin{equation*}
    R^nz = \left(T^n b , (b_{n}^1)^{-1} x , (b_{n}^1)^{-1}x'\right) = \left(T^n b, (b_{n}^1)^{-1} x,
    p_{z,n} w_{z,n} (b_{n}^1)^{-1}  x\right),
  \end{equation*}
  where
  \begin{equation*}
    p_{z,n} = (b_{n}^1)^{-1} p(z) b_{n}^1, \quad w_{z,n} = (b_{n}^1)^{-1}w(z)
    b_{n}^1.
  \end{equation*}
  In particular, by the construction of the metrics $\dist_{\XS}, \dist_{G^S},$ there is $C>0$ such that for all large enough $n$,
  \begin{equation*}
    \Theta(R^nz) = \mrm{dist}_{\XS} \left( (b_{n}^1)^{-1} x, w_{z,n} (b_{n}^1)^{-1} x \right) \leq  C \,
    \dist_{G^S} \left(\mathrm{Id},w_{z,n} \right) .
  \end{equation*}
  In particular $\Theta(R^nz) \to 0,$ on a set $\mathcal{Z}$ of positive measure.
  This gives a contradiction to
  Poincar\'e recurrence, and completes the proof.
\end{proof}

\begin{lem}\label{lem:AvoidanceContracting}
  Suppose that for $\beta^X$-a.e. $(b,x) \in B^X$ we have $\nu_b(Zx)=0$ (where $Z$ is the centralizer of the random walk as in \textsection \ref{sec:nonatomic}).
Then for $\beta^X$-a.e. $(b,x) \in B^X$ we have $\nu_b(\Hne x)=0$.
\end{lem}
\begin{proof}
  Recall from \eqref{eq: def Hne} and Proposition \ref{prop: projection Z trivial} that $\Hne = Z\gUdt$.
  Let
  \begin{equation*}
    E \df \left\{(b,x) \colon \nu_b(\Hne x) > 0\right \} \ \ \text{ and } \ \  F \df \left\{(b,x) \in B^X \colon \nu_b(Z x) > 0\right\}.
  \end{equation*}
  We assume $\beta^X(F)=0$ and
  suppose for sake of contradiction that $\beta^X(E) > 0$. Since $\bar \mu$ is finitely supported, we can
  assume without loss of generality that for all $(b,x) \in E$, for all $k \in
  \mathbb{N}$, and for all $a \in \mrm{supp}(\bar \mu)^k$ we have
  $  \nu_{ab} = a_\ast\nu_b.
  $
   Using Lusin's theorem, let
  $K\subset E \setminus F$ be a compact  set with $\beta^X(K)>0$ and such that the map $(b,x) \mapsto \nu_b$ is continuous on $K$.

Let
  $
  T^X: B^X \to B^X$
  be as in  \eqref{eq: def beta X}.
  By  Proposition \ref{prop: commutes with RW},  $\beta^X$ is $T$-invariant and ergodic.
  Let
$  L \colon L^1(B^X,\beta^X) \to L^1(B^X,\beta^X)$
  be the
  adjoint operator of $T^X$; that is, for all $f \in
  L^1(B^X,\beta^X)$ and for all $\varphi \in L^\infty(B^X,\beta^X)$
  we have
  \begin{equation*}
    \int_{B^X} f \cdot (\varphi \circ T^X) \, \der \beta^X = \int_{B^X} Lf \cdot
    \varphi \, \der \beta^X.
  \end{equation*}
From the definition and by a straightforward induction, one finds that $\beta^X$-almost surely
  \begin{equation*}
    L^nf(b,x) = \sum_{a \in B^*, \, \mathrm{len}(a)=n} f(a b,a_n^1 \cdot x) \,  \beta([a]).
  \end{equation*}
  We recall that by the Chacon-Ornstein ergodic theorem~\cite{ChaconOrnstein}
  the averages
  \begin{equation*}
    A_N^\ast f(b,x) = \frac{1}{N}\sum_{n=0}^{N-1}L^nf(b,x)
  \end{equation*}
  converge pointwise $\beta^X$-almost surely.
  We choose large compact subsets $C$ and $M$ of $\mrm{U}_{\mrm{dt}}$ and $Z$ respectively, and $\varepsilon>0$, such that
   the set
  \begin{equation*}
    Q' \df \{(b,x) \in K \colon \nu_b(MC x) > \varepsilon\}
  \end{equation*}
  has positive measure.
  We let $Q \subset Q'$ be compact and of positive
  measure. By the Chacon-Ornstein ergodic theorem and ergodicity of $T^X$, we
  know that $A_N^\ast\mathbf{1}_Q(b,x) \to \beta^X(Q) > 0$
  for $\beta^X$-a.e. $(b,x) \in B^X$. Thus, there are $(b,x) \in Q$, a
  sequence $n_j \to \infty$, and a sequence of words $a(j) \in B^*$ of length $n_j$, such that
  \begin{equation*}
    \forall j, \quad (b_j,x_j) \df \big(a(j)b,a(j)\cdot x\big) \in Q.
  \end{equation*}
Here we view $a(j)$ as both an element of $B^*$, and as an element of $G^S$ obtained by the product  $\left(a(j)\right)_j^1$.   Passing to a subsequence, we assume that $(b_j,x_j) \stackrel{j \to \infty}{\longrightarrow}  (b',x') \in Q$.
  We claim that
  \begin{equation}\label{eq: to see this}
    \nu_{b'}(M x') > 0.
  \end{equation}
This will give us the desired contradiction, since  $M \subset Z$ and $(b',x') \in K \subset E \sm F$.

  To see \eqref{eq: to see this}, let $\mathcal{O} \subset G$ be a compact
  neighborhood of the identity in $G^S$. Using Urysohn's lemma and outer
  regularity of $\nu_{b'}$, we have that
  \begin{equation*}
    \nu_{b'}(M\mathcal{O} x') = \inf\{\nu_{b'}(\varphi) \colon \varphi \in
    C_c(X) \text{ and }\varphi|_{M\mathcal{O} x'} = 1\}.
  \end{equation*}
If follows from the definitions of the group $U$, the places $S_{\mrm{dt}}$, and the elements $\bar h_i$ (see \eqref{eq: def u+}, \eqref{eq: def Sdt} and \eqref{eq: def h'}) that the conjugation of elements of $\mathrm{U}_{\mrm{dt}}$ by elements of the $\bar \mu$ random walk is uniformly contracting,   in the following sense. For any compact subset $C_1 \subset \mrm{U}_{\mrm{dt}}$ and any open neighborhood $\mathcal{O}_1$ of the identity in $\mrm{U}_{\mrm{dt}}$, there is $k_0$ such that for all $k \geq k_0$, for any $a \in B^*$ of length $k$, we have
$$a_k^{1} C_1 (a_k^1)^{-1} \subset \mathcal{O}_1.$$
In particular, taking $C_1 = C$ and $\mathcal{O}_1$ so that $\mathcal{O}_1$ such that the closure of $\mathcal{O}_1$ is contained in the interior of $ \mathcal{O}$,
 we have that
  \begin{equation*}
    a(j)Cx = a(j)Ca(j)^{-1}  x_j\subset \mathcal{O}x'
  \end{equation*}
  for all sufficiently large
  $j$. Therefore, for $\varphi\in C_c(X)$ satisfying
  $\varphi|_{M\mathcal{O} x'} = 1$, we have
  \begin{equation*}
    \int \varphi \, \der \nu_{b'} \geq \limsup_{j\to \infty} \nu_{a(j)b}(M\mathcal{O} x')
    \geq \limsup_{j\to \infty}\nu_{a(j)b}(Ma(j)C x) = \nu_b(MC x)
    > \varepsilon,
  \end{equation*}
  and therefore $\nu_{b'}(M\mathcal{O} x') \geq \varepsilon$. Now $M x'$ is
  the intersection of all sets of the form $M\mathcal{O} x'$, where
  $\mathcal{O}$ is a neighborhood of the identity in $G^S$. This yields $\nu_{b'}(M x') \geq
  \varepsilon$ and
proves \eqref{eq: to see this}.
\end{proof}

\begin{proof}[Proof of Theorem~\ref{prop:NonAlignmentLimitMeasures}]
 Using Lemma \ref{lem: hypotheses imply hypotheses} we see that under our hypotheses, we can apply
Theorem~\ref{thm:NonAtomicityLimitMeasures}. This yields
$$\text{ for } \beta^X \text{-a.e. } (b,x), \ \
\nu_b(Zx)=0.
$$
Hence, by
  Lemma~\ref{lem:AvoidanceContracting},
$$\text{ for } \beta^X \text{-a.e. } (b,x), \ \
\nu_b(\Hne x)=0.
$$
 Combining this with
  Lemma~\ref{lem:AvoidanceStable}, it follows that
  $$\text{ for } \beta^X \text{-a.e. } (b,x), \ \
\nu_b(\Hne W^{\mrm{bc}}  x)=0.
$$
Applying Lemma \ref{lem: conclusions imply conclusions} we get that
$$\text{ for } \beta^X \text{-a.e. } (b,x), \ \
\nu_b(\Wst x)=0.
$$\qedhere
\end{proof}

\section{Additional Invariance: The Exponential Drift} \label{sec: exponential drift}
The goal of this section is to prove that, under the assumptions of Theorem~\ref{prop:NonAlignmentLimitMeasures}, almost every limit
measure decomposes as a convex combination of measures invariant under
one-parameter subgroups of $\gUue$. To this end, we will employ  the exponential drift argument introduced by Benoist and Quint in~\cite{BQ1}. Running the argument will require the preliminary work done in \textsection\ref{sec:GrowthPropertiesRandomWalk}
--\textsection \ref{sec: non alignment}.
The following result is the main result of this section.

\begin{theorem}\label{thm:AdditionalInvariance}
  Suppose that $\nu$ is a stationary measure satisfying the hypotheses of Theorem~\ref{prop:NonAlignmentLimitMeasures}. Then for $\beta$-a.e.~$b\in B$ and for
  $\nu_b$-a.e.~$x \in \XS$, there exists a non-trivial subgroup $W_{(b,x)}
  \subset \gUue$ generated by one-parameter unipotent subgroups and a
  $W_{(b,x)}$-invariant probability measure $\nu_{(b,x)}$ on $\XS$ such that
  \begin{equation*}
    \nu_{b} = \int_{X} \nu_{(b,x)} \, \mrm{d}\nu_b(x).
  \end{equation*}
  Moreover, for $\beta^X$-almost every $(b,x) \in B^X$ we have the equivariance property:
  \begin{align}\label{eq:equivariance of W_z}
   W_{(b,x)} = b_1
  W_{T^X(b,x)}b_1^{-1}, \qquad \text{and}
  \qquad \nu_{(b,x)} = {b_1}_\ast \nu_{T^X(b,x)}.   
  \end{align}
  
\end{theorem}

For the rest of this section we retain the notations and assumptions of Theorem \ref{thm:AdditionalInvariance}. The proof relies on an analysis of leafwise measures, for which we follow ~\cite[\textsection4]{BQ1}. For more information on leafwise measures, we refer the reader to \cite{Einsiedler_Lindenstrauss_Clay}.
We now recall the necessary notations and results.

Denote by $\mathcal{M}(\ufrak_{\mrm{ue}})$ the convex cone of positive non-null
Radon measures on $\ufrak_{\mrm{ue}}$.
Also let $\mathcal{M}_1(\ufrak_{\mrm{ue}})$ be the set of rays in
$\mathcal{M}(\ufrak_{\mrm{ue}})$, i.e., the quotient space for the equivalence relation
of {\em proportionality}, defined by
\begin{equation*}
  \sigma_1 \propto \sigma_2 \iff \exists c > 0 \text{ such that } \sigma_2 = c\sigma_1.
\end{equation*}
The choice of $\s$ to denote leafwise measures is consistent with \cite{BQ1}, although we also use $\s$ to denote elements of $S$; even punctilious readers should have no difficulty disambiguating these two uses.
Let
\begin{equation*}
    \overline{\sigma} \colon B^X \rightarrow \mathcal{M}_1(\ufrak_{\mrm{ue}}), \ \ \ \ \ \
    (b,x)  \mapsto \overline{\sigma}_{(b,x)}
\end{equation*}
be a family of leafwise measures for $\beta^X$ with respect to the action of
$\ufrak_{\mrm{ue}}$ given by
\begin{equation}
\label{eq: group action rhs}
    \Psi \colon B^X \times \ufrak_{\mrm{ue}} \rightarrow B^X, \ \ \ \ \ \
   \Psi \big((b,x), u\big)  \df\big(b,\exp(u)x\big).
\end{equation}
For $u \in \ufrak_{\mrm{ue}}$ we will use the notation
\begin{equation}\label{eq: use the notation}
\Psi_u : B^X \to B^X, \ \ \ \Psi_u(z) \df \Psi(z,u).
\end{equation}
Since the Lie algebra $\ufrak_{\mrm{ue}}$ is abelian, $\exp: \ufrak_{\mrm{ue}} \to \mrm{U}_{\mrm{ue}}$ is a group isomorphism and the expression on the right-hand side of \eqref{eq: group action rhs} coincides with the action of $\mrm{U_{ue}}$ on $\XS.$ It will be more convenient (and also consistent with the notation of \cite{BQ1}) to stick with exponential notation and use $\ufrak_{\mrm{ue}}$ instead of $\mrm{U_{ue}}.$

Note that $\overline{\sigma}_{(b,x)}$ denotes an equivalence class of measures.
In order to choose a concrete representative, we let $(K_n)_{n \in \N}$ be a sequence of compact sets which form an exhaustion of $\ufrak_{\mrm{ue}}$, and we denote by $\sigma_{(b,x)}$ the element in the proportionality class
$\overline{\sigma}_{(b,x)}$ satisfying $\sigma_{(b,x)}(K_{n_0}) = 1$, where $n_0$ is
the minimal $n$ for which  $\overline{\sigma}_{(b,x)}(K_n)>0.$
Similarly, given $\sigma \in \mathcal{M}(\ufrak_{\mrm{ue}})$, we will denote by
$\overline{\sigma} \in \mathcal{M}_1(\ufrak_{\mrm{ue}})$ its equivalence class.
We will fix a conull subset $E \subset B^X$ on which
$\sigma$ is well-defined up to proportionality and satisfies all the
characterizing properties of leafwise measures; cf.~\cite[Prop.~4.2]{BQ1}.

Recall that $\varrho\in
\bQ$ denote the contraction ratio of the maps in the IFS $\Phi$, and for any $u \in \ufrak_{\mrm{ue}}$ and any $\bar h_i \in \supp \bar \mu$ we have $\Ad_{\bar h_i}(u) = \varrho u$. Therefore
\begin{equation}\label{eq:EquivarianceHorocycleFlow}
  \forall u \in \ufrak_{\mrm{ue}},
  \quad  \quad \Psi_{\varrho^{-1} u} \circ T^X  =
  T^X \circ \Psi_u.
\end{equation}
Given $a \in \Q_S^\times$, denote
\begin{equation*}
    \eta_a \colon \ufrak_{\mrm{ue}} \rightarrow \ufrak_{\mrm{ue}}, \ \ \ \ \
    \eta_a(u)  \df au.
\end{equation*}
Let $\mathcal{B}^X$ denote the Borel $\sigma$-algebra on $B^X$ and let
\begin{equation}\label{eq:TailSigmaAlgebra}
  \mathcal{Q}_\infty^X = \bigcap_{n\in\mathbb{N}}\left(T^X \right)^{-n}\mathcal{B}^X.
\end{equation}
The following is an immediate consequence of the uniqueness of leafwise
measures and~\eqref{eq:EquivarianceHorocycleFlow}.
\begin{corollary}[{cf.~\cite[Cor.~6.12 and Cor.~6.13]{BQ1}}]\label{cor: E as in}
  There exists a conull subset $E\subset B^X$ such that
\begin{equation}\label{eq:LeafwiseUnderAutomorphism}
    \forall (b,x) \in E, \ \forall n\in\mathbb{N}, \quad (T^X)^n(b,x) \in E \implies
    \sigma_{(b,x)} \propto (\eta_{\varrho^n})_\ast \sigma_{(T^X)^n(b,x)}.
  \end{equation}
  Moreover, the map $\overline{\sigma}
  $ is
  $\Qcal_\infty^X$-measurable.
\end{corollary}

Note that the right hand side of formula \eqref{eq: group action rhs} involves the action of the group $\mrm{U_{ue}}$ on $\XS$, and in contrast with \cite{BQ1}, the acting group does not depend on $b$. Therefore
for $\beta^X$-a.e.~$(b,x)$, $\bar \sigma(b,x)$ is the leafwise measure of $\nu_b$ with respect to the $\mrm{U_{ue}}$-action on $\XS.$

We also introduce notation for the translation map
\begin{equation*}
    \tau \colon \ufrak_{\mrm{ue}} \times \ufrak_{\mrm{ue}}
    \rightarrow \ufrak_{\mrm{ue}}, \quad \quad
 \tau_u(v) \df  v+u.
\end{equation*}
We recall from~\cite[Prop.~4.2]{BQ1} that (after removing a nullset),  for all $
z \in
E$ and all $u \in \ufrak_{\mrm{ue}}$ such that $\Psi_u(z) \in E$, we have that
\begin{equation}\label{eq:TranslationLeafwiseMeasure}
  \sigma_{z} \propto {\tau_u}_\ast \sigma_{\Psi_u(z)}.
\end{equation}

Given Theorem~\ref{thm:NonAtomicityLimitMeasures}, the main step towards a proof
of Theorem~\ref{thm:AdditionalInvariance} is to show that there exist many
generic pairs $(b,x)$ for which the leafwise measure agrees  up to proportionality with the translate
by arbitrarily small non-trivial elements of
$\ufrak_{\mrm{ue}}$.
The first step in the argument is the proof that there exist many nearby pairs of
points whose displacement is
not contained in
\begin{equation}
\label{eq: def wstfrak}
\wstfrak \df \Lie(\Wst) = \ufrak_{\mrm{ue} }\oplus
\gfrak_{\mrm{dt}} \oplus \gfrak_{\mrm{tr}}
.\end{equation}

\begin{lem}\label{lem:SequenceOfDisplacements}
  For every measurable set $L \subset B^X$ there exists a conull subset $L'\subset L$ and, for every element $(b,x) \in L'$, a sequence
  $(v_m)_{m\in\mathbb{N}}$ of elements of $\gfrak \sm \wstfrak $
  such that $v_m \to 0$ as $m \to \infty$ and
  \begin{equation*}
    \forall m \in \mathbb{N}, \quad \quad \big(b , \exp(v_m) x\big) \in L.
  \end{equation*}
\end{lem}
\begin{proof}
  Assume without loss of generality that $\beta^X(L) > 0$ and, using inner
  regularity, that $L$ is compact. Denote by $F'$ the projection of $L$ to $B$
  and for any $b\in F'$ let $L_b$ denote the fiber
  $\{x \in \XS : (b,x) \in L\}$,
  so that
  \begin{equation*}
    \beta^X(L) = \int_{F'}\nu_b(L_b) \, \mrm{d}\beta(b).
  \end{equation*}
  Let $F \subset F'$ consist of those $b \in F'$ satisfying $\nu_b(L_b) > 0$.
  The intersection $\tilde{L} \subset B^X$ of the preimage of $F$ with $L$ is
  a conull subset of $L$. Let $b \in F$, then for $\nu_b$-a.e.~$x \in L_b$ and for
  every neighborhood $V$ of $0$ in $ \mathfrak{g}$  for which the
  exponential map is well-defined,
  \begin{equation}\label{eq: is well defined 1}
    \nu_b\big(L_b \cap \exp(V)x\big) > 0.
  \end{equation}
  From
 Theorem \ref{prop:NonAlignmentLimitMeasures}  we know that for all $b\in F$
  we have
  \begin{equation}\label{eq: is well defined 2}
    \text{for $\nu_b$-a.e. $x$,} \quad \nu_b\big(L_b  \sm \Wst x
      \big) > 0.
  \end{equation}
 Then the claim holds if we take $L'$ to be the subset of $\tilde{L}$ consisting of $(b,x)$ for which \eqref{eq: is well defined 2} holds, and \eqref{eq: is well defined 1} holds for a countable basis of identity neighborhoods $V$.
\end{proof}
We set up some more notation. We write $(\supp \bar \mu)^n$ for the words in $B^*$ of length $n$. Given  $(b,x) \in B^X$,
define a map
\begin{equation*}
    h_{n,(b,x)} \colon (\supp \bar \mu)^n
    \rightarrow B^X, \ \ \ \ \ \ \ \
   h_{n,(b,x)}(a) \df  \left(aT^nb,a_n \cdots a_1 b_n^{-1} \cdots
    b_1^{-1}x\right);
\end{equation*}
i.e., for any $a \in (\supp \bar \mu)^n$, and $z \in B^X$, $z' = h_{n,z}(a) $ has the same future
(under $T^X$) as $z$ and, in fact, $(T^X)^n(z) = (T^X)^n(z')$. The
image of $h_{n,z}$ is the set of all points whose futures agree with that
of $z$ after time $n$. Put differently,
\begin{equation*}
  \mrm{im}\big(h_{n,z}\big) = \left(T^X \right)^{-n}\big\{(T^X)^n(z)\big\}.
\end{equation*}
As in~\cite[Prop.~2.3]{BQ1}, using the decreasing martingale theorem, we have
that for all $\varphi \in \mrm{L}^{\infty}(B^X,\beta^X)$,
\begin{equation}\label{eq:MartingaleTheorem}
  \text{for $\beta^X$-a.e. $z$,} \quad \quad\mathbb{E}\left(\varphi | \mathcal{Q}_\infty^X \right)(z) =
  \lim_{n\to\infty} \sum_{a \in (\supp \bar \mu)^n} \beta([a]) \, \varphi \circ h_{n,z}(a)
  .
\end{equation}
Indeed, the image of $h_{n,z}$ is
the atom of $z$ in $(T^X)^{-n}\mathcal{B}^X$ and thus one only has
to check that the right hand side defines a conditional measure for $\beta^X$
with respect to the sub-$\sigma$-algebra $\left(T^X \right)^{-n}\mathcal{B}^X$.

\begin{lem}
Let $E$ be as in
Corollary \ref{cor: E as in}.
There exists a conull subset $F \subset E$ such that for
  $
  z\in F$, for all $n \in \mathbb{N}$, and for all $a \in (\supp \bar \mu)^n$ we have
  \begin{equation*}
    h_{n,
    z}(a) \in E \implies \sigma_{
    z} \propto
    \sigma_{
    z'},  \ \ \text{ where }
    z'= h_{n,
    z}(a).
  \end{equation*}
\end{lem}
\begin{proof}
Formula~\eqref{eq:LeafwiseUnderAutomorphism} means that as long as $z = (b,x) \in E$, the measure $\sigma_{
z}$ can be recovered from $\sigma_{(T^X)^nz}$ for every $n$, by composition with the contraction map $\eta_{\varrho^n}$.
Thus the conclusion of the Lemma follows from the fact
  $(T^X)^n(
  z') = (T^X)^n(
  z)$.
\end{proof}
Given $n\in\bN$ and $c\in B^\ast \cup B
$ such that $\mrm{len}(c)\geq n$, following \cite{BQ1}, we use the letter $R$ to denote the adjoint representation, so that
$$
R\left(c_n^1 \right) \df \Ad_{c_1} \circ \cdots \circ \Ad_{c_n} , \ \ \ \ \ R\left(c_1^n \right) \df \Ad_{c_n} \circ \cdots \circ \Ad_{c_1},
$$
and set
\begin{equation}\label{eq:ShorthandAdjointAction}
 F_{n,b}(a) =
  R\left(a_1^n \right) \circ R\left(b_n^1 \right)^{-1}.
\end{equation}
 Note that
     for $\s \in S \sm \Sue,$ the
   adjoint  action of $F_{n,b}(a)$ on $\gfrak_\s$ is trivial for any $n,b,a$.

Using Lemma~\ref{lem:SequenceOfDisplacements}, we can run the core of the
exponential drift argument.
\begin{prop}[{cf.~\cite[Prop.~7.1]{BQ1}}]\label{prop:ExponentialDrift}
For any $\delta_0>0$, for $\beta^X$-almost every $
z \in E$
  there exist $u \in \ufrak_{\mrm{ue}}$
  and $
  z' \in E$ such that
  \begin{equation}\label{eq:GoalInvariance}
    0 < \lVert u \rVert < \delta_0, \quad
    \Psi_u(z')
    \in E, \quad  \quad
    \sigma_{\Psi_u(z')}
    \propto
    \sigma_{z'} \propto \sigma_{z}.
  \end{equation}
\end{prop}
\begin{proof}
  Let $\mathcal{Q}_\infty^X $ be the $\sigma$-algebra defined by \eqref{eq:TailSigmaAlgebra}, let $\varepsilon \in (0,1)$ be arbitrary and fix a compact subset $K \subset E$
  such that $\overline{\sigma}$ is continuous on $K$ and such that
  $\beta^X(K) > 1-\varepsilon^2/2$. We denote by $\mathbf{1}_K$ the indicator function
  of $K$. Recall that the conditional expectation  $\mathbb{E}\left(\mathbf{1}_K |
  \mathcal{Q}_\infty^X \right)(\cdot) $ is an almost everywhere defined $\mathcal{Q}_\infty^X $-measurable function on $B^X$ which is uniquely defined up to a nullset; we fix one  representative and continue to denote it by $\mathbb{E}\left(\mathbf{1}_K |
  \mathcal{Q}_\infty^X \right)$.  We let
  \begin{equation*}
    F_\varepsilon = \left\{
    z \in E \colon \mathbb{E}\left(\mathbf{1}_K |
    \mathcal{Q}_\infty^X \right)
    (z) > 1-\varepsilon\right\}.
  \end{equation*}
  One computes that
  \begin{align*}
    1 - \frac{\varepsilon^2}{2} & < \beta^X(K)
    \leq (1 - \varepsilon)\big(1 - \beta^X(F_\varepsilon)
    \big) + \beta^X(F_\varepsilon) \\
    & = 1 - \varepsilon + \varepsilon \beta^X(F_\varepsilon).
  \end{align*}
  Hence $\beta^X(F_\varepsilon) > 1 - \frac{\varepsilon}{2}$. Given
  $n\in\mathbb{N}$, let
  \begin{equation*}
      \psi_n \colon B^X
      \longrightarrow \mathbb{R},
     \ \ \ \ \ \ \ \ \  \psi_n
     (z)
     \df \sum_{a \in (\supp \bar \mu)^n}  \bar \beta([a]) \, \mathbf{1}_K\circ
      h_{n,
      z }(a).
  \end{equation*}
  Using~\eqref{eq:MartingaleTheorem}, we fix a full measure subset $E_1
  \subset E$ such that
  \begin{equation}\label{eq: convergence in E}
    \forall
    z\in E_1, \quad  \quad \psi_n(
    z) \overset{n\to
    \infty}{\longrightarrow} \mathbb{E}\left(\mathbf{1}_K | \mathcal{Q}_\infty^X \right)(
    z).
  \end{equation}
  We let $L_1 \subset E_1$ be a compact continuity set for $\mathbb{E}\left(\mathbf{1}_K |
  \mathcal{Q}_\infty^X \right)$ such that $\beta^X(L_1) > 1 - \varepsilon$. Using
  Egorov's theorem, we can additionally assume that the convergence in \eqref{eq: convergence in E}
  is uniform on $L_1$, and thus there exists $n_0 \in \mathbb{N}$ such that
\begin{equation}\label{eq:UniformityOfGoodProportion}
    \forall n \geq n_0,\ \forall
    z \in L_1, \quad \psi_n(
    z) > 1-\varepsilon.
  \end{equation}
  This implies that
  \begin{equation*}
    \forall n\geq n_0,\ \forall
    z \in L_1,\ \quad \bar
    \beta \left(\bigcup\{[a]: a \in
    (\supp \bar \mu)^n, \ h_{n,
    z}(a) \in K \}\right) > 1 - \varepsilon.
  \end{equation*}
  Using Theorem~\ref{prop:NonAlignmentLimitMeasures}, we fix a compact set
  $L \subset L_1$ such that $\beta^X(L) > 1 - \varepsilon$ and
  \begin{equation*}
    \forall (b,x)
    \in L,\ \quad \nu_b(\Wst x) = 0.
  \end{equation*}
  We further fix a compact subset $L' \subset L$ of measure $\beta^X(L') >
  1-\varepsilon$ satisfying the conclusion of
  Lemma~\ref{lem:SequenceOfDisplacements}.

  For what follows, we fix an element $z_0 = (b,x) \in L'$ and a sequence
  $(v_m)_{m \in \mathbb{N}}$ of vectors in $\gfrak\sm \wstfrak $ such that $\exp(v_m)$ is well-defined for all $m$, $v_m \to 0$ as $m \to \infty$,  and
  \begin{equation*}
    \forall m\in\bN  \quad z_m \in L, \text{ where } z_m \df \big( b , \exp(v_m)x \big) .
  \end{equation*}
Using~\eqref{eq:UniformityOfGoodProportion}, we know that for all $m\geq 0$
  and for all $n\geq n_0$ we have
  \begin{equation}\label{eq:ManyGoodPrefixes}
    \beta\big(\{a \in (\supp \bar \mu)^n \colon h_{n,z_m}(a) \in K\}\big) > 1 -
    \varepsilon.
  \end{equation}

Extending the definition in \eqref{eq: use the notation}, we write $\Psi_v(b,x) =(b, \exp(v)x)$ for every $v \in \gfrak$ for which $\exp(v)$ is well-defined. Note that the map $(z,v) \mapsto \Psi_v(z)$ is continuous where defined. Recall (or see \cite[\textsection 5]{BQ2}) that if $\exp(v)$ is well-defined, then so is $\exp(\Ad(g)v)$ for every $g \in G^S$, with $\exp(\Ad(g)v) = g\exp(v) g^{-1}$. From this and using~\eqref{eq:ShorthandAdjointAction}
we  see that if  $y \in X$ satisfies $y = \exp(v)x$, where $\exp(v)$ is well-defined,  then
  \begin{equation*}
    \forall a \in (\supp \bar \mu)^n \quad \quad h_{n,(b,y)}(a) =
    \big(\Psi_{F_{n,b}(a)v}\circ h_{n,(b,x)}\big)(a).
  \end{equation*}
  Given $\s\in S$, let $v^{(\s)}_{m} \in \gfrak_\s$ denote the
  projection of $v_m$ to $\gfrak_\s$.
  Given $m, n\in\bN$ and $\s \in S$,
  define
   $$ r_{n,m,\s} \df  \lvert \varrho\rvert_\s^n \, \left\|
   R\left(b_n^1 \right)^{-1}v^{(\s)}_{m}\right \|_\s \ \ \text{ and } \
    r_{n,m} \df \max_{\s
   \in S}r_{n,m,\s}.
   $$
   We claim that
\begin{equation}\label{eq: enumerate1}
  \limsup_{n \to \infty} r_{n,m} = \infty.
 \end{equation}
   To see this, let $a(\varrho) = \mrm{diag}(\varrho, \ldots, \varrho, 1)$, and for each $\s \in \Sue$, let
   $$\gfrak_\s = \ufrak_\s \oplus \ufrak_{\s,1}^\perp \oplus \ufrak_{\s,\varrho}^\perp$$
   be the decomposition of $\gfrak_\s$ into the eigenspaces for the adjoint action of $a(\varrho)^{-1}$, corresponding respectively to the eigenvalues $\varrho^{-1}, 1, \varrho. $
   Using  this direct sum decomposition we write
    $$v_m^{(\s)} = v_{m,1}^{(\s)}+ v_{m,2}^{(\s)} + v^{(\s)}_{m,3}.
   $$
Since $v_m \notin \wfrak_{\mrm{st}} $   we see from \eqref{eq: def wstfrak} that there is $\s \in S_{\mrm{ue}}$ such that at least one of
$v_{m,2}^{(\s)} $ and $v_{m,3}^{(\s)}$ is nonzero. 
For this choice of $\s$ and each $i$ we can write $\bar h_i^{(\s)}=
 a(\varrho) u_i$, 
 where $u_i \in \mrm{U}_{\s}$.
The adjoint action of $u_i^{-1}$ satisfies
$$\Ad_{u_i}^{-1}\left(v_{m,2}^{(\s)}\right) - v_{m,2}^{(\s)} \in \ufrak \ \  \text{ and }\ \ \Ad_{u_i}^{-1}\left(v_{m,3}^{(\s)} \right) - v_{m,3}^{(\s)} \in \ufrak \oplus \ufrak^\perp_{\s, 1}.$$
Using this one verifies that
$$\inf_n \left \|R\left(b_n^1 \right)^{-1}v_{m}^{(\s)} \right \|_\s>0;$$
indeed, the component in at least one of $\ufrak^\perp_{\s,1 }, \ufrak^\perp_{\s, \varrho}$ is not contracted by $R\left(b_n^1 \right)^{-1}$. Since $\s \in \Sue$, we have $|\varrho|_\s >1$. From this it follows that
$|\varrho|_\s^n \,\left\|R\left(b_n^1 \right)^{-1}v_{m}^{(\s)}  \right\|_\s \to \infty.$ This shows \eqref{eq: enumerate1}.

  Let $C>1$ be as in  Lemma~\ref{lem:NormGrowthBound}, let
  $$C_1 \df  \max\{ \lvert \varrho \rvert_\s \colon \s \in S\}, \quad C_2 \df
  \max\left\{ \left \lVert \Ad(\bar h_i^{-1}) \right \rVert_\s \colon \s \in S, \, \bar h_i \in
  \mrm{supp}\, \bar \mu \right\},$$
  and
  $$A' \df  \frac{\delta_0}{2C_1C_2C}.$$
 For each $n \in \N$ let $n_m \in \bN$ be the minimal index
  for which $A' \leq r_{n_m,m}$, which is well-defined in light of \eqref{eq: enumerate1}. It follows from the definition of $C_1$ and $C_2$ that $r_{n+1,m} \leq C_1C_2 r_{n,m}$, and hence
  \begin{equation*}
    \forall m \in \bN \quad A' \leq r_{n_m,m} \leq A'C_1C_2.
  \end{equation*}
Since $v_m \to 0$ we have  $n_m \overset{m \to \infty}{\longrightarrow} \infty$, and since the action of $F_{n,b}(a)$ on $\gfrak_\s$ is trivial for $\s \notin \Sue$, the place $\s$ for which the maximum in the definition of $r_{n_m, m}$ is attained,  belongs to $\Sue$ for all large enough $m$.

  Given $\s\in\Sue$ and $m,M \in \mathbb{N}$, let
  \begin{align*}
    \mrm{Sh}_{m,M}^\s(b) & \df \left\{
      a \in B
      \colon \mrm{dist}(\bQ_\s
      F_{n_m,b}(a)v_{m}^{(\s)} , \ufrak_{ \s}) \leq \frac{1}{M}
    \right\}, \\
    \mrm{Eq}_{m}^\s(b) & \df \left\{
      a \in B
      \colon \frac{1}{C} r_{n_m,m,\s} \leq \left\lVert
      F_{n_m,b}(a)v_{m}^{(\s)} \right\rVert_\s \leq C r_{n_m,m,\s}
    \right\}.
  \end{align*}
  Given $m\in\bN$, we let $M_{m,\s} \in \bN \cup \{\infty\}$ maximal such that
  \begin{equation*}
    \forall M \leq M_{m,\s}\quad
  \beta \big(\mrm{Sh}_{m,M}^\s(b)\big) > 1 - \varepsilon.
  \end{equation*}

  Since $\varepsilon$ was small, using
Lemmas~\ref{lem:NormGrowthBound} and ~\ref{lem: more linear2}, there exists $m_0 \in \bN$
  such that for every $m \geq m_0$ we can find a word
  \begin{equation*}
    a_m \in \bigcap_{\s \in \Sue}\big(\mrm{Sh}_{m,M_{m,\s}}^\s(b) \cap
    \mrm{Eq}_m^\s(b)\big)
  \end{equation*}
  such that $h_{n_m,z_0}(a_m), h_{n_m,z_m}(a_m) \in K$.
  Note that
  \begin{equation}\label{eq: growth of Mmsigma}
  \text{ for all } \s \in \Sue, \quad M_{m,\s}
  \overset{m \to \infty}{\longrightarrow} \infty. \end{equation}
  Hence, after passing to a subsequence, there is $u \in \gfrak$ such that
  $$h_{n_m,z_0}(a_m) \to z' \in K, \ \  h_{n_m,z_m}(a_m) \to z'' \in
  K, \ \ \ \text{ and } \ \
  F_{n_m,b}(a_m)v_m \to u.$$
  We claim that $u \in \ufrak_{\mrm{ue}}.$
  Indeed, for $\s \notin \Sue$ the action of $F_{n,b}(a)$ on $\gfrak_\s$ is trivial and hence the $\s$-component of $ F_{n_m,b}(a_m)v_m$ is equal to $v_m^{(\s)} $, which goes to zero as $m
\to \infty$. Furthermore, for $\s \in \Sue$, it follows from the definition of $\mrm{Sh}_{m,M}^\s(b) $ and from
\eqref{eq: growth of Mmsigma} that any limit of $ F_{n_m,b}(a_m)v_m$ belongs to $\ufrak_\s$ (and in particular, to the domain of definition of $\Psi$).
From the definition of $ \mrm{Eq}_{m}^\s(b)$ that for $\s \in \Sue$ for which the maximum in the definition of $r_{n_m, m}$ is realized, we have
$$0< \frac{A'}{C} \leq \lVert u^{(\s)} \rVert \leq \lVert u\rVert \leq A'C_1C_2C < \delta_0.$$
From the continuity of the map $\Psi $ we have
  \begin{equation*}
    z'' = \Psi_uz'.
  \end{equation*}
  Because $K$ is a continuity set, we have that
  \begin{equation*}
    \overline{\sigma}_{z'} = \overline{\sigma}_{z_0} \quad \text{ and }
    \overline{\sigma}_{z''} = \overline{\sigma}_{z_0}.
  \end{equation*}
  This proves~\eqref{eq:GoalInvariance}.
\end{proof}
In what follows, given $z \in B^X$, we denote by $W_z$ the stabilizer of
$\overline{\sigma}_z$ in $
\mrm{U}_{\mrm{ue}}$.
Combining  Proposition~\ref{prop:ExponentialDrift} and formula ~\eqref{eq:TranslationLeafwiseMeasure}, we obtain:

\begin{corollary}\label{cor:NonTrivialStabilizers}
  There exists $E \subset B^X$ of full measure such that $W_z$ is non-trivial
  for every $z \in E$.
\end{corollary}

\begin{lem}\label{lem:StabilizersLeafwiseMeasures}
There exists a conull subset $E \subset B^X$ such that:
  \begin{enumerate}
    \item \label{item:SLMEquivariance} If $z = (b,x)\in E$ and $T^Xz \in E$,
      then $W_{T^Xz} = b_1^{-1}W_z b_1$ and $\overline{\sigma}_{T^Xz} =
      \overline{(\Ad b_1^{-1})_\ast \sigma_z}$.
    \item \label{item:SLMProduct} For every $z \in E$, we have that
      \begin{equation*}
        W_z = \prod_{\s \in \Sue} W_{z,\s},
      \end{equation*}
      where $W_{z,\s} = \exp(\wfrak_{z,\s})$ for a Lie subalgebra $\wfrak_{z,\s}
      \subset \ufrak_\s$.
    \item \label{item:SLMInvariance} For every $z \in E$, $\sigma_z$ is
      $W_z$-invariant.
  \end{enumerate}
\end{lem}
\begin{proof}
  The proof is mostly a combination of facts about closed subgroups of unipotent
  $S$-arithmetic groups and techniques used in~\cite[\textsection8.2]{BQ2}.
  Item~\eqref{item:SLMEquivariance} is an immediate consequence
  of
  the characterizing properties of leafwise
  measures.
  For item~\eqref{item:SLMProduct}, we know from Corollary \ref{cor:NonTrivialStabilizers} that
  for typical $z \in B^X$ the stabilizer $W_z$ of $\overline{\sigma}_z$ is
  non-trivial.
  Let $u$ be a nontrivial element of $W_z$. Then the logarithm $v =
  \log u
  \in  \ufrak_\mrm{ue}$ is well-defined and
  non-zero, and we write $v = (v^{(\s)})_{\s \in \Sue}$.
  It follows from strong approximation (see~\cite[Ch.~3, Lem.~3.1]{Cassels}) that
  \begin{equation*}
    \overline{\bZ v} = \bigoplus_{\s \in \Sue}\bZ_\s v^{(\s)}.
  \end{equation*}
  In particular, for any element in $W_z$, each
  component fixes $\overline{\sigma}_z$ separately, i.e., $W_z$ is a direct
  product of subgroups of $\mrm{U_{ue}}$. We denote these groups by $W_{z, \s}$, and note that each of the $W_{z,\s}$ is closed, since $W_z$  is closed.

  Let $\wfrak_{z,\s} \subset \ufrak_{\mrm{ue}}$ be the Lie algebra generated by $\log(W_{z,\s})$ and
  note that, since $W_z$ is abelian,
  \begin{equation*}
    \wfrak_{z,\s} = \mrm{span}_{\bQ_\s} \log(W_{z,\s}).
  \end{equation*}
  By the preceding argument $\wfrak_{z,\s} = \{0\}$ if and only if $W_{z,\s}$ is trivial.
  We claim that
  \begin{equation*}
    \exp(\wfrak_{z,\s}) = W_{z,\s}.
  \end{equation*}
 This certainly holds if $W_{z, \s}$ is trivial.  By construction,  $W_{z,\s} \subset \exp(\wfrak_{z,\s})$. For the
  opposite inclusion, we use the argument given in the proof
  of~\cite[Lem.~8.3]{BQ2}. We suppose by contradiction that the opposite inclusion does not hold for some  $z \in B^X$, and let
  \begin{equation*}
    \varphi_\s(z) = \inf \left\{ \left\lVert v^{(\s)} \right \rVert \colon v^{(\s)} \in \wfrak_{z,\s},
      \, \exp(v) \not\in W_{z,\s}
    \right\}.
  \end{equation*}
  Then there is a subset of $B^X$ of positive measure on which $\varphi_\s> 0$.
  By the definition of $\Sue$ we have that $\lvert \varrho \rvert_\s>
  1$ for every $\s \in \Sue$, and this implies that for every $b \in B$ we have that $\left \lVert
  (b_1^n)^{-1}|_{\ufrak_\s} \right \rVert \to 0$ as $n \to \infty$. By
  item \eqref{item:SLMEquivariance} we have  that $\varphi_\s ((T^X)^n(z))$ tends to zero for $\beta^X$-a.e. $z$, and this contradiction to
  the Poincar\'e
  recurrence theorem proves the claim.

  In order to prove item~\eqref{item:SLMInvariance}, note that, since $\ufrak_{\mrm{ue}}$ is
  abelian and since the action of $\ufrak$ on Radon measures on $\mrm{U_{ue}}$ is
  continuous, for every $\s \in \Sue$ the map $\alpha_{z,\s} \colon \wfrak_{z,\s} \to
  \bR$ given by
  \begin{equation*}
    \forall v_\s \in \wfrak_{z,\s} \quad {\tau_{v_\s}}_\ast \sigma_z =
    e^{\alpha_{z,q}(v_\s)} \sigma_z
  \end{equation*}
  is a continuous group homomorphism. Since $\wfrak_{z,\s}$ is a $\Q_\s$-vector
  space, $\alpha_{z,\s}$ is trivial.
\end{proof}
\begin{proof}[Proof of Theorem~\ref{thm:AdditionalInvariance}]
  For the proof we disintegrate each measure $\nu_b$ along the map that sends $z = (b,x)$ to the Lie algebra $\wfrak_z = \Lie(W_z)$. This is done as in Propositions~7.5 and~7.6
  in~\cite{BQ1}. Note that \cite{BQ1} treats actions of $\R^d$. Recall that we showed in the proof of Lemma \ref{lem:StabilizersLeafwiseMeasures} that if $v \in
  \ufrak_{\mrm{ue}}$ satisfies that $\exp(v)$ fixes a
  leafwise measure, then so does $\exp(v')$
  for any $v' \in \mrm{span}_{\Q_\s}(v).$
  Since the exponential map is a bijection between $\ufrak_\mrm{ue}$ and $\mrm{U_{ue}}$, all the statements in \cite[\textsection4]{BQ1} adapt mutatis mutandi to our situation.
  \end{proof}

\section{Ratner's Theorem and Concluding the Proof of Theorem \ref{thm: stationary main}}\label{sec:endgame}
Armed with the additional invariance of limit measures by unipotents, provided by Theorem \ref{thm:AdditionalInvariance}, and following the ideas of \cite{BQ1}, we will use Ratner's theorem to reduce the analysis of stationary measures on $\XS$ to the analysis of stationary measures on the projectivization of an exterior power of $\mathfrak{g}$. 
A description of such stationary measures is provided in \textsection\ref{subsec: exterior powers}.
The proof of Theorem \ref{thm: stationary main} is then concluded in \textsection\ref{subsec:conclude stationary main}.

\subsection{Stationary measures on projective spaces}\label{subsec: exterior powers}
In this section we give a description of stationary measures under the action of the random walk on projective spaces in linear representations of $\gG_\s, \s\in \Sue$; Propositions \ref{prop:measures on proj spaces} and \ref{prop:measures on proj spaces2}.
The proof uses a lemma showing that a $\s$-adic analog of the fractal measure associated to the IFS $\Phi$ gives $0$ mass to proper subvarieties of $\Q_\s^d$; Lemma \ref{lem:zero mass on proper subvarieties}. The latter is used to show that vectors grow uniformly under the random walk (Lemma \ref{lem:uniform growth in irreps}) to conclude that trajectories converge towards a (deterministic) subspace of maximal expansion.

\subsubsection{Zero mass on proper subvarieties}
Let $\s \in \Sue.$
Using \eqref{eq: useful formula and correct}, for every $b\in B$ and $n\in \N$, we write $b_1^n = a(\varrho^n) u_n(b)$ where 
\begin{equation}
    \label{eq: useful formula and correct2}
a(\varrho) \df \begin{pmatrix}
        \varrho \mrm{Id}_d & 0 \\
        0 & 1
    \end{pmatrix} \ \ \text{ and } \ \ u_n(b) \df \begin{pmatrix}
    \mathrm{Id}_d & -\varrho^{-n}\mathbf{y}_1^n \\ 0 & 1 
        \end{pmatrix}
    \in \gU_\s.
\end{equation}

We begin with the following preliminary algebraic lemma.
\begin{lem}
    \label{lem:Zariski closure and finite index}
    Let $\s\in S$, and let $\Gamma_\sigma$ be the projection of $\Gamma_{\bar{\mu}}$ to $\gG_\sigma$.
    Then, the following hold.
    \begin{enumerate}
        \item \label{item: Zariski closure of a(rho)} For every $\s\in \Sdt\cup \Sue$, and every $k\in \Z\smallsetminus\{0\}$, the $\Q_\s$-Zariski closure of the cyclic group $\langle a(\varrho^k)\rangle$ is the torus $T \df \{a(t): t\in \Q_\s^\ast\}$, where $\Q_\s^\ast$ is the multiplicative group of $\Q_\s$.

        \item \label{item:Zariski closure of Gamma_sigma contains U_sigma} For every $\s\in \Sue$, the $\Q_\s$-Zariski closure of $\Gamma_\sigma$ contains $\gU_\sigma$.

        \item \label{item:Zariski closure of finite index subgroups of Gamma_sigma} For every $\s\in S$, $\Gamma_\sigma$ has the same $\Q_\s$-Zariski closure as any of its finite index subgroups.
    \end{enumerate}

\end{lem}
\begin{proof}
    For Part \eqref{item: Zariski closure of a(rho)}, note that $|\varrho|_\s \neq 1$ for $\s\in \Sdt\cup\Sue$, and hence the cyclic group $\langle a(\varrho^k)\rangle$ is infinite whenever $k\neq 0$.
    Moreover, its Zariski closure is isomorphic to a Zariski closed subgroup of the multiplicative algebraic group $T\cong \mathbb{G}_\mrm{m}\cong \Q_\s^\ast$, and
    the only proper Zariski closed subgroups of $\mathbb{G}_{\mrm{m}}$ are finite. This follows by Zariski connectedness of $\mathbb{G}_{\mrm{m}}$, or alternatively, by the correspondence between such algebraic subgroups and quotients of the group of characters $X(\mathbb{G}_{\mrm{m}})\cong\Z$; cf.\ \cite[Section 8.2]{Borel}. This proves Part \eqref{item: Zariski closure of a(rho)}.

    For Part \eqref{item:Zariski closure of Gamma_sigma contains U_sigma}, let $\s\in \Sue$, and denote by $H_1$ the Zariski closure of $\Gamma_\s$.
    Then, a straightforward computation shows that for all pairs of indices $i, j$, $h_i \circ  h_j^{-1}=u(\varrho^{-1}(\mathbf{y}_i -\mathbf{y}_j)). $  It follows from the irreducibility of the IFS $\Phi$ (as in the proof of Proposition \ref{prop: projection Z trivial}) that the Zariski closure of the group generated by $\{u(\varrho^{-1}(\mathbf{y}_i - \mathbf{y}_j)): i,j=1, \ldots, k\}$ coincides with $\gU_\s$, and hence Part \eqref{item:Zariski closure of Gamma_sigma contains U_sigma} follows.

    For 
    Part \eqref{item:Zariski closure of finite index subgroups of Gamma_sigma}, note that its assertion is vacuous for $\s\in \Str$.
    Note further that Part \eqref{item: Zariski closure of a(rho)} implies Part \eqref{item:Zariski closure of finite index subgroups of Gamma_sigma} for $\s\in \Sdt$.
    Now let $\s\in \Sue$.
    Let $\Gamma_0$ be a finite index subgroup of $\Gamma_\sigma$, and denote by
     $H_0$ its Zariski closure.
        We claim that $\gU_\sigma$ is contained in $H_0$.
     Indeed, note that $H_0$ is a finite index subgroup of $H_1$. This follows by writing $\Gamma_\sigma= \cup_i \g_i \Gamma_0$ as a finite disjoint union of cosets of $\Gamma_0$ and taking Zariski closure of both sides of this equation.
     In particular, $H_1=\cup_i \g_i H_0$, where after possibly passing to a subset of $\{\g_i\}$,m the right side is a finite disjoint union of proper subvarieties. 
      By Part \eqref{item:Zariski closure of Gamma_sigma contains U_sigma} and irreducibility of the affine algebraic group $\gU_\sigma\cong \Q_\sigma^d$, we see that $\gU_\sigma \subseteq \g_i H_0$ for some $i$. Since $\gU_\sigma$ contains the identity, we see that in fact $\gU_\sigma$ is contained in $H_0$.

      Finally, we note that any element $\g$ of $\Gamma_0\subseteq \Gamma_\sigma$ can be written as $\g=a(\varrho^k)u$ for some $k\in\Z$ and $u\in \gU_\sigma$.
      It follows that $a(\varrho^k)\in H_0$ for some $k\neq 0$, and hence the Zariski closure of the cyclic group $\langle a(\varrho^k)\rangle$ is also contained $H_0$. Hence, by Part \eqref{item: Zariski closure of a(rho)}, $H_0$ contains all diagonal elements of the form $a(\varrho^n)$ for any $n\in \Z$, and thus $H_0$ contains all of $\Gamma_\sigma$. It follows that $H_0=H_1$ as desired.
\end{proof}

The following lemma is the main result of this subsection.
\begin{lem}\label{lem:zero mass on proper subvarieties}
For every $b= (i_1, i_2, \ldots)$, the limit $u_\infty(b)\df\lim_{n\to\infty} u_n(b)$ exists in $\gU_\s$, and is equal to 
$u_\infty(b)= u(\mathbf{y_\infty}(b))$, where 
\begin{align}\label{eq: define y infinity}
    \mathbf{y}_\infty(b) = -\sum_{n=1}^\infty \varrho^{-n} \mathbf{y}_{i_n}.
\end{align}
In particular, the set $\{u_n(b):n\in\N, b\in B\}$ is bounded.
The map $u_\infty: B \to \mrm{U}_\s$ satisfies 
\begin{equation}\label{eq: equivariance property for varieties}
u_\infty(ib) = u(-\varrho^{-1} \mathbf{y}_i)\, \Ad\left(\bar h_i^{-1} \right)(u_\infty(b)),
\end{equation}
and for any proper subvariety $\Ucal \subset \gU_\s$ we have $u_\infty(b)_*\beta(\Ucal)=0$. 
\end{lem}

\begin{proof} 
Define $\mathbf{y}_\infty(b) $ via \eqref{eq: define y infinity}. This is a geometric series, and since $\|\varrho\|_\s>1$ for $\s \in S_{\mrm{ue}}$, it converges. Using the easily verified identity $\mathbf{y}_1^n = \varrho \mathbf{y}_1^{n-1}+\mathbf{y}_{i_n}$, we see that the  $n$th partial sum of this series is equal to $-\varrho \mathbf{y}_1^n$. Since $u(-\varrho \mathbf{y}_1^n)=u_n(b)$ by \eqref{eq: useful formula and correct2}, this proves the first assertion.  
Boundedness follows from the definition of $\mathbf{y}_\infty(b)$.

It is clear from the definition that 
\begin{equation*}\label{eq: equivariance for the map}
\mathbf{y}_\infty(ib) = - \varrho^{-1} \mathbf{y}_i + \varrho^{-1}\mathbf{y}_\infty(b)  \ \ \text{ (for $i=1, \ldots, k$).}
\end{equation*}
This implies \eqref{eq: equivariance property for varieties}.

For the last assertion we follow standard arguments,  see e.g. \cite[Chap. 3]{BQbook}.
Let $\hat h_i : \mathfrak{u}_\sigma \to \mathfrak{u}_\s$ be the affine map appearing in 
\eqref{eq: equivariance property for varieties}, that is, 
$$\hat h_i (x) \df \varrho^{-1}\left( x -\mathbf{y}_i\right).$$ 
Identifying $\mathfrak{u}_\s$ with $\mrm{U}_\s$ via the exponential map, 
 we get an affine random walk on $\mrm{U}_\s$ induced by the maps $\hat h_i$ and the measure $\hat \mu = \sum p_i \delta_{\hat h_i}$. 

Suppose by contradiction that for some subvariety $\Ucal \subset \gU_\s,$ we have that there is a subset of $b$, of positive measure with respect to $\beta$, for which $u_\infty(b) \in \Ucal$. 
Let $\theta_b$ denote the Dirac measure  on $u_\infty(b)$ and 
let
$$\theta \df \int \theta_b\, \der \beta(b).$$ It follows  from Proposition \ref{prop: commutes with RW} and \eqref{eq: equivariance property for varieties} that $\theta$ is $\hat \mu$-stationary, and by our assumption, $\theta(\Ucal)>0$. 
Let $\Ucal_0$ be an irreducible subvariety of minimal dimension for which $\theta(\Ucal_0)>0$, and assume further that $\theta(\Ucal_0)$ is the maximal mass assigned by $\theta$ to irreducible subvarieties in this dimension.  
It follows from the equivariance property $u_\infty(ib) = \hat h_i( u_\infty(b))$, maximality of the mass, and from stationarity that $\theta\left(\hat h_i \left (\Ucal_0 \right) \right) = \theta(\Ucal_0) >0$ for each $i$. 
By  minimality of dimension, and by irreducibility of $\Ucal_0$, for each $i$ and $j$, the images  $\hat h_i(\Ucal_0)$ satisfy that either $\theta \left(\hat h_i \left(\Ucal_0 \right) \cap \hat h_j \left(\Ucal_0\right)\right)=0$ or $\hat h_i(\Ucal_0) = \hat h_j(\Ucal_0)$. This means that the $\hat \mu$ random walk induces a permutation action on a finite collection of such subvarieties. 

Let $\hat \Gamma \subset \mathrm{Aff}(\mathfrak{u}_\s)$ be the subgroup of affine maps of $\mathfrak{u}_\s$ generated by the maps $\hat h_i$. Then there is a finite index subgroup $\hat \Gamma_0 \subset \hat \Gamma$ which preserves $\Ucal_0.$ By Lemma \ref{lem:Zariski closure and finite index}\eqref{item:Zariski closure of finite index subgroups of Gamma_sigma}, the Zariski closure of $\hat \Gamma_0$ in $\mrm{Aff}(\mathfrak{u}_\s)$ coincides with the Zariski closure of 
$\hat \Gamma$ and contains the full group of translations $\mathfrak{u}_\s$.
Thus $\Ucal_0$ is invariant under translation by all elements of $\mathfrak{u}_\s$. This implies that $\Ucal = \mathfrak{u}_\s$, so there is no proper subvariety of $\mathfrak{u}_\s$ which is of positive measure with respect to $\theta$.
\end{proof}

\subsubsection{Uniform growth and stationary measures on projective spaces}
Given $\s\in S$,
let $\bar{\mu}_\s$ be the projection of $\bar\mu$ to $\mrm{G}_\sigma$ and denote by $\Gamma_{\bar{\mu}_\s}$ the subgroup of $\mrm{G}_\sigma$ generated by its support. 
The main result of this subsection is the following statement. 

\begin{prop}\label{prop:measures on proj spaces}
    Let $\s\in \Sue$ and let $V$ be a 
     $\Q_\s$-algebraic 
    representation of $\mrm{G}_\sigma$. 
    Then, every $\bar{\mu}_\s$-stationary measure $\eta$ on $\mathbb{P}(V)$ is a Dirac mass on a line through an eigenvector of $a(\varrho)$ which is fixed by $\gU_\s$. In particular, the support of $\eta$ is a fixed point of $\Gamma_{\bar \mu_\s}$.


    \end{prop}

 Given a $\Q_\s$-algebraic representation $\mrm{G}_\s \to \mathrm{GL}(V)$, given $a \in \mrm{G}_\s$, and given an eigenvalue $\lambda \in \Q_\s$ for the action of $a$ on $V$, we say that $\lambda$ is {\em maximal in $V$} if for any eigenvalue $\lambda'$ for $a$ in $V$, we have $|\lambda|_\s \geq |\lambda'|_\s$.

\begin{lem}\label{lem: unique maximal eigenvalue}
    Let $V$ be a $\Q_\s$-algebraic representation of $\gG_\sigma$. Then, there is a unique maximal eigenvalue in $V$  for the action of $a(\varrho)$. If this maximal eigenvalue is of modulus 1, then $V$ is the trivial representation. 
\end{lem}

\begin{proof}
By Lemma \ref{lem:Zariski closure and finite index}\eqref{item: Zariski closure of a(rho)}, the $\Q_\sigma$-Zariski closure of the cyclic group $\langle a(\varrho)\rangle$ is the torus $T= \{a(t): t \in \Q^\ast_\sigma\}$. In particular, eigenvalues of $a(\varrho)$ are restrictions of joint eigenvalues of $T$ in $V$, i.e., of algebraic characters of $T$ to $a(\varrho)$. Moreover, the group $\mrm{Hom}(T,\Q_\s^\ast)\cong \Z$ of algebraic characters consists of homomorphisms of the form $a(t) \mapsto t^n$, for some integer $n\in\Z$. Hence, since $|\varrho|_\sigma \neq 1$, this implies uniqueness of the eigenvalue of maximal modulus. 
If this maximal value is of modulus one, then the action of $a(\varrho)$  $V$ is 
trivial, and thus
$a(\varrho)$ lies in the kernel of the homomorphism $\gG_\sigma\to \mrm{PGL}(V)$. Since $\gG_\s$ is simple, the kernel of this homomorphism is either finite or all of $\gG_\s$. Since $\s\in \Sue$, $|\varrho|_\s\neq 1$ and hence $a(\varrho)$ generates an infinite cyclic group. It follows that all of $\gG_\s$ acts trivially on $V$.
\end{proof}

We denote the set of maximal eigenvalues in $V$ by $\Lambda_{\max}(V)$, the eigenspace corresponding to an eigenvalue $\lambda$ by $V^{\lambda}$, let $\lambda_{\max}(V) = |\l|_\s$ for some (any) $\l \in \Lambda_{\max}(V)$, and write 
$$V^{\max} \df \bigoplus_{\l \in \Lambda_{\max}(V)}V^\l.$$
We will show: 
\begin{prop}\label{prop:measures on proj spaces2}
    Let the notation be as in Proposition \ref{prop:measures on proj spaces}. 
   Let $V=\oplus_j V_j$ be a decomposition of  $V$ into $\mrm{G}_\sigma$-irreducible representations,
    and let
    $$V_\star \df \bigcup_{\delta >0} \bigoplus_{j:\l^{\max}(V_j)=\delta}   V_j^{\max}.$$
    Then every $\bar{\mu}_\s$-stationary measure on $\mathbb{P}(V)$ is supported on $\P(V_\star)$.
\end{prop}

\begin{proof}[Proof of Proposition \ref{prop:measures on proj spaces} assuming Proposition \ref{prop:measures on proj spaces2}]
We first note that each $v \in V_\star$ is fixed by the group $\gU_\s$. Indeed, it suffices to show this for $v \in V^{\l}_j$ for some irreducible component $V_j$ and some $\l \in \Lambda_{\max}(V_j).$ In this case, the fact that $\gU_\s$ acts trivially on $V_j^\l$ follows since $\gU_\s$ is the expanding horospherical subgroup of $\gG_\s$ associated to $a(\varrho)$, i.e.,
 since $\mathfrak{u}_\s $ is the eigenspace of maximal eigenvalue for $\Ad(a(\varrho))$ acting on $\mathfrak{g}_\s$.
Since each $\bar h_i$ is of the form $a(\varrho)u_i$ for some $u_i \in \gU_\s$, it follows that the action of the $\bar \mu_\s$ random walk on $\P(V_\star)$ factors through the (deterministic) $\Z$-action generated by the map induced by $a(\varrho)$. Thus any stationary measure is $a(\varrho)$-invariant. 
Finally, by Lemma \ref{lem: unique maximal eigenvalue}, since $a(\varrho)$ acts by multiplication by a scalar of modulus $\delta$ on $\bigoplus_{\lambda^{\max}(V_j) = \delta} V_j^{\max}$, we see that $a(\varrho)$ fixes $\P(V_\star)$ pointwise. Hence, every $a(\varrho)$-ergodic invariant measure is a Dirac mass on a fixed point by $a(\varrho)$ and $\gU_\s$, and hence also by the group
$\Gamma_{\bar{\mu}}$. This completes the proof.
\end{proof}

We will denote the action of $g \in \mrm{G}_\s$ on $V$ by $(g,v) \mapsto g \cdot v$. 
Denote 
$$V^{<\max} \df \bigoplus_{\lambda' \notin \Lambda_{\max}(V)} V^{\lambda'},$$
and let $\pi_{\max}$ denote the projection to $V^{\max}$ parallel to $V^{<\max}$.
As a first step toward the proof of Proposition \ref{prop:measures on proj spaces2}, we give the following lemma regarding uniform growth in irreducible representations of $\mrm{G}_\sigma$. Its proof adapts ideas from \cite{SimmonsWeiss}.

\begin{lem}\label{lem:uniform growth in irreps}
     Let $\s\in \Sue$ and let $V$ be a 
     representation of $\mrm{G}_\sigma$, endowed with a norm $\norm{\cdot}$.
     Then, for every non-zero $v\in V$, we have that for $\b$-almost every $b$, 
     \begin{align*}
         \liminf_n \norm{\pi_{\max}(u_n(b)\cdot v)} >0.
     \end{align*}
\end{lem}
\begin{proof}
By decomposing $V$ into irreducible subrepresentations, we may assume that $V$ is irreducible. 
Fix a non-zero $v\in V$. 
 Let
  $$\Ucal_v \df \left\{u\in \gU_\s :  u\cdot v \in V^{<\max}\right \}.$$
Since $u_n(b) \to u_\infty(b)$, it suffices to show that 
for $\beta$-a.e. $b \in B$, $u_\infty(b) \notin \mathcal{U}_v.$
In light of Lemma \ref{lem:zero mass on proper subvarieties}, it suffices to show that $\mathcal{U}_v$ is a proper subvariety of $\mrm{U}_\s.$

To see this, note first 
that  $\Ucal_v$ is characterized by vanishing of the polynomial map $u\mapsto \pi_{\max}(u\cdot v)$, and, hence, it is a subvariety of $\mrm{U}_\s$.
Now
suppose for a contradiction that $\gU_\s \cdot v \subset V^{<\max}$. 
Note that the adjoint action of $\Ad(a(\varrho))$ on $\mathfrak{g}_\s$ has three eigenvalues $\lambda^-, \lambda^0, \lambda^+$, satisfying 
$$
|\lambda^-|_\s < 1 = \lambda^0 < |\lambda^+|_\s,
$$
and $\mathfrak{u}_\s$ is the eigenspace for $\lambda^+$.
Let 
$$\gP^-_\s = \left\{g\in \mrm{G}_\sigma: \{a(\varrho^n) g a(\varrho^{-n}) : n\geq 1\} \text{ is bounded} \right \}.$$
That is, its Lie algebra $\mathfrak{p}^-_\s$ is the sum of the $\lambda^-$ and $\lambda^0$ eigenspaces. 
This implies that $V^{<\max}$ is $\gP^-_\s$-invariant, and hence $\gP^-_\s \gU_\s \cdot v \subset V^{<\max}$. But $\mathfrak{u}_\s \oplus \mathfrak{p}^-_\s = \mathfrak{g}_\s$ and thus $\gP^-_\s \gU_\s$ is Zariski-dense in $\mrm{G}_\sigma$. We deduce  $\mrm{G}_\sigma \cdot v \subset V^{<\max}$. In particular, the span of the orbit $\mrm{G}_\sigma \cdot v$ is a proper invariant subspace of $V$, which contradicts irreducibility of $V$.
\end{proof}

\begin{proof}[Proof of Proposition \ref{prop:measures on proj spaces2}]
Let $\eta$ be an ergodic stationary measure on $\P(V)$ and let $v \in V.$
Let $V = \bigoplus_j V_j$ be a decomposition of $V$ into irreducibles, let $\pi_j: V \to V_j$ be the corresponding projections, and let
$$\mathcal{J}(v) = \{j : \pi_j(v)\neq 0\}.$$
Finally let 
$$\l_{\max} (v)\df \max \{\l_{\max}(V_j) : j \in \mathcal{J}(v)\}, \ \ \ \ \ \ V_{\mathcal{J}} \df \bigoplus_{j \in \mathcal{J}(v)} V_j.$$  
 Writing $b_1^n=a(\varrho^n)u_n(b)$ as before, denoting by $\| b_1^n\|$ the operator norm of $b_1^n$ acting on $V_{\mathcal{J}}$, and writing $\lambda_{\max} = 
 \lambda_{\max}(v)$, 
 we have from Lemma \ref{lem:zero mass on proper subvarieties} that the set $\seti{u_n(b):b\in B, n\in\N}$ is pre-compact in $\gU_\s$, and hence 
 that $\norm{b_1^n}\asymp \l_{\max}^n$ uniformly in $n$ and $b$.
Moreover, 
    \begin{align*}
        b_1^n \cdot v = \l_{\max}^n \pi_{\max}(u_n(b)v) + o(\l_{\max}^n).
    \end{align*}
This estimate together with Lemma \ref{lem:uniform growth in irreps} imply that all projective accumulation points of $(b_1^n \cdot v)$ belong to $\mathbb{P}(V_{\mathcal{J}}^{\max}) \subseteq \P(V_\star)$. It follows that the limit measures $\eta_b$ of $\eta$, defined in \eqref{eq: def nu b}, are supported on $\mathbb{P}(V_\star)$ almost surely.
From \eqref{eq: stationarity equivariance2}, we conclude that $\eta$ is supported on $\mathbb{P}(V_\star)$.
\end{proof}

As in the proof of Proposition \ref{prop:measures on proj spaces}, since the $\bar \mu$ random walk in the $\Sdt$ places is actually deterministic, and consists of the $\Z$-action induced by the action of $a(\varrho)$, and since the random walk in the $\Str$ places is trivial, we have the following. 

\begin{lem}
    \label{lem:meas on proj spaces Sdt}
    Let $\s\in \Sdt \cup \Str$ and let $V$ be a representation of $\mrm{G}_\sigma$.
        Then, every $\bar{\mu}_\s$-stationary measure $\eta$ on $\mathbb{P}(V)$ is a Dirac mass on a line through an eigenvector of $a(\varrho)$. In particular, the support of $\eta$ is a fixed point of $\Gamma_{\bar \mu_\s}$.
\end{lem}

\subsection{Concluding the proof}\label{subsec:conclude stationary main}

In this section, we complete the proof of Theorem \ref{thm: stationary main}. We recall the statement of the theorem for the reader's convenience.

\begin{thm-nonumber}
    Let $\Phi$ be an irreducible carpet IFS satisfying the open set condition, let $\bar h_i$ be as in \eqref{eq: def h'},  let $\mathbf{p}$ be a probability vector, let $\bar \mu$ be as in \eqref{eq: def mu bar}, and let $\Wst$ be as in \eqref{eq: def Wst}. Then for any $\bar \mu$-ergodic $\bar \mu$-stationary measure $\nu$, one of the following holds:
    \begin{enumerate}
    \item  \label{item: case 2 repeat}
    $
    \nu\left( \left\{x \in \X_{d+1}^S: \mathrm{Stab}(x) \cap W^{\mathrm{st}} \neq \{\mathrm{Id} \} \right\} \right) =1.$
    \item \label{item: case 1 repeat}
    there is $g_0 \in G^S$ and  a closed subgroup $M \subset G^S,$  such that for $x_0 = g_0 \Lambda^S \in \X_{d+1}^S$, 
    we have:
    \begin{enumerate}[(A)]
    \item \label{item: A list repeat} $\nu$ is $M$-invariant;
    \item \label{item: B list repeat} $Mx_0$ is a closed orbit and $M \cap \mathrm{Stab}(x_0)$ is a lattice in $M$; 
        \item \label{item: B1 list repeat}
    The conjugate $g_0^{-1} M g_0$ is of finite index in $\mathbf{M}(\Q_S)$, where $\mathbf{M}$ is a $\Q$-algebraic subgroup of $\mathbf{M}$;
    \item \label{item: C list repeat}
    $N_1(M)x_0$ is a closed orbit satisfying
    $\nu(N_1(M)x_0)=1$; 
    \item \label{item: D list repeat}
    $\bar h_i \in N_1(M)$ for each $i$; 
    \item \label{item: E list repeat} the element $b$ of \eqref{eq: def a2} normalizes $M$, and we have $b_* m_{M}=c\, m_{M}$ for $c<1$.
    \end{enumerate}
    \item \label{item: case 3 repeat} $\nu = m_{\X_{d+1}^S}.$
    \end{enumerate}
\end{thm-nonumber}

\begin{proof}[Proof of Theorem \ref{thm: stationary main}]
We will assume that $\nu$ is an ergodic stationary measure for the $\bar \mu$ random walk, which satisfies the 
hypotheses of Theorem~\ref{prop:NonAlignmentLimitMeasures} and is not equal to $m_{\XS}$, 
and conclude that the statements in item \eqref{item: case 1 repeat} of Theorem \ref{thm: stationary main} hold.



Let $\nu_z, W_z$ be as in the conclusion of Theorem \ref{thm:AdditionalInvariance}, where $z = (b,x) \in B \times \XS$.
Note that for every $\s\in S$, $b\in B$ and $n\in \N$, the action of $\Ad(b_1^n)$ on $\mathfrak{u}_{\mrm{ue}} = \Lie (\gU_\s)$ is given by multiplication by a scalar.
In particular, $b_1^n$ normalizes every subgroup of $\gUue$.
Hence, the equivariance property of the groups $W_z$ in \eqref{eq:equivariance of W_z} and ergodicity imply that $W_z$ is constant almost surely.
Denote by $W$ the almost sure common value of these groups.
It follows that $\nu_z$ is $W$-invariant almost surely, and hence 
$\nu$ is $W$-invariant.
Since the group $\Gamma_{\bar\mu}$ is generated (as a group) by the elements $b_1^n$, we have from  
\eqref{eq:equivariance of W_z} that  
$\Gamma_{\bar\mu}$ normalizes $W$.

Let us write $\nu= \int_{\XS}\nu_x \, \der \nu(x)$ for the decomposition of $\nu$ into its $W$-ergodic components $\nu_x$.
By Ratner's theorem \cite{Ratnerpadic}, for $\nu$-almost every $x \in \XS$, there is a group $M(x)$ containing $W$ such that the orbit $M(x) \, x$ is closed and equal to $\overline{Wx}$.
Moreover, by \cite{Tomanov}, if $g_0\in G^S $ satisfies $x = g_0\Gamma^S$ then the Zariski closure $\mathbf{M}$ of the group  $M'(g_0,x) \df g_0^{-1}M(x)g_0$ is a $\Q$-algebraic subgroup of $ \mathbf{G} $ and  $M'(g_0,x)$ is a finite-index subgroup of  $\mathbf{M}(\Q_S)$.

We denote by $\mathfrak{m}(x)$ the Lie algebra of $M(x)$, and claim that for $\nu$-almost every $x \in \XS$ and every $g \in \mathrm{supp} (\bar \mu^{*n})$, we have 
$\mathfrak{m}(gx) = \Ad(g)\mathfrak{m}(x) .$
Indeed, since $\Gamma_{\bar \mu}$ normalizes $W$, for $g 
\in \Gamma_{\bar \mu}
$
and $\nu$-almost every $x$, we have 
$$g^{-1}M(gx)gx = g^{-1}\overline{Wgx} = \overline {g^{-1}Wgx}= \overline{Wx}=M(x)x. $$
Thus the groups $M(gx)$ and $gM(x)g^{-1}$ coincide in a neighborhood of the identity, and the claim follows. 

We can assign to $M(x)$  the dimension $\dim (\mathbf{M})$, where $\mathbf{M}$ is the algebraic group above. Then $\mf{m}(x) = \prod_{\s \in S} \mf{m}_\s(x)$, where $\mf{m}_\s$ is the Lie algebra of $M_\s$, and $\mf{m}_\s(x)$ is a vector space over $\Q_\s$ of dimension $\dim (\mathbf{M})$. We denote this dimension by $\dim (\mf{m}(x)). $
Then $\dim (\mf{m}(x) )= \dim (\Ad(g) \mf{m}(x)) =\dim( \mf{m}(gx))$ for $g \in \supp(\bar \mu^{*n})$, and hence 
by ergodicity, $\dim (\mathfrak{m}(x))$ is a.e.-constant, and we denote its almost sure value by $k_0$. 
For each $\s \in S$, we define a measure $\theta $ on $\mathbb{P}(\bigwedge^{k_0}\mathfrak{g}_\s)$ by $\theta \df \int \delta_{[\mathfrak{m}(x)_\s]} \, \der \nu(x),$
where $[\mathfrak{m}(x)_\s]$ is the line spanned by $\bigwedge^{k_0} \mathfrak{m}(x)_\s $ in $\bigwedge^{k_0} \mathfrak{g}_\s.$ Then, by Proposition \ref{prop: commutes with RW},  $\theta$ is stationary and ergodic for the $\bar \mu$ random walk induced by the action of $\bigwedge^{k_0}\Ad(\bar h_i)$ on $\mathbb{P}(\bigwedge^{k_0} \mathfrak{g}_\s)$. By Proposition \ref{prop:measures on proj spaces} and Lemma \ref{lem:meas on proj spaces Sdt}, 
$\theta$ is supported on $\Gamma_{\bar \mu}$ fixed points. This implies that $\Gamma_{\bar \mu}$ normalizes $\mathfrak{m}(x)_\s$ for every $\s \in S$ and $\nu$-a.e. $x$, and hence normalizes $M(x)$ for $\nu$-a.e. $x$.

By ergodicity, this implies that $M(x)$ is a.e.\ constant, and we denote the almost sure value of $M(x)$ by $M$. Choosing $x_0 = g_0 \Lambda^S$ in the set of full measure for which $M(x_0) = M$ and $Mx_0 = \overline{Wx_0}$ is a finite volume homogeneous subspace of $\X$, we see that assertions \eqref{item: A list repeat}, \eqref{item: B list repeat} and \eqref{item: B1 list repeat} hold. Furthermore, we have obtained that each $\bar h_i$ normalizes $M$.

In order to prove \eqref{item: D list repeat} we 
let $N(M)$ denote the normalizer of $M$, and define
\begin{equation}\label{eq: definition of phi}
\varphi: N(M) \to \R_{>0}, \ \ \ \ \varphi(g) = \prod_{\s \in S} \left|\det(\Ad(g)|_{\mf{m}_\s})\right|_\s.
\end{equation}
Then, by \cite[Chapter 3, Proposition 55(ii)]{Bourbaki-LieGroupsAlgebrasChap1-3}, $\varphi$ is a homomorphism which records the effect of conjugation on the Haar measure of $M$; i.e., $\Ad(g)_*m_M = \varphi(g) \, m_M $ for all $g \in N(M)$. 
We need to show that $\varphi$ is trivial on $\Gamma_{\bar \mu}$. Clearly, for $\s \in \Sdt \cup \Str,$, $(\bar h_i)_\s$ is independent of $i$ and hence $|\det(\Ad(\bar h_i)|_{\mf{m}_\s})|_\s$ is also  independent of $i$. 
By Proposition \ref{prop:measures on proj spaces}, for $\s\in \Sue$, we have $|\det(\Ad(\bar h_i)|_{\mf{m}_\s})|_\s$ is equal to $ |\det(\Ad(a(\varrho))|_{\mf{m}_\s})|_\s$, since the latter is the modulus of the eigenvalue of $a(\varrho)$ in its action on $\bigwedge^{k_0}\mf{m}_\s$. Hence, $\varphi(\bar h_i)$ is independent of $i$. We denote  $\varphi(\bar h_i) = c_0$. Consider the forward random walk transformation 
$$R: B \times X \to B \times X,  \ \ \ \ R(b,x)\df \left(Tb, \bar h_{b_1}x \right). $$
Then $R$
 preserves the measure $\beta \otimes \nu$. Define 
$$f: B \times X \to \R_{>0}, \ \ \ \ f(b,x) \df \varphi(\bar h_{b_1}),$$
then $f \circ R = c_0 f$ and hence
$$\int f \, \der \beta \otimes \nu = \int f \circ R \, \der \beta \otimes \nu = c_0 \int f \, \der \beta \otimes \nu.$$
This implies $c_0 = 1 $, and proves \eqref{item: D list repeat}. 

For item \eqref{item: C list repeat}, since $\Gamma_{\bar\mu}\subset N_1(M)$ by item \eqref{item: D list repeat}, every $N_1(M)$ orbit is invariant by the random walk.
Moreover, by \cite[Thm. 3.4]{DM_linearization} (see also \cite[\textsection 4d]{Handbook_KSS})\footnote{The cited reference proves this statement for quotients of real Lie groups, but the proof extends readily to our $S$-arithmetic setting.}, the orbit $N_1(M)y$ closed whenever $My$ has finite volume, and hence $N_1(M)x$ is closed for $\nu$-a.e.\ $x$.
By ergodicity of $\nu$, and since the random walk preserves each $N_1(M)x$, we see that $\nu$ lives on a single closed orbit of $N_1(M)$. 

Finally, for \eqref{item: E list repeat} we let $b = (b_\s)_{\s \in S}$ be as in \eqref{eq: def a2} and note that $b_\s$ is trivial for $\s \in \Sdt \cup \Str$, and that for $\s \in \Sue$, the action of $\Ad(b)$ on $\bigwedge^{k_0} \mf{m}_\s$ is the same as the action of $\Ad(\bar h_i^{-1})$ for any $i$. In particular $b \in N(M)$. Let $m_M$ denote the Haar measure on $M$ and let $c >0$ be such that $b_* m_M = c\,m_M$. Then, we have  $c = \varphi(b)$, where $\varphi$ is as in \eqref{eq: definition of phi}. By Propositions \ref{prop:measures on proj spaces} and \ref{prop:measures on proj spaces2}, for $\s \in \Sue$,  the vector $\bigwedge^{k_0} \mf{m}_\s$ is an eigenvector for a maximal eigenvalue $\delta$ of $a(\varrho)$ in some representation $V$ of $\gG_\s$.
Moreover, we have
$$\Ad(\bar h_i)\cdot \bigwedge^{k_0} \mf{m}_\s = a(\varrho) \cdot \bigwedge^{k_0} \mf{m}_\s = \delta \bigwedge^{k_0} \mf{m}_\s.$$
In particular, $c=\prod_{\s\in \Sue}|\delta|_\s^{-1}$.
We wish to show that $c<1$, which verifies item \eqref{item: E list repeat} and completes the proof of Theorem \ref{thm: stationary main}.
To this end, it suffices to show that $|\delta_\s|>1$ for all $\s\in \Sue$. 
Fix some $\s\in\Sue$, and suppose that $|\delta|_\s=1$.
Then, by Lemma \ref{lem: unique maximal eigenvalue}, we obtain that $\gG_\s$ acts trivially on the representation $V$ containing $\bigwedge^{k_0}\mf{m}_\s$.
Hence, $\bigwedge^{k_0} \mf{m}_\s$ is fixed by $\gG_\s$, i.e., $\mf{m}_\s$ is an ideal in $\mf{g}_\s$.
By simplicity of $\mf{g}_\s$, we conclude that this can only hold if $\mf{m}_\s=\mf{g}_\s$, which in turn implies that $\nu=m_{\XS}$, since the only $\gG_\s$-invariant measure is $m_{\XS}$. This contradicts our hypothesis that $\nu\neq m_{\XS}$ and shows that $|\delta|_\s\neq 1$. As $\delta$ is a maximal eigenvalue, we further conclude that $|\delta|_\s >1$. This being true for all $\s\in \Sue$, it follows that $c<1$ as desired.
\end{proof}

\bibliographystyle{amsalpha}
\bibliography{bibfile}

\end{document}